\documentclass{amsart}

\title[Representation of partially ordered sets over regular algebras]
{Representation of partially ordered sets over Von Neumann regular algebras.\\ More prime, non-primitive regular rings}

\date{August 19, 2025}

\usepackage{amsmath,amsthm,amscd,amsxtra,amssymb}

\newtheorem{theo}{Theorem}[section]
\newtheorem{pro}[theo]{Proposition}

\newtheorem{lem}[theo]{Lemma}
\newtheorem{cor}[theo]{Corollary}

\newtheorem{ex}[theo]{Example}

\theoremstyle{remark}
\newtheorem{re}[theo]{Remark}
\newtheorem{res}[theo]{Remarks}

\newtheorem{notations}[theo]{Notations}
\newtheorem{notationsettings}[theo]{Notations and Settings}
\newtheorem{definition}[theo]{Definition}

\numberwithin{equation}{section}

\newcommand{\putl}[5]{\put(#1,#2){\line(#3,#4){#5}}}

\newcommand{\putbb}[3]{\put(#1,#2){\makebox(0,0)[b]{#3}}}

\newcommand{\CA}{\mathcal{A}}

\newcommand{\CE}{\mathcal{E}}
\newcommand{\CF}{\mathcal{F}}
\newcommand{\CG}{\mathcal{G}}
\newcommand{\CH}{\mathcal{H}}

\newcommand{\CM}{\mathcal{M}}
\newcommand{\CN}{\mathcal{N}}

\newcommand{\CP}{\mathcal{P}}

\newcommand{\CR}{\mathcal{R}}
\newcommand{\CS}{\mathcal{S}}

\newcommand{\CV}{\mathcal{V}}

\newcommand{\BC}{\mathbb C}

\newcommand{\BF}{\mathbb F}

\newcommand{\BK}{\mathbb K}
\newcommand{\BL}{\mathbb L}
\newcommand{\BM}{\mathbb M}
\newcommand{\BN}{\mathbb N}

\newcommand{\BR}{\mathbb R}

\newcommand{\BZ}{\mathbb Z}

\newcommand{\ZA}{\mathbf{A}}

\newcommand{\ZC}{\mathbf{C}}
\newcommand{\ZD}{\mathbf{D}}
\newcommand{\ZE}{\mathbf{E}}
\newcommand{\ZF}{\mathbf{F}}
\newcommand{\ZG}{\mathbf{G}}

\newcommand{\ZK}{\mathbf{K}}

\newcommand{\ZM}{\mathbf{M}}

\newcommand{\ZP}{\mathbf{P}}

\newcommand{\ZR}{\mathbf{R}}
\newcommand{\ZS}{\mathbf{S}}

\newcommand{\ZU}{\mathbf{U}}

\newcommand{\za}{\mathbf{a}}
\newcommand{\zb}{\mathbf{b}}

\newcommand{\zd}{\mathbf{d}}
\newcommand{\ze}{\mathbf{e}}

\newcommand{\zg}{\mathbf{g}}

\newcommand{\zi}{\mathbf{i}}

\newcommand{\zl}{\mathbf{l}}
\newcommand{\zm}{\mathbf{m}}
\newcommand{\zn}{\mathbf{n}}
\newcommand{\zo}{\mathbf{o}}
\newcommand{\zp}{\mathbf{p}}

\newcommand{\zr}{\mathbf{r}}
\newcommand{\zs}{\mathbf{s}}

\newcommand{\zu}{\mathbf{u}}
\newcommand{\zv}{\mathbf{v}}

\renewcommand{\a}{\alpha}
\renewcommand{\b}{\beta}
\newcommand{\g}{\gamma}
\renewcommand{\d}{\delta}

\newcommand{\et}{\eta}

\renewcommand{\l}{\lambda}

\newcommand{\f}{\varphi}

\newcommand{\G}{\varGamma}

\renewcommand{\Xi}{\varXi}
\renewcommand{\Pi}{\varPi}
\renewcommand{\Sigma}{\varSigma}

\newcommand{\al}{\boldsymbol{\aleph}}

\newcommand{\nin}{\not\in} %non appartenente a
\newcommand{\sbs}{\subset} %contenuto
\newcommand{\psbs}{\subsetneqq} %propriamente contenuto
\newcommand{\nsbs}{\not\subset} %non contenuto
\newcommand{\sps}{\supset} %contiene
\newcommand{\vu}{\emptyset} %insieme vuoto
\newcommand{\is}{\simeq} %isomorfo
\renewcommand{\le}{\leqslant} %minore o uguale
\renewcommand{\ge}{\geqslant} %maggiore o uguale
\newcommand{\nle}{\nleqslant} %non minore o uguale
\newcommand{\defug}{\,\,\colon\!\!=}

\DeclareMathOperator{\Ker}{Ker} 
\DeclareMathOperator{\Imm}{Im}

\DeclareMathOperator{\Supp}{Supp} 
 \DeclareMathOperator{\Soc}{Soc}

\newcommand{\FP}{\ZF\ZP}
\newcommand{\ESFP}{\ZE\ZS\ZF\ZP}
\newcommand{\CFM}{\BC\BF\BM}

\newcommand{\FR}{\BF\BR}

\newcommand{\Alg}{{\ZA\zl\zg}}

\newcommand{\Simp}{{\ZS\zi\zm\zp}}

\newcommand{\Prim}{{\ZP\zr\zi\zm}}
\newcommand{\Pos}{{\ZP\zo\zs}}
\newcommand{\KO}{\ZK_0}
\newcommand{\Poab}{{\ZP\zo\ZA\zb}}

\newcommand{\map}[3]{#1\colon #2\to#3}

\newcommand{\lmap}[3]{#1\colon#2\longrightarrow#3}

\newcommand{\maximal}[1]{#1^{\bigstar}}

\begin{document}

\author{Giuseppe Baccella}

\address{Universit\`a di L'Aquila, Dipartimento di Ingegneria e scienze dell'informazione e matematica,
              Via Vetoio, 67100 L'Aquila, Italy}

\email{giuseppe.baccella@guest.univaq.it}

\subjclass{Primary 16E50; Secondary 16D60, 16S50, 19A49, 06A06, 06F20}

\keywords{Von Neumann regular ring; partially ordered set; prime ring; primitive ring; dimension group, Grothendieck group; (restricted) Hahn power.}

\begin{abstract}
    Given an \emph{infinite} cardinal $\al$, let $\Pos_{\al}$ be the category whose objects are all partially ordered sets $I$, having a finite cofinal subset and such that both $I$ and the set of all maximal chains of $I$ have at most cardinality $\al$, while the morphisms are all maps which are order imbeddings of the domain as an upper subset of the codomain. Given a field $\BK$, for every object $I$ of $\Pos_{\al}$ we find a unital, locally matricial and hence unit-regular $\BK$-algebra $\mathfrak{B}_{\al,\BK}(I)$ such that the lattice of all its ideals is order isomorphic to the lattice of all lower subsets of $I$. We show that the Grothendieck group $\KO(\mathfrak{B}_{\al,\BK}(I))$, with its natural partial order, is order isomorphic to the restricted Hahn power of $\BZ$ by $I$; this gives a contribution to solve the \emph{Realization Problem for Dimension Groups with order-unit}. We also show that the algebra $\mathfrak{B}_{\al,\BK}(I)$ has the following features: (a) $\mathfrak{B}_{\al,\BK}(I)$ is prime if and only if $I$ is lower directed; (b) $\mathfrak{B}_{\al,\BK}(I)$ is primitive if and only if $I$ has a coinitial chain; (c) $\mathfrak{B}_{\al,\BK}(I)$ is semiartinian if and only if $I$ is artinian, in which case $I$ is order isomorphic to the primitive spectrum of $\mathfrak{B}_{\al,\BK}(I)$. Finally we show that the assignment $I\mapsto\mathfrak{B}_{\al,\BK}(I)$ extends to a pair of functors from $\Pos_{\al}$ to the category of $\BK$-algebras, one covariant and the other contravariant.
\end{abstract}

\maketitle

\setcounter{section}{-1}
\section{Introduction}

We recall that if $L$ is a complete lattice, an element $a\in L$ is said to be \emph{compact} if whenever $a\le \bigvee M$ for some subset $M\sbs L$, then $a\le \bigvee N$ for some finite subset $N\sbs M$. Given a ring $R$ (associative with multiplicative identity), the compact elements of the complete lattice $\BL_{2}(R)$ of ideals of $R$ are the finitely generated ideals, consequently $\BL_{2}(R)$ is an \emph{algebraic} lattice, meaning that all its elements are joins of (possibly infinitely many) compact elements. It is well known that if, in addition, $R$ is a von Neumann regular ring, then $\BL_{2}(R)$ is also distributive.

The problem of determining which algebraic and distributive lattices are isomorphic to $\BL_{2}(R)$ for some von Neumann regular ring $R$ is one among the so called \emph{Realization (or Representation) Problems for Von Neumann regular rings}. A first answer to this problem was given by Bergman in his famous 1986 unpublished note (see \cite{Bergman:001}):

\textbf{Bergman’s Theorem.} \emph{An algebraic distributive lattice $D$ in which the greatest element is compact is isomorphic to $\BL_{2}(R)$, for some unital von Neumann regular ring $R$, if either $D$ has countably many compact elements, or every element of $D$ is a join of \emph{join-irreducible} compact elements}.

In \cite{GoodWeh:1} Goodearl and Wehrung wrote an extended account about the representation problems of distributive algebraic lattices and the connection with the analogous problems for distributive semilattices. The first condition stated in Bergman's Theorem raised the obvious question: is it possible to remove the countability condition on the set of compact elements of $D$? A negative answer was given by Wehrung in \cite{Weh:4}, where he proved that Bergman’s Theorem cannot be extended to those algebraic distributive lattices having $\al_{2}$ or more compact elements. However Wehrung proved in \cite{Weh:2} that Bergman’s Theorem still holds for algebraic distributive lattices having  at most $\al_{1}$ compact elements.

Concerning the second condition, we note that the algebraic distributive lattices satisfying it are precisely those isomorphic to the lattice (partially ordered by inclusion) $\Downarrow\!\!I$ of all lower subsets of some partially ordered set $I$; so, actually, given a partially ordered set $I$ having a finite cofinal subset, Bergman had built a unit-regular algebra $R$ whose lattice of ideals is isomorphic to $\Downarrow\!\!I$.

Other important realization problems arise when one considers the additive abelian monoid $\CV(R)$ of isomorphism classes of finitely generated projective right modules over a ring $R$, with the addition defined as $[A]+[B]\defug [A\oplus B]$, and the corresponding Grothendieck group $\KO(R)$. If $R$ is a regular ring, then $\CV(R)$ enjoys some fundamental and well known properties and the problem of deciding wether, given an abelian monoid $M$ having the same properties, there exists a regular ring $R$ such that $\CV(R)$ is isomorphic to $M$, was the first which was named as \emph{Realization Problem for Von Neumann Regular Rings\/} (see \cite{Ara:011}).

The group $\KO(R)$ carries a natural structure of a pre-ordered and directed abelian group with order-unit, which is  partially ordered in case $R$ is unit-regular (see \cite[Ch 15]{Good:3}). This gives rise to the corresponding \emph{Realization Problem}: which partially ordered, directed abelian groups with order unit are isomorphic to $\KO(R)$ for some regular ring $R$\,? (see \cite[Problem 29]{Good:3}).

Partially ordered abelian groups of special interest are the \emph{dimension groups}, namely the directed, unperforated interpolation groups. Originally, the name ``dimension group" was introduced by Elliott in \cite{Elliot:1} to designate a group which is a direct limit of a countable sequence of simplicial groups, i. e. of groups each order-isomorphic to $\BZ^{n}$ for some $n\in\BN$, with the partial order induced by the submonoid $\BN^{n}$ as positive cone. Later on it was proved by Effros, Handelman and Shen (see \cite[Theorem 2.2]{EffrosHandShen:1}) that a partially ordered abelian group is a (arbitrary) direct limit of simplicial groups if and only if it is a directed, unperforated interpolation group. It was well known that a countable dimension group is isomorphic to $\KO(R)$ for an ultramatricial algebra $R$ over a given field. Goodearl and Handelman proved in \cite{GoodHand:1} that every dimension group with order unit having cardinality $\al_{1}$ can be realized as $\KO(R)$ for a locally matricial (and hence unit-regular) algebra $R$ over a given field and, in addition, every infinite tensor product of such dimension groups can be realized similarly. By way of contrast, dimension groups which cannot be realized as $\KO(R)$ for a regular ring $R$ were found by Wehrung in \cite{Weh:5} and \cite{Weh:1}. Actually the two realization problems for Von Neumann regular rings, the first concerning algebraic distributive lattices and the second concerning dimension groups are intimately connected, as illustrated by P. P.~R{\r u}\v{z}i\v{c}ka in \cite{Ruz:1} and \cite{Ruz:2}.

Recently L.Vas accomplished in \cite{Vas:1} the task of extending the concepts of a simplicial group and of a dimension group to partially ordered abelian groups having an additional action of a group $\G$; by considering the $\G$-graded Grothendieck group $\KO^{\G}(R)$ of a $\G$-graded ring $R$ (as introduced in \cite{Hazrat:1}) she studied the corresponding realization problems of the above graded dimension groups as graded Grothendieck groups.

The deep role of countable dimension groups in the theory of topological dyna\-mical systems is widely displayed in the monograph \cite{DurandPerrin:1}.

If $I$ is a partially ordered set and $G(I)$ denotes the free abelian group generated by $I$, the \emph{(restricted) Hahn power} of $\BZ$ by $I$ is the group $G(I)$, equipped with the partial order defined by declaring positive those elements $x = \sum\{x_{i}i\mid x_{i}\in\BZ, i\in I\}$ of $G(I)$ such that $x_{i}\ge 0$ whenever $i$ is a maximal element of the support of $x$ (see Section \ref{sect algebraBKO} for the general definition). In \cite{BacSpin:17} we proved that $G(I)$ is a dimension group and, if $I$ has a finite cofinal subset, the sum of all maximal elements of $I$ is a order unit; thus we proved that if $I$ is an artinian poset with a finite cofinal subset, then $G(I)$ is realizable as $\mathbf{K}_0(R)$, where $R$ is a unit-regular and semiartinian ring whose primitive spectrum is isomorphic to $I$.

Our main goal in the present work is to show that if $I$ is any partially ordered set, with the only condition of having a finite cofinal subset, then there exists a locally matricial unital algebra $R$ such that $\KO(R)$ is isomorphic to $G(I)$. In order to give some more detail, for a given \emph{infinite} cardinal $\al$, let us call \emph{$\al$-poset} every partially ordered set $I$  such that both the cardinalities of $I$ and of the set $\CM(I)$ of all maximal chains of $I$ do not exceed $\al$. Given a field $\BK$, for every $\al$-poset $I$ \emph{having a finite cofinal subset} we build a unit-regular $\BK$-algebra $\mathfrak{B}_{\al,\BK}(I)$ having the following main properties: the lattice of all ideals of $\mathfrak{B}_{\al,\BK}(I)$ is isomorphic to the lattice of all lower subsets of $I$ and the Grothendieck group $\mathbf{K}_0(\mathfrak{B}_{\al,\BK}(I))$ is order isomorphic to $G(I)$. Thus, while the first property simply confirms one part of Bergman's Theorem, with the second property we enlarge the class of dimension groups with order unit which are realizable as $\mathbf{K}_0$ of some unit-regular algebra, by adding the restricted Hahn powers of $\BZ$ by any partially ordered set $I$ having a finite cofinal subset.

The first two sections have a merely technical nature. We start with an $\al$-poset $I$ and we consider the set $X(I)\defug \al^{(I)}\times\CM(I)$, where $\al^{(I)}$ denotes the set of all maps $p\colon I\to\al$ such that $p(i) = \emptyset$ for all but a finite number of $i\in I$, and we establish the algebra $Q(I)\defug \CFM_{\BK}(X(I))$ of all $X(I)\times X(I)$ column-finite matrices over $\BK$ as the scenario for the construction of the algebra $\mathfrak{B}_{\al,\BK}(I)$. Note that $Q(I)\is \CFM_{\BK}(\al)$, as $|X(I)|=\al$. By using suitable partitions of $X(I)$, for every $i\in I$ we make two subsequent selections of specially shaped matrices in $Q(I)$. With the first one we obtain a $\BK$-subalgebra $Q(I,i)$ of $Q(I)$, isomorphic to $\CFM_{\BK}(\al)$, in such a way that if $i,j\in I$, then $Q(I,i)\sbs Q(I,j)$ if and only if $j\le i$. With the second selection we find for each $i\in I$ a $\BK$-subalgebra $S(I,i)$ of $Q(I,i)$, still isomorphic to $\CFM_{\BK}(\al)$, in such a way that if $i,j\in I$, then $S(I,i)S(I,j) = 0$ if and only if $i,j$ are not comparable, while $S(I,i)S(I,j)\cup S(I,j)S(I,i)\sbs S(I,i)$ if $i\le j$.

In the third section we define the family $(H(I,i))_{i\in I}$ of ``building blocks" which will enable us to define the algebra $\mathfrak{B}_{\al,\BK}(I)$. For every $i\in I$ we take $H(I,i) = \Soc(S(I,i))$ if $i$ is not a maximal element of $I$, so that $H(I,i)$ is isomorphic to the algebra of all finite rank linear transformations of an $\al$-dimensional vector space over $\BK$, hence it is a simple, semisimple non-artinian ring; if $i$ is a maximal element of $I$ we select $H(I,i)$ as a suitable subalgebra isomorphic to $\BK$. We prove that $(H(I,i))_{i\in I}$ is an independent family of $\BK$-subspaces of $Q(I)$ with the following properties: given $i,j\in I$, if $i,j$ are not comparable then $H(I,i)H(I,j) = \{\mathbf{0}\}$ otherwise, if $i\le j$ and $\mathbf{0}\ne\za\in H(I,j)$, then both $\za H(I,i)$ and $H(I,i)\za$ are not zero and are contained in $H(I,i)$. As a consequence, for every subset $J\sbs I$ the subspace $H(I,J) = \bigoplus_{j\in J}H(I,j)$ is a $\BK$-subalgebra of $Q(I)$. We explore the conditions under which $H(I,J)$ has a multiplicative identity; in particular this turns out to be the case if $J$ is an upper subset of $I$ which is bounded from above by a finite number of maximal elements of $I$.

Starting with the fourth section we assume that the $\al$-poset $I$ \emph{has a finite cofinal subset} and, finally, we define $\mathfrak{B}_{\al,\BK}(I) \defug H(I,I)$, which is a unital $\BK$-algebra by the above. We provide an explicit description of the ideals of $\mathfrak{B}_{\al,\BK}(I)$; more generally we prove that, given any subset $J$ of $I$, the assignment $L\mapsto H(I,L)$ defines an isomorphism from the lattice of all lower subsets of $J$ to the lattice of ideals of $H(I,J)$. We also prove that $H(I,J)$ is locally matricial; as a consequence, if $J$ is an upper subset of $I$ then $H(I,J)$ is unit-regular and so, in particular, $\mathfrak{B}_{\al,\BK}(I)$ is unit-regular. As previously announced, this represent a confirmation of the part of Bergman's Theorem which states that every algebraic distributive lattice $D$, in which the greatest element is compact and every element of $D$ is a join of \emph{join-irreducible} compact elements, is isomorphic to the lattice of ideals of some unital von Neumann regular ring $R$.

The fifth section is entirely devoted to the computation of the Grothendieck group of $\mathfrak{B}_{\al,\BK}(I)$ and, as we already mentioned previously, the main result is that $\mathbf{K}_0(\mathfrak{B}_{\al,\BK}(I))$ is isomorphic to $G(I)$, the restricted Hahn power of $\BZ$ by $I$.

In the sixth section we investigate when the algebra  $\mathfrak{B}_{\al,\BK}(I)$ is prime and when it is primitive.
In his great two volumes book ``Ring Theory" Rowen wrote: ``The abstract theory of noncommutative regular rings has been a disappointment" (see \cite[p.277]{Row:1}) because some known basic facts about commutative regular rings, which were hoped to be true in general, turned out to fail. For instance, the question whether every prime regular ring is primitive has an obvious positive answer in the commutative case. By 1970 Kaplansky conjectured an affirmative answer in general (see \cite[p.2]{Kapl:1}) and, again, Rowen commented in \cite[p.308]{Row:2} that ``...until recently the major question in the theory of (von Neumann) regular rings was whether a prime regular ring need be primitive", a comment already mentioned in \cite{AbramsBellRangaswamy:011}. A first counterexample to Kaplansky's conjecture was found in 1977 by Domanov \cite{Domanov:1}, but only thirty seven years later a whole class of counterexamples was found by Abrams, Bell and Rangaswamy; in \cite{AbramsBellRangaswamy:011} they found necessary and sufficient conditions for an oriented graph $E$ in order that the Leavitt algebra $L(E)$ (over any field) is prime and those for which $L(E)$ is primitive. These conditions involve the preorder induced in the set $E_{0}$ of vertices of $E$ by the paths; namely, if $v, w\in E_{0}$, then declare that $v\le w$ if and only if there is a path from $v$ to $w$ (including vertices as zero length paths); thus they proved that $L(E)$ is prime if and only if $E_{0}$ is downward directed, while $L(E)$ is primitive if and only if: (i) $E_{0}$ is downward directed, (ii) every cycle in $E$ has an exit and (iii) $E_{0}$ has a countable coinitial subset $S$. Now $L(E)$ is regular if and only if $E$ is acyclic (see \cite[Theorem 1]{AbramsRangaswamy:012}), in which case the above preorder $\le$ is a partial order; as a consequence, the Abrams, Bell and Rangaswamy result gives rise to whole classes of prime and non-primitive regular rings.

With our result we get new classes of regular algebras which are prime but not primitive. In fact we prove that $\mathfrak{B}_{\al,\BK}(I)$ is prime if and only if the poset $I$ is downward directed, while it is right primitive if and only if $I$ has a coinitial chain, in which case it is also left primitive; if $\al>\al_{0}$ there are plenty of downward directed $\al$-posets without a coinitial chain. Curiously enough, the set of conditions on an acyclic graph $E$ for primitivity of the Leavitt algebra $L(E)$ is very close to the conditions on a poset $I$ in order that $\mathfrak{B}_{\al,\BK}(I)$ is primitive; in fact, it is not difficult to show that if a poset is downward directed and has a countable coinitial subset, then it has a countable coinitial chain, therefore the Abrams, Bell and Rangaswamy result may be restated as follows: if $E$ is an acyclic graph, then the regular algebra $L(E)$ is primitive if and only if the set of vertices $E_{0}$, partially ordered by paths, has a countable coinitial chain.

By combining our result with the characterization of the ideals of $\mathfrak{B}_{\al,\BK}(I)$, we obtain that an ideal $\mathfrak{P}$ of $\mathfrak{B}_{\al,\BK}(I)$ is primitive if and only if there is a (not necessarily unique) chain $A$ of $I$ such that $\mathfrak{P} = H(I,\{A\nle\})$, where $\{A\nle\}$ is the complement set in $I$ of the upper subset $\{A\le\}$ generated by $A$; of course there is a $\BK$-algebra isomorphism $\mathfrak{B}_{\al,\BK}(I)/\mathfrak{P} \is H(I,\{A\le\})$. In particular, for every $i\in I$ we have the primitive ideal $H(I,\{\{i\}\nle\})$ of $\mathfrak{B}_{\al,\BK}(I)$; by denoting with $\Prim_{\mathfrak{B}_{\al,\BK}(I)}$ the primitive spectrum of $\mathfrak{B}_{\al,\BK}(I)$ (partially ordered by inclusion), the assignment $i\mapsto H(I,\{\{i\}\nle\})$ defines an order embedding $\Psi\colon I\to\Prim_{\mathfrak{B}_{\al,\BK}(I)}$. We prove that $\Psi$ is an order isomorphism if and only if $I$ is artinian, if and only if the algebra $\mathfrak{B}_{\al,\BK}(I)$ is semiartinian. This means that every artinian poset, with a finite cofinal subset, can be realized as the primitive spectrum of a regular ring.

In the seventh and last section we explore the possible functorial extensions of the assignments $I\mapsto G(I)$ and $I\mapsto\mathfrak{B}_{\al,\BK}(I)$. We find that if we consider the category $\Pos$ whose objects are all posets, while the morphisms are all maps which are order imbeddings of the domain as an upper subset of the codomain, then the assignment $I\mapsto G(I)$ extends in a ``natural" way to a pair of functors $G(-)$, $G^{*}(-)$, the first covariant and the second contravariant, from $\Pos$ to the category whose objects are all partially ordered abelian groups and isotone group homomorphisms. On the other hand, given an infinite cardinal $\al$ and a field $\BK$, the assignment $I\mapsto \mathfrak{B}_{\al,\BK}(I)$ extends to a pair of functors $\mathfrak{B}_{\al,\BK}(-)$, $\mathfrak{B}_{\al,\BK}^{*}(-)$, the first covariant and the second contravariant, from the subcategory $\Pos_{\al}$ of $\Pos$ of all $\al$-posets \emph{having a finite cofinal subset} to the category $\Alg_{\BK}$ of \emph{unital} $\BK$-algebras and \emph{not necessarily unital} $\BK$-algebra homomorphisms. It turns out that these four functors match with the functor $\KO(-)$, in the sense that there are natural equivalences $\rho(-)\colon G(-)\approx \KO(-)\circ \mathfrak{B}(-)$ and $\rho^{*}(-)\colon G^{*}(-)\approx \KO(-)\circ \mathfrak{B}^{*}(-)$.

Concerning notations and terminology, we follow the (by now) universally established practice in ring and module theory. However, we use the symbol $A\sbs B$ with the meaning ``$A$ is a subset of $B$'' and write $A\psbs B$ to indicate that ``$A$ is a proper subset of $B$''. Finally, we always consider the number $0$ as a member of the set $\BN$ of natural numbers.

\section{Representing a partially ordered set $I$: the main frame.}
\label{sect represposetQi}

Let us start by recalling some basic concepts and introduce notations concerning partially ordered sets that we shall use extensively. If $I$ is a partially ordered set, we denote by $\CM(I)$ the set of all maximal chains of $I$; the \emph{Hausdorff maximal principle} (which is equivalent to the Choice Axiom) states that every chain of $I$ is contained in some maximal one. If not otherwise stated, we shall consider every subset $J$ of $I$ as a partially ordered set endowed with the partial order induced by that of $I$. We denote by $\min(I)$ and $\max(I)$ respectively the smallest and the greatest element of $I$ (when they exist); also, we denote by $\ZM(I)$ the set of all maximal elements of $I$.

For every subset $J\sbs I$ we define
\begin{gather*}
\{\le J\} \defug \{k\in I\mid k\le j\text{ for some }j\in J\}, \\
 \{J\le\} \defug \{k\in I\mid j\le k\text{ for some }j\in J\}, \\
\{\nle J\} \defug I\setminus\{\le J\} \\
 \{J\nle\} \defug I\setminus\{J\le\};
\end{gather*}
if $J = \{i\}$ for a single element $i\in I$, we shall simplify notations by writing $\{i\le\}$ or $\{i\nle\}$, for instance, in place of $\{\{i\}\le\}$ or $\{\{i\}\nle\}$ respectively. A {\em lower subset\/} (resp.
{\em upper subset\/}) of $I$ is a subset $J\sbs I$ such that
if $j\in J$, then $\{\le j\}\sbs J$ (resp. $\{j\le\}\sbs J$).
Clearly $\{\le J\}$ and $\{J\le\}$ are respectively the smallest lower subset and the smallest upper subset of $I$ which contain $J$. While the above notations are largely standard, in order to simplify notations when nested parentheses should occur and provided that it does not create any conflict in the specific context, we will consistently omit almost everywhere curled braces; thus we will write
\[
\le J\,,\quad J\le\,,\quad \le i\,,\quad i\le\,,\quad \nle J\,,\quad J\nle\,,\quad
 \nle i\,,\quad i\nle\,
\]
instead of
\[
\{\le J\}\,,\quad \{J\le\}\,,\quad \{\le i\}\,,\quad \{i \le\}\,,\quad \{\nle J\}\,,\quad \{J\nle\}\,,\quad
 \{\nle i\}\,,\quad \{i\nle\}
\]
respectively. We denote by $\Uparrow\!\!I$ (resp. $\Downarrow\!\!I$) the set of all upper
subsets (resp. lower subsets) of $I$; both $\Uparrow\!\!I$ and
$\Downarrow\!\!I$ are complete and distributive lattices and the map
$J\mapsto I\setminus J$ is an anti-isomorphism from
$\Uparrow\!\!I$ to $\Downarrow\!\!I$.

Throughout the present paper $\al$ stands for a given \underline{infinite} cardinal. For every set $A$, we denote by $\al^{(A)}$ the set of all maps $p$ from $A$ to $\al$ such that $p(a) = \emptyset$ for all but a finite number of $a\in A$ \footnote{As is the standard practice within ZFC set theory, we consider $\al$ as a special ordinal; so $\al$ is the set of all ordinals strictly less than $\al$ and thus, in particular, $\emptyset\in\al$.}.

\begin{pro}\label{pro cardX}
If $A$ is a set such that $0<|A|\le \al$, then
\[
    \left|\al^{(A)}\right| = \al.
\]
\end{pro}
\begin{proof}
Let $\CF$ be the set of all finite subsets of $A$ and, for every $K\in\CF$, let $X_{K}$ be the set of those maps $\map{p}{A}{\al}$ such that $p(A\setminus K) = \{\emptyset\}$. Then
\[
|X_{K}| = \left|\al^{K}\right| = |\al|
\]
for every $K\in\CF$. Since
\[
\al^{(A)} = \bigcup\{X_{K}\mid K\in\CF\}
\]
and $|\CF| = |A|$ in case $A$ is infinite, we infer that
\[
    \left|\al^{(A)}\right| = |\CF|\cdot\al = \al,
\]
as claimed.
\end{proof}

If $A$ is a set with $|A| = \al$, we say that a partition $\CE$ of $A$ is an $\al$-{\em partition\/} if $|\CE| = \al$ and $|B| = \al$ for all $B\in\CE$. Given two partitions $\CE$, $\CF$ of $A$, we say that $\CF$ is $\al$-{\em coarser\/} than $\CE$ if each element of $\CF$ is the union of $\al$ elements of $\CE$.

We say that a partially ordered set $I$ is an \emph{$\al$-poset} in case both $I$ and $\CM(I)$ have cardinality at most $\al$. As we are going to see, if $I$ is a non-empty $\al$-poset, then $\al^{(I)}$ admits a family $(\CE_{i})_{i\in I}$ of partitions with the following features: (1) $\CE_{i}$ is an $\al$-partition for every $i \ne \min(I)$, (2) if $i<j$, then $\CE_{j}$ is $\al$-coarser than $\CE_{i}$, (3) if $i\nle j$, then no member of $\CE_{i}$ is contained in any member of $\CE_{j}$. First observe that, given $i\in I$, every $p\in \al^{(I)}$ decomposes uniquely as a disjoint union
\[
p = r\cup u,
\]
where $r\in\al^{(i\nle)}$ and $u\in\al^{(i\le)}$: simply take $r = p\left|_{\{i\nle\}}\right.$ and $u = p\left|_{\{i\le\}}\right.$. For every $u\in\al^{(i\le)}$, let us consider the subset
\begin{equation}\label{eq 0canonpartition}
E_{u}\defug \left\{r\cup u\left|\,\, r\in\al^{(i\nle)}\right.\right\} = \left\{\left.p\in \al^{(I)}\right|\,\, p\left|_{\{i\le\}}\right. = u\right\} \ne\vu
\end{equation}
of $\al^{(I)}$; clearly, if $u,u'\in\al^{(i\le)}$ and $u\ne u'$, then $E_{u}\cap E_{u'} = \vu$ and therefore
\begin{equation}\label{eq canonpartition}
\CE_{i}\defug \left\{E_{u}\left|\,\, u\in\al^{(i\le)}\right.\right\}
\end{equation}
is a partition of $\al^{(I)}$. If $i \ne \min(I)$, then $\CE_{i}$ is an $\al$-partition by Proposition \ref{pro cardX}; if $i = \min(I)$, namely $\{i\nle\} = \vu$, then $\al^{(i\nle)} = \{\vu\}$ and therefore $E_{u} = \{\vu\cup u\} = \{u\}$ for every $u\in\al^{(i\le)} = \al^{(I)}$, so that $\CE_{i} = \left\{\{u\}\left|\,\,u\in\al^{(I)}\right.\right\}$.

\begin{lem}\label{lem partpart}
With the above settings and notations, the following pro\-perties hold:
\begin{enumerate}
  \item If $i,j\in I$ and $i<j$, then $\CE_{j}$ is $\al$-coarser than $\CE_{i}$; specifically, for every $v\in\al^{(j\le)}$ we have that
\begin{equation}\label{eq partpart1}
  E_{v} =
\bigcup\left\{E_{u}\left|\,\, u\in\al^{(i\le)} \text{ and } u\left|_{\{j\le\}}\right. = v\right.\right\}.
\end{equation}
  \item Given $i\in I$ and a finite subset $J\sbs I$, assume that there are
\[
B\in\CE_{i} \quad\text{and}\quad C_{1},\ldots,C_{n}\in \bigcup_{j\in J}\CE_{j}
\]
such that $B\sbs C_{1}\cup\cdots\cup C_{n}$. Then $i\in \{\le J\}$. Consequently, if $j\in I$ and $i\nle j$, then no member of $\CE_{i}$ is contained in some member of $\CE_{j}$.
\end{enumerate}
\end{lem}
\begin{proof}
(1) Assume that $i<j$ and observe that
\begin{equation}\label{eqpartpat2}
\{\,\{i\nle\},\,\, \{i\le\}\setminus\{j\le\},\,\{j\le\}\,\}
\end{equation}
is a partition of $I$. Accordingly, every element $p\in \al^{(I)}$ decomposes uniquely as a disjoint union
\[
p = r\cup s\cup v,
\]
where $r\in\al^{(i\nle)}$, $s\in\al^{(\{i\le\}\setminus\{j\le\})}$ and $v\in\al^{(j\le)}$. Consequently, for every $v\in\al^{(j\le)}$ we have that
\[
E_{v} = \left\{r\cup s\cup v\left|\,\, r\in\al^{(i\nle)} \text{ and } s\in\al^{(\{i\le\}\setminus\{j\le\})}\right.\right\}
\]
and hence \eqref{eq partpart1} follows.

(2) Under the assumption of (2) suppose that, on the contrary, $i\nin \{\le J\}$. There are $u\in\al^{(i\le)}$, $j_{1},\ldots,j_{n}\in J$ and  $u_{\a}\in\al^{(j_{\a}\le)}$, for $\a\in\{1,\ldots,n\}$, such that
\[
B = E_{u}\quad\text{and}\quad  C_{\a} = E_{u_{\a}}.
\]
According to the assumption, if $p\in\al^{(I)}$ and $\displaystyle p\left|_{\{i\le\}}\right. = u$, then there is some $\a\in\{1,\ldots,n\}$ such that $\displaystyle p\left|_{\{j_{\a}\le\}}\right. = u_{\a}$. Set
\[
D = u_{1}(\{j_{1}\le\})\cup\cdots\cup u_{n}(\{j_{n}\le\}).
\]
By the finiteness of $D$, we can choose some $b\in\al\setminus D$ and, since $j_{\a}\nin\{i\le\}$ for all $\a\in\{1,\ldots,n\}$, we can define a map $p\in\al^{(I)}$ by setting
\[
p\left|_{\{i\le\}}\right. = u,\quad p(j_{1}) =\cdots = p(j_{n}) = b \quad \text{and}\quad p(I\setminus(\{i\le\}\cup\{j_{1},\ldots,j_{n}\})) = \{\emptyset\}.
\]
Then we have that $p\in E_{u}$. However, by the choice of $b$ we have that  $p(j_{\a}) = b \ne u_{\a}(j_{\a})$ for every $\a\in\{1,\ldots,n\}$. This means that $p\nin C_{1}\cup\cdots\cup C_{n}$ and we have a contradiction.
\end{proof}

Starting from this point we let $\BK$ be a given field. Given an $\al$-poset $I$, we set
\[
X(I) \defug \al^{(I)}\times \CM(I);
\]
the main frame within which we will find the regular $\BK$-algebra $\mathfrak{B}_{\al,\BK}(I)$ associated to $I$ is the $\BK$-algebra
\[
Q(I) \defug \CFM_{X(I)}(\BK)
\]
of all finite-column $X(I)\times X(I)$-matrices with entries in $\BK$;
note that if $I \ne \emptyset$, then $|X(I)| = \al$ by Proposition \ref{pro cardX}, therefore $Q(I) \is \CFM_{\al}(\BK)$. If $I = \vu$, then $X(I) = \{(\vu,\vu)\}$ and $Q(I)$ is identified with $\BK$.

\begin{notations}
\phantom{With the   setting, we adopt the following notations:}
\begin{itemize}
    \item We denote respectively by $\mathbf{0}$ and $\mathbf{1}$ the zero and the identity matrices of $Q(I)$.
     \item   If $\za\in Q(I)$ and $x,y\in X(I)$, we use the
symbol $\za(x,y)$ to denote the entry at the intersection of the
$x$-th row with the $y$-th column of $\za$ (i.e. the $(x,y)$-entry of $\za$), instead of the more
traditional symbol $\za_{xy}$; since we often use more complex
arrays, other than single letters, in order to designate the position of the entries of the matrices we deal with, our choice should guarantee a better readability. If $Y,Z\sbs X(I)$, then $\za(Y,Z)$ denotes the $(Y,Z)$-block of $\za$, that is the
    submatrix $(\za(y,z))_{y\in Y, z\in Z}$ of $\za$.
      \item For every $Y\sbs X(I)$, we denote with $\ze_Y$ the
idempotent diagonal matrix such that $\ze_Y(x,x)$ is $1$ if $x\in Y$
and is $0$ otherwise. If $x,y\in X(I)$, we write $\ze_x$ instead of $\ze_{\{x\}}$, while $\ze_{x,y}$ stands for the usual $(x,y)$-matrix unit, namely the matrix whose $(x,y)$-entry is $1$ and all others are zero; so, in particular $\ze_{x} = \ze_{x,x}$.
    \item $\BF\BR_{X(I)}(\BK)$ denotes the subset of $Q(I)$ of all matrices having only finitely many nonzero rows.
\end{itemize}

Of course we shall consider $\BK$ as a
subring of $Q(I)$ by identifying each element of $\BK$ with the
corresponding scalar matrix in $Q(I)$. Without any assumption on the ring $\BK$ (apart from being unital), we have that $\BF\BR_{X(I)}(\BK)$ is an ideal of $Q(I)$ and the following equality holds:
\[
\BF\BR_{X(I)}(\BK) = \bigoplus\{\ze_{x}Q(I) \mid x\in X(I)\};
\]
moreover, for every $x\in X(I)$ the following hold:
\begin{gather*}
\ze_{x}Q(I) = \ze_{x}\BF\BR_{X(I)}(\BK), \\
\BF\BR_{X(I)}(\BK)  = \BF\BR_{X(I)}(\BK)\ze_{x}\BF\BR_{X(I)}(\BK) \label{eq FR2} = Q(I)\ze_{x}Q(I).
\end{gather*}

With our assumption that $\BK$ is a field, $Q(I)$ is regular and right selfinjective; moreover $\BF\BR_{X(I)}(\BK) = \Soc(Q(I))$ is a minimal ideal and $\{\ze_{x} \mid x\in X(I)\}$ is a set of pairwise orthogonal primitive idempotents; it is \emph{complete}, in the sense that it generates $\Soc(Q(I))$ as a right ideal.

If $R$ is any ring and $\map{\f}{Q(I)}R$ is a ring isomorphism, then $\CE = \{\f(\ze_{x})\mid x\in X(I)\}$ generates $\Soc(R)$ as a right ideal; we emphasize that this latter does not depend on the specific isomorphism $\f$. We say that a set $\CS$ of idempotents of $R$ is a \emph{complete set of pairwise orthogonal and primitive idempotents} if $\CS = \CE$ for some isomorphism $\f$ as above.
\end{notations}

The family $(\CE_{i})_{i\in I}$ of partitions of $\al^{(I)}$ previously introduced induces the family $(\CP_{i})_{i\in I}$ of partitions of $X(I)$ defined as follows. Given $i\in I$, for every $(u,A)\in\al^{(i\le)}\times \CM(I)$ we consider the subset
\begin{gather*}X(I)_{(u,A)} \defug E_{u}\times\{A\}
  = \left\{(r\cup u,A)\left|\,\, r\in\al^{(i\nle)}\right.\right\} \\ = \left\{(p,A) \left|\,\, p\in\al^{(I)}\text{ and } p\left|_{\{i\le\}}\right.\right. = u\right\}
\end{gather*}
(see \eqref{eq 0canonpartition}) of $X(I)$ and set
\begin{equation}\label{eq part}
\CP_{i} \defug \left\{X(I)_{(u,A)}\left|\,\, (u,A)\in\al^{(i\le)}\times \CM(I)\right.\right\}.
\end{equation}
Note that if $i= \min(I)$, then $X(I)_{(u,A)} = \{(u,A)\}$ and hence
\[
\CP_{\min(I)} = \left\{\{(u,A)\}\left|\,\,u\in\al^{(I)}, A\in \CM(I)\right.\right\} = \{\{x\} \mid x\in X(I)\}.
\]
As a consequence of Lemma \ref{lem partpart} we have the following:

\begin{cor}\label{cor PartP}
Given $i\in I$, if $i\ne \min(I)$ then $\CP_{i}$ is an $\al$-partition of $X(I)$. If $i,j\in I$, then $\CP_{j}$ is $\al$-coarser than $\CP_{i}$ in case $i<j$, while no member of $\CP_{i}$ is contained in any member of $\CP_{j}$ if $i\nleqslant j$.
\end{cor}

As a first step toward our goals, we determine in $Q(I)$ a family $(Q(I,i))_{i\in I}$ of unital $\BK$-subalgebras, each isomorphic to $Q(I)$, such that $Q(I,i)\sbs Q(I,j)$ exactly when $i\le j$.

For every $i\in I$, let $Q(I,i)$ be the subset of $Q(I)$ consisting of those matrices $\za$ satisfying the following condition:
\begin{equation}\label{magicmatrix}
\begin{split}
&\text{for every $r,s,t\in\al^{(i\nle)}$, $u,v\in\al^{(i\le)}$ and
 $A,B\in \CM(I)$}, \\
&\za((r\cup u,A),(s\cup v,B)) = \d(r,s)\,\za((t\cup u,A),
(t\cup v,B));
\end{split}
\end{equation}
equivalently, for every $(u,A), (v,B)\in\al^{(i\le)}\times \CM(I)$ there is a scalar $d\in \BK$ such that
\[
\za((r\cup u,A),(s\cup v,B)) = \d(r,s)d
\]
for all $r,s\in\al^{(i\nle)}$.

\begin{res}\label{res boh}
Some comments concerning the shape of the matrices in $Q(I,i)$ are in order. Let $\za\in Q(I,i)$ and let $x,y\in X(I)$ be such that $\za(x,y)\ne 0$. Then there are $r,s\in\al^{(i\nle)}$, $u,v\in\al^{(i\le)}$ and $A,B\in \CM(I)$ such that
\[
x = (r\cup u, A), \qquad y = (s\cup v, B)
\]
and condition \eqref{magicmatrix} imposes that $r = s$. We have that $x\in X(I)_{(u,A)}\in \CP_{i}$ and it follows again from \eqref{magicmatrix} that, for every $(t\cup u, A)\in X(I)_{(u,A)}$,
\[
\za((t\cup u,A),(t\cup v,B)) = \za((r\cup u,A),(r\cup v,B)) \ne 0.
\]
This means that the set of those elements $x\in X(I)$ such that the $x$-th row of $\za$ is not zero is always a union of members of the partition $\CP_{i}$ of $X(I)$.

We observe that if $(u,A)\in\al^{(i\le)}\times \CM(I)$, then the assignment $r\mapsto (r\cup u,A)$ defines a bijection from $\al^{(i\nle)}$ to $X(I)_{(u,A)}$; consequently, for every $(u,A),(v,B)\in\al^{(i\le)}\times \CM(I)$, we are allowed to consider the block $\za(X(I)_{(u,A)},X(I)_{(v,B)})$ of $\za$ as a $\al^{(i\nle)}\times\al^{(i\nle)}$-matrix and, as such, the above condition \eqref{magicmatrix} just imposes that this is a scalar matrix.

If $i=\min(I)$, it is clear that $Q(I,i)=Q(I)$.
\end{res}

\begin{theo}\label{theo-posetflr}
With the above settings and notations,
the following properties hold:

\begin{enumerate}
    \item For every $i\in I$, the map
     \[
    \lmap{\f_{i}}{\CFM_{\al^{(i\le)}\times \CM(I)}(\BK)}{Q(I)}
    \]
    defined by
 \begin{equation}\label{a}
\f_{i}(\za)((r\cup u,A),(s\cup v,B)) = \d(r,s)\,\za((u,A),(v,B))
\end{equation}
for
all $\za\in\CFM_{\al^{(i\le)}\times \CM(I)}(\BK)$, $r,s,\in\al^{(i\nle)}$, $u,v\in\al^{(i\le)}$ and $A,B\in \CM(I)$, is a unital and injective $\BK$-algebra homomorphism such that $\Imm(\f_{i}) = Q(I,i)$; consequently
$Q(I,i)$ is a unital $\BK$-subalgebra of $Q(I)$ isomorphic to $\CFM_{\al}(\BK)$.
    \item For every $i\in I$, $u\in\al^{(i\le)}$ and $A\in \CM(I)$
         \begin{equation}\label{abc}
\f_{i}(\ze_{(u,A)}) = \ze_{X(I)_{(u,A)}};
\end{equation}
consequently $\{\ze_{X}\mid X\in\CP_{i}\}$ is a complete set of pairwise orthogonal and primitive idempotents of $Q(I,i)$; in addition, if $Y$ is any union of members of $\CP_{i}$, then $\ze_{Y}\in Q(I,i)$.
    \item For every $i\in I$, a matrix\, $\zb\in Q(I,i)$
belongs to
    \[
    \Soc(Q(I,i)) = \f_{i}\left(\FR_{\al^{(i\le)}\times \CM(I)}(\BK)\right)
     \]
if and only if it the set of those $x\in X(I)$ such that the $x$-th row of $\zb$ is not zero is a union of finitely many members of $\CP_{i}$.
    \item If $i,j\in I$, then $Q(I,j)\sbs Q(I,i)$ if and only if $i\le j$.
    \item The set $\{\Soc(Q(I,i))\mid i\in I\}$ of $\BK$-subspaces of $Q(I)$ is independent.
\end{enumerate}
\end{theo}
\begin{proof}
(1) Let $i\in I$. Given $r,s,t\in\al^{(i\nle)}$, $u,v\in\al^{(i\le)}$ and $A,B\in \CM(I)$, we have that
\begin{align*}
\f_{i}(\za)((r\cup u,A),(s\cup v,B)) &= \d(r,s)\,\za((u,A),(v,B)) \\
                    &= \d(r,s)\,\d(t,t)\,\za((u,A),(v,B)) \\
                  &= \d(r,s)\,\f_{i}(\za)((t\cup u,A),(t\cup v,B)),
\end{align*}
therefore $\Imm(\f_{i})\sbs Q(I,i)$ in view of \eqref{magicmatrix}.
Conversely, given $\zb\in
Q(I,i)$, let $\za\in \CFM_{\al^{(i\le)}\times \CM(I)}(\BK)$ be the
matrix defined as follows: take an arbitrary $r\in\al^{(i\nle)}$ and, for every $u,v\in\al^{(i\le)}$ and $A,B\in \CM(I)$, set
\[
\za((u,A),(v,B)) =
\zb((r\cup u,A),(r\cup v,B));
\]
this definition does not depend on the choice of $r$, thanks to the condition \eqref{magicmatrix}. Then, for every $r,s,\in\al^{(\nle i)}$, $u,v\in\al^{(\le i)}$ and $A,B\in \CM(I)$, we
have
\begin{align*}
\f_{i}(\za)((r\cup u,A),(s\cup v,B)) &= \d(r,s)\,\za((u,A),(v,B)) \\
                  &= \d(r,s)\,\zb((r\cup u,A),(r\cup v,B))\\
                  &= \zb((r\cup u,A),(s\cup v,B)) \qquad\text{again by \eqref{magicmatrix}};
\end{align*}
consequently $\zb = \f_{i}(\za)$ and hence $\Imm(\f_{i}) = Q(I,i)$. It is immediate to check that $\f_{i}(\mathbf{1}) = \mathbf{1}$ and that $\f_{i}$ is $\BK$-linear.
Given $\za, \zb\in\CFM_{\al^{(i\le)}\times \CM(I)}(\BK)$, for all $r,s,\in\al^{(i\nle)}$, $u,v\in\al^{(i\le)}$ and $A,B\in \CM(I)$ we have the following:
\begin{gather*}
\quad\f_{i}(\za\zb)((r\cup u,A),(s\cup v,B)) =
\d(r,s)\,(\za\zb)((u,A),(v,B))\qquad\qquad\qquad \\
\begin{align*}
 &= \sum_{(w,C)\in\al^{(i\le)}\times \CM(I)}
\d(r,s)\,\za((u,A),(w,C))\,\zb((w,C),(v,B)) \\
&= \sum_{\substack{(w,C)\in\al^{(i\le)}\times \CM(I) \\ t\in\al^{(i\nle)}}}
\d(r,t)\,\d(t,s)\,
\za((u,A),(w,C))\,\zb((w,C),(v,B)) \\
&= \sum_{\substack{(w,C)\in\al^{(i\le)}\times \CM(I) \\ t\in\al^{(i\nle)}}}
\d(r,t)\,\za((u,A),(w,C))\,\d(t,s)\,\zb((w,C),(v,B)) \\
&= \sum_{\substack{(w,C)\in\al^{(i\le)}\times \CM(I) \\ t\in\al^{(i\nle)}}}
\f_{i}(\za)((r\cup u,A),(t\cup w,C))\cdot\f_{i}(\zb)((t\cup w,C),(s\cup v,B)) \\
&= \left(\f_{i}(\za)\,\f_{i}(\zb)\right)((r\cup u,A),(s\cup v,B)).
\end{align*}
\end{gather*}
Thus $\f_{i}$ is a $\BK$-algebra homomorphism. Finally, let $\za\in
\CFM_{\al^{(i\le)}\times \CM(I)}(\BK)$ and assume that $\za((u,A),(v,B))\ne 0$ for some $u,v\in\al^{(i\le)}$ and $A,B\in \CM(I)$. Then
\begin{gather*}
(\f_{i}(\za))((r\cup u,A),(r\cup v,B)) \\ = \d(r,r)\,\za((u,A),(v,B)) = \za((u,A),(v,B)) \ne 0
\end{gather*}
for all $r\in\al^{(i\nle)}$; this shows that $\f_{i}$ is injective. As $\left|\al^{(i\le)}\times \CM(I)\right| = \al$, we infer that $Q(I,i)\is
\CFM_{\al}(\BK)$.

(2) Equality \eqref{abc} follows by a direct computation.

(3) It follows from (2) that $\Soc(Q(I,i))$, as a right ideal of $Q(I,i)$, is generated by the set $\{\ze_{X}\mid X\in\CP_{i}\}$ of pairwise orthogonal idempotents. Thus, given $\zb\in Q(I,i)$, we have that $\zb\in \Soc(Q(I,i))$ if and only if there are $X_{1},\ldots,X_{k}\in\CP_{i}$ such that
\[
\zb = (\ze_{X_{1}}+\cdots+\ze_{X_{k}})\zb.
\]
Now the thesis follows from the fact that the set of those $x\in X(I)$ such that the $x$-th row of $\zb$ is not zero is a union of members of the partition $\CP_{i}$ (Remarks \ref{res boh}).

(4) Let $i,j\in I$, suppose that $i<j$ and let $\za\in Q(I,j)$. For every $x\in X(I)$, let us denote by $r_{x}$, $s_{x}$, $u_{x}$ and $A_{x}$ the unique elements of $\al^{(i\nle)}$, $\al^{(\{i\le\}\setminus\{j\le\})}$, $\al^{(j\le)}$ and $\CM(I)$, respectively, such that
\[
x = (r_{x}\cup s_{x}\cup u_{x},A_{x}).
\]
Then, given $x,y\in X(I)$ and any $t\in\al^{(i\nle)}$, by using \eqref{magicmatrix} with $j$ in place of $i$ we have the
following equalities:
\begin{align*}
\za(x,y)&= \d(r_{x}\cup s_{x},r_{y}\cup s_{y})\,\za((t\cup s_{x}\cup u_{x},A_{x}),(t\cup s_{x}\cup u_{y},A_{y})) \\
          &= \d(r_{x},r_{y})\,\d(s_{x},s_{y})\,
          \za((t\cup s_{x}\cup u_{x},A_{x}),(t\cup s_{x}\cup u_{y},A_{y})) \\
          &= \d(r_{x},r_{y})\,\d(t\cup s_{x},t\cup s_{y})\,
          \za((t\cup s_{x}\cup u_{x},A_{x}),(t\cup s_{x}\cup u_{y},A_{y})) \\
          &= \d(r_{x},r_{y})\,
          \za((t\cup s_{x}\cup u_{x},A_{x}),(t\cup s_{y}\cup u_{y},A_{y})).
\end{align*}
Noting that both $s_{x}\cup u_{x}$ and $s_{y}\cup u_{y}$ are in $\al^{(i\le)}$, again from \eqref{magicmatrix} we infer that $\za\in Q(I,i)$.

Conversely, assume that $Q(I,j)\sbs Q(I,i)$ and take any $X\in\CP_{j}$. Then $\mathbf{0}\ne\ze_{X}\in Q(I,j)\sbs Q(I,i)$ and so, by property (2), there is some $Y\in\CP_{i}$ such that $\ze_{Y}\ze_{X}\ne \mathbf{0}$, therefore $X\cap Y\ne\vu$. Let $x\in X\cap Y$. As $x = (p,A)$ for some $p\in\al^{(I)}$ and $A\in\CM(I)$, we infer that $X = X(I)_{(u,A)}$ and $Y = X(I)_{(v,A)}$, where $u = p|_{\{j\le\}}$ and $v = p|_{\{i\le\}}$. As $\ze_{X}(x,x) = 1$ and $x\in Y$, since $\ze_{X}\in Q(I,i)$ we infer that $\ze_{X}(y,y) = 1$ for every $y\in Y$, proving that $Y\sbs X$. Finally it follows from Lemma \ref{lem partpart},(2) that $i\le j$, by taking into account that $X = E_{u}\times\{A\}$ and $Y = E_{v}\times\{A\}$.

(5) For every subset $J$ of $I$ let us consider the family
\[
\CF_{J} \defug \{\Soc(Q(I,j))\mid j\in J\}
\]
and let $J$ be any finite subset of $I$.
Since $\CF_{J}$ is independent if $J = \vu$, given a positive integer $n$, assume that $\CF_{J}$ is independent whenever $|J|<n$ and suppose that $|J| = n$. We can choose a maximal element $m$ of $J$, so that if $J' = J\setminus\{m\}$, then $\CF_{J'}$ is independent by the inductive assumption. We claim that
\begin{equation}\label{eq aaa}
\Soc(Q(I,m))\cap \bigoplus \CF_{J'} =\{\mathbf{0}\},
\end{equation}
which proves the independence of $\CF_{J}$. Assume that $\mathbf{0}\ne \za\in \Soc(Q(I,m))$. Then there are $Y,Z\in\CP_{m}$ such that the the block $\za(Y,Z)$ is a nonzero scalar $\al^{(i\nle)}\times\al^{(i\nle)}$-matrix; in particular the $y$-th row of $\za$ is not zero for all $y\in Y$. On the other hand, if $\mathbf{0}\ne \zb\in \bigoplus \CF_{J'}$, it follows from the already proved property (3) that there are $Y_{1},\ldots,Y_{n}\in \bigcup\{\CP_{j}\mid j\in J'\}$ such that the $y$-th row of $\zb$ is not zero only if $y\in Y_{1}\cup\cdots\cup Y_{n}$. Since $m\nleqslant j$ for every $j\in J'$, it follows from (2) of Lemma \ref{lem partpart} that $Y\nsbs Y_{1}\cup\cdots\cup Y_{n}$. We conclude that $\za\nin \bigoplus \CF_{J'}$ and hence \eqref{eq aaa} follows.
\end{proof}

\section{Representing a partially ordered set $I$: the family $(S(I,i))_{i\in I}$ of subalgebras of $\CFM_{X(I)}(\BK)$.}\label{sect represposetSi}

As a further step toward the construction of the $\BK$-algebra $\mathfrak{B}_{\al,\BK}(I)$, we start from the family $(Q(I,i))_{i\in I}$ of $\BK$-subalgebras of $Q(I) = \CFM_{X(I)}(\BK)$ defined in the previous section and select in each $Q(I,i)$ a $\BK$-subalgebra $S(I,i)$, still isomorphic to $Q(I)$, in such a way that $S(I,i)S(I,j) = \{\mathbf{0}\}$ exactly when $i$ and $j$ are not comparable, while $S(I,i)S(I,j)\cup S(I,j)S(I,i)\sbs S(I,i)$ if $i\le j$. In order to reach this goal we need an additional setup.

\begin{notationsettings}\label{not main}
By keeping the same setting and notations as in the previous section, we add the following data and further notations:
\begin{itemize}
    \item We let $I$ be a given $\al$-poset. For every $i\in I$ we denote by $\CM(\le i\le)$ the set of all maximal chains which pass through $i$:
        \[
        \CM(\le i\le)\defug \{A\in\CM(I)\mid i\in A\}.
        \]
\item Given $i\in I$ and $C\in\CM(\le i)$, we define
\begin{align*}
\CP_{i,C} &\defug \left\{X(I)_{(u,C\cup E)}\left|\,\, u\in\al^{(i\le)},  E\in\CM(i\le)\right.\right\} \sbs \CP_{i},\\
\CN_{C} &\defug \al^{(I)}\times\{C\cup E\mid E\in\CM(i\le)\} = \bigcup\CP_{i,C}\sbs X(I),\\
X(I)_i &\defug \al^{(I)}\times\CM(\le i\le) = \bigcup\{\CN_{C}\mid C\in\CM(\le i)\}\sbs X(I).
\end{align*}
Note that $\CP_{i,C}$ is an $\al$-partition of $\CN_{C}$ (thus $\CN_{C}$ is the union of $\al$ members of $\CP_{i}$), while $\{\CN_{C}\mid C\in\CM(\le i)\}$ is a partition of $X(I)_i$.
\item Given $C,C'\in\CM(\le i)$, we consider the map
\[
\lmap{f_{C'C}}{\CN_{C}}{\CN_{C'}}
\]
defined by
\[
f_{C'C}(p,C\cup E) = (p,C'\cup E)
\]
for every $p\in\al^{(I)}$ and $E\in\CM(i\le)$. It is clear that $f_{C'C}$ is a bijection; moreover, for every $C,C',C''\in\CM(\le i)$ we have
\begin{equation}\label{eq tCA}
f_{C''C'}\,f_{C'C} =
f_{C''C}\,\quad\text{and}\quad f_{CC} = 1_{\CN_{C}.}
\end{equation}
\end{itemize}
\end{notationsettings}

\begin{pro}\label{pro max}
Two elements $i,j\in I$ are comparable if and only if
\[
X(I)_{i}\cap X(I)_{j}\ne\emptyset.
\]
 Consequently, if every maximal chain of $I$ has a supremum, then the set
 \[
 \{X(I)_{m}\mid m\in\ZM(I)\}
 \]
 is a partition of $X(I)$. In particular $X(I)_{\max(X(I))} = X(I)$.
\end{pro}
\begin{proof}
First observe that $X(I)_{i}\cap X(I)_{j}\ne\emptyset$ if and only if ${\CM(\le i\le)}\cap\CM({\le j\le})\ne\emptyset$, if and only if there is some $A\in\CM(I)$ such that $i,j\in A$. By the Hausdorff Maximal Principle this latter condition holds if and only if $i$ and $j$ are comparable.
\end{proof}

\begin{lem}\label{lem-partition2}
With the above notations, if $i,j\in I$, then the following properties hold:
\begin{enumerate}
\item Given $C\in\CM(\le i)$ and $D\in\CM(\le j)$, if $\CN_{C}\cap\CN_{D} \ne \emptyset$, then $i$ and $j$ are comparable and $\CN_{D}\sbs \CN_{C}$ (resp. $\CN_{C}\sbs \CN_{D}$) in case $i\le j$ (resp $j\le i$).
\item Given $ C\in\CM(\le i)$ and $D\in\CM(\le j)$, if $C\sbs D$ then $\CN_{D}\sbs \CN_{C}$.
\item If\, $i$ and $j$ are not comparable, then $\CN_{C}\cap\CN_{D} = \emptyset$ for every $C\in\CM(\le i)$ and $D\in\CM(\le j)$.
\item If $i\le j$, for every $C\in\CM(\le i)$ there exists some $D\in\CM(\le j)$ such that $\CN_{D}\sbs \CN_{C}$.
\item  Assume that $i\le j$, let $C,C'\in\CM(\le i)$ and let $D\in\CM(\le j)$ be such that $\CN_{D}\sbs \CN_{C}$. Then $f_{C'C}(\CN_{D}) = \CN_{D'}$ for some $D'\in\CM(\le j)$ and we have
\begin{equation}\label{eq fAB}
f_{C'C}(x) = f_{D'D}(x)\qquad \text{for every $x\in \CN_{D}$.}
\end{equation}
\end{enumerate}
\end{lem}
\begin{proof}
(1). Assume that $\CN_{C}\cap\CN_{D} \ne \emptyset$. Then there is $p\in\al^{(I)}$, $E\in\{i\le\}$ and $F\in\{j\le\}$ such that $(p,C\cup E) = (p,D\cup F)$ and hence $C\cup E = D\cup F$. This implies that $i,j$ are comparable. If $i\le j$, then
$D = C\cup(D\cap[i,j])$ and this implies that $\CN_{D}\sbs \CN_{C}$. A similar argument applies if $j\le i$.

(2) is clear from the definition and (3) is an obvious consequence of (1).

(4) Suppose that $i\le j$ and let $C\in\CM\{\le i\}$. If $C'$ is any maximal chain of $[i,j]$, by taking $D = C\cup C'$ it is clear from (2) that $\CN_{D}\sbs \CN_{C}$.

(5) Suppose that $i<j$ and let $C, C', D$ be as in the assumptions. Then $D = C\cup E$ for some $E\in\CM([i,j])$. By taking $D' = C'\cup E$ we see that, for every $p\in\al^{(I)}$ and $F\in\CM\{j\le\}$, the following equalities hold:
\[
f_{C'C}(p,D\cup F) = f_{C'C}(p,C\cup E\cup F) = (p,C'\cup E\cup F) = (p,D'\cup F) = f_{D'D}(p,D\cup F).
\]
This proves that $f_{C'C}(\CN_{D}) = \CN_{D'}$ and \eqref{eq fAB}.
\end{proof}

\begin{re}\label{remarklem-partition2}
Let $i<j$ in $I$, let $C\in\CM(\le i)$, $D\in\CM(\le j)$ and assume that $C\sbs D$; then $\CN_{D}\sbs \CN_{C}$ according to Lemma \ref{lem-partition2}, (2). If $Y\in\CP_{j,D}$,
then $Y$ is the union of $\al$ members of $\CP_{i,C}$. In fact, since $\CP_{j}$ is $\al$-coarser than $\CP_{i}$ and $\CP_{j,D}\sbs\CP_{j}$, then $Y$ is the union of $\al$ members of $\CP_{i}$. But if $Z\in \CP_{i}$ and $Z\sbs Y$, then $Z\sbs Y\sbs \CN_{D}\sbs \CN_{C}$ and so $Z\in\CP_{i,C}$.
\end{re}

The necessary setup is now ready and we can proceed to find, for each $i\in I$, a $\BK$-subalgebra $S(I,i)$ of $Q(I,i)$ having the properties announced at the beginning of this section.

Given $i\in I$, by their definitions all $\CN_{C}$ (for $C\in\CM(\le i)$) and $X(I)_i$ are unions of members of $\CP_i$, therefore all the idempotents $\ze_{\CN_{C}}$ and $\ze_{X(I)_i}$ belong to $Q(I,i)$. Let us denote by $S(I,i)$ the
subset of $Q(I,i)$ of those matrices $\za$ satisfying the following two conditions:
\begin{enumerate}
\item[($*_{i}$)] for every $p,q\in\al^{(I)}$ and $A,B\in \CM(I)$,
\[
  \za((p,A),(q,B)) \ne 0 \text{ only if } i\in A\cap B \text{ and } A\cap\{\le i\} = B\cap\{\le i\},
\]
  \item[($**_{i}$)] for every $p,q\in\al^{(I)}$, $C,D\in \CM(\le i)$ and $E,F\in \CM(i\le)$,
\[
\za((p,C\cup E),(q,C\cup F)) = \za((p,D\cup E),(q,D\cup F)).
\]
\end{enumerate}
Note that these conditions may be restated as follows:
\begin{enumerate}
\item[($*_{i}$)] $\za = \ze_{X(I)_i}\za = \za\ze_{X(I)_i}$ and
\begin{equation}\label{mamagicmatrix}
  \ze_{\CN_{C}}\za = \ze_{\CN_{C}}\za\ze_{\CN_{C}} = \za\ze_{\CN_{C}} \quad \text{for every $C\in\CM(\le i)$},
\end{equation}
  \item[($**_{i}$)] for every $C,C'\in\CM(\le i)$ and $x,y\in \CN_{C}$,
\begin{equation}\label{mamamagicmatrix}
\za(x,y) = \za(f_{C'C}(x),f_{C'C}(y));
\end{equation}
\end{enumerate}
note that $\ze_{X(I)_{i}}$ and $\ze_{\CN_{C}}$ are in $S(I,i)$ for all $C\in\CM(\le i)$.

Roughly speaking, $S(I,i)$ consists of those
matrices of $Q(I,i)$ which have zero entries outside the
$(\CN_{C},\CN_{C})$-blocks for $C\in\CM(\le i)$ (which are mutually
disjoint) and, if $C,C'\in\CM(\le i)$, the
$(\CN_{C'},\CN_{C'})$-block coincides with the
$(\CN_{C},\CN_{C})$-block ``\,up to the bijection $f_{C'C}$\,''.

For every $i\in I$ let us consider the set
\[
X(i\le) \defug \al^{(i\le)}\times \CM(i\le).
\]
As we shall see with the next result, $S(I,i)$ turns out to be a unital $\BK$-subalgebra of $\ze_{X(I)_i}Q(I,i)\ze_{X(I)_i}$ isomorphic to $\CFM_{X(i\le)}(\BK)$ and hence to $\CFM_{\al}(\BK)$.

\begin{pro}\label{subrings}
With the above notations and settings, for every ${i\in I}$ let us consider the map
\[
\lmap{\psi_{I,i}}{\CFM_{X(i\le)}(\BK)}{Q(I)}
\]
defined as follows: for all $\za\in\CFM_{X(i\le)}(\BK)$, $r,s\in\al^{(i\nle)}$, $u,v\in\al^{(i\le)}$ and $A,B\in \CM(I)$,
\[
\psi_{I,i}(\za)((r\cup u,A),(s\cup v,B)) =
\begin{cases}\d(r,s)\,\za((u,A\cap\{i\le\}),(v,B\cap\{i\le\})),\\
\text{ if }i\in A\cap B\text{ and }
A\cap\{\le i\} = B\cap\{\le i\}, \\
0,\text{ otherwise. }
\end{cases}
\]
Then $\psi_{I,i}$ is an injective (not necessarily unital) $\BK$-algebra homorphism and
\begin{equation}\label{eq paraponzi}
\psi_{I,i}(\CFM_{X(i\le)}(\BK)) = S(I,i).
\end{equation}
Moreover the following pro\-perties hold:
\begin{enumerate}
\item $S(I,i)S(I,j) = 0$ if and only if $i,j$ are not comparable.
\item
If $i\le j$, then $S(I,i)S(I,j)\cup S(I,j)S(I,i)\sbs S(I,i)$.
\end{enumerate}
\end{pro}
\begin{proof}
Let $\za\in\CFM_{X(i\le)}(\BK)$. Given $r,s,t\in\al^{(i\nle)}$, $u,v\in\al^{(i\le)}$ and $A,B\in \CM(I)$ we have that
\[
\psi_{I,i}(\za)((r\cup u,A),(s\cup v,B)) = \d(r,s)\psi_{I,i}(\za)((t\cup u,A),(t\cup v,B)),
\]
consequently $\psi_{I,i}(\za)\in Q(I,i)$. It is clear that $\psi_{I,i}(\za)$ satisfies ($*_{i}$). Assume that $C,D\in \CM(\le i)$ and $E,F\in \CM(i\le)$. Then
\begin{gather*}
\psi_{I,i}(\za)((r\cup u,C\cup E),(s\cup v,C\cup F)) = \d(r,s)\za((u,E),(v,F)) \\
= \psi_{I,i}(\za)((r\cup u,D\cup E),(s\cup v,D\cup F))
\end{gather*}
therefore $\psi_{I,i}(\za)$ satisfies ($**_{i}$) and then $\psi_{I,i}(\za)\in S(I,i)$. Conversely, suppose that $\zb\in S(I,i)$ and let us define the matrix $\za\in\CFM_{X(i\le)}(\BK)$ as follows: given $u,v\in\al^{(i\le)}$ and $E,F\in \CM(i\le)$, choose any $D\in\CM(\le i)$, $r\in\al^{(i\nle)}$ and set
\begin{equation}\label{eq ponzipa}
    \za((u,E), (v,F) = \zb((r\cup u,D\cup E),
(r\cup v,D\cup F)).
\end{equation}
We note that, since $\zb\in Q(I,i)$, the second member of \eqref{eq ponzipa} does not depend on the choice of $r$ in $\al^{(i\nle)}$; it does not depend on the choice of $D$ in $\CM(\le i)$ either, because $\zb$ satisfies condition ($**_{i}$). A straightforward computation shows that $\psi_{I,i}(\za) = \zb$ and we conclude that \eqref{eq paraponzi} holds.

The map $\psi_{I,i}$ is $\BK$-linear; let $\za\in\CFM_{X(i\le)}(\BK)$ and suppose that $\psi_{I,i}(\za) = \textbf{0}$. Given $u,v\in\al^{(i\le)}$ and $E,F\in\CM({i\le})$, by choosing any $D\in \CM({\le i})$ we have that
\[
\d(r,s)\za((u,E),(v,F)) = \psi_{I,i}(\za)((r\cup u, D\cup E), (s\cup v, D\cup F)) = 0
\]
for every $r,s\in\al^{(i\nle)}$,  hence $\za((u,E),(v,F)) = 0$. We conclude that $\za = \textbf{0}$ and this shows that $\psi_{I,i}$ is injective.

Next, given $\za,\zb\in \CFM_{X(i\le)}(\BK)$, we must prove that
\[
 \psi_{I,i}(\za\zb) = \psi_{I,i}(\za)\psi_{I,i}(\zb),
 \]
that is, given $r,s,\in\al^{(i\nle)}$, $u,v\in\al^{(i\le)}$ and $A,B\in \CM(I)$, by writing $x = (r\cup u,A)$ and $y = (s\cup v,B)$  we must show that the equality
\begin{equation}\label{eq S(I,i)2}
    \psi_{I,i}(\za\zb)(x, y)
    = \sum_{z\in X(I)}\psi_{I,i}(\za)(x, z)\psi_{I,i}(\zb)(z, y)
\end{equation}
holds. If either $i\nin A$, or $i\nin B$, or $i\in A\cap B$ but $A\cap\{\le i\} \ne B\cap\{\le i\}$, then it follows from the definition of $\psi_{I,i}$ that the first member and each summand in the second member of \eqref{eq S(I,i)2} are zero. Otherwise, suppose that $i\in A\cap B$, set $D = A\cap\{\le i\} = B\cap\{\le i\}$ and note that $x,y\in\CN_{D}$. Then we have:
\begin{multline*}
\psi_{I,i}(\za\zb)(x, y) = \d(r,s)(\za\zb)((u,A\cap\{i\le\}), (v,B\cap\{i\le\})\\
\begin{aligned}
&= \d(r,s)\sum_{(w,E)\in X(i\le)}\za((u,A\cap\{i\le\}), (w,E))\zb((w,E), (v,B\cap\{i\le\})) \\
&= \sum_{\substack{t\in\al^{(i\nle)},w\in\al^{(i\le)}\\
C\in\{D\cup F\mid F\in\CM(i\le)\}}}
\bigg(\d(r,t)\d(t,s)\za((u,A\cap\{i\le\}), (w,C\cap\{i\le\})) \\
&\phantom{= \sum_{\substack{t\in\al^{(i\nle)},w\in\al^{(i\le)}\\
C\in\{D\cup F\mid F\in\CM(i\le)\}}}
\bigg(\d(r,t)\d(t,s)}\cdot\zb((w,C\cap\{i\le\}), (v,B\cap\{i\le\}))\bigg)  \\
&= \sum_{\substack{t\in\al^{(i\nle)},w\in\al^{(i\le)}\\
C\in\{D\cup F\mid F\in\CM(i\le)\}}}\psi_{I,i}(\za)(x, (t\cup w,C))\psi_{I,i}(\zb)((t\cup w,C), y)  \\
&= \sum_{z\in \CN_{D}}\psi_{I,i}(\za)(x, z)\psi_{I,i}(\zb)(z, y)  \\
&= \sum_{z\in X(I)}\psi_{I,i}(\za)(x, z)\psi_{I,i}(\zb)(z, y),
\end{aligned}
\end{multline*}
as wanted, and we conclude that $\psi_{I,i}$ is a $\BK$-algebra homomorphism.

(1) Let $i,j\in I$ and assume that $i,j$ are not comparable. Then,
given ${C\in \CM(\le i)}$ and $D\in \CM(\le j)$, we have $\CN_{C}\cap\CN_{D} =
\emptyset$ according to Lemma \ref{lem-partition2}, (3) and therefore $X(I)_i\cap
X(I)_j = \emptyset$. Consequently, if $\za\in S(I,i)$ and $\zb\in
S(I,j)$, then $\za\zb =
\ze_{X(I)_{i}}\za\ze_{X(I)_{i}}\ze_{X(I)_{j}}\zb\ze_{X(I)_{j}} = 0$. If, on the
contrary, $i\le j$ and $A$ is any maximal chain of $I$ such that
$i,j\in A$, then $A\in\CM(\le i\le)\cap \CM(\le j\le)$ and hence
$X(I)_{i}\cap X(I)_{j}\ne\emptyset$. Consequently $\mathbf{0} \ne
\ze_{X(I)_{i}}\,\ze_{X(I)_{j}} = \ze_{X(I)_{j}}\,\ze_{X(I)_{i}}\in S(I,i)S(I,j)\cap
S(I,j)S(I,i)$.

(2) Suppose that $i<j$, let $\za\in S(I,i)$, $\zb\in S(I,j)$ and assume
that $0\ne (\za\zb)(x,y) = \sum_{z\in X(I)}\za(x,z)\zb(z,y)$ for some
$x,y\in X(I)$. Then $\za(x,z) \ne 0 \ne \zb(z,y)$ for some $z\in X(I)$ and
therefore, by conditions ($*_{i}$) and ($*_{j}$), there are some ${C\in\CM(\le i)}$ and $D\in\CM(\le j)$ such that $x,z\in \CN_{C}$ and $z,y\in\CN_{D}$, so that $\CN_{C}\cap\CN_{D}\ne\vu$ and hence $\CN_{D}\sbs\CN_{C}$ in view of
Lemma \ref{lem-partition2}, (1). We infer that the matrix $\za\zb$
has zero entries outside the $(\CN_{C},\CN_{C})$-blocks for
$C\in\CM(\le i)$. Suppose that $C,C'\in\CM(\le i)$ and let us prove that
\begin{equation}\label{eqblockss}
(\za\zb)(x,y) = (\za\zb)(f_{C'C}(x),f_{C'C}(y))
\end{equation}
for all $x,y\in \CN_{C}$. It follows from Lemma \ref{lem-partition2}, (5) that there is no $D\in\CM(\le j)$ such
that $y\in\CN_{D}$ if and only if there is no $D'\in\CM(\le j)$ such
that $f_{C'C}(y)\in\CN_{D'}$; if that is the case, since $\zb\in
S(I,j)$, both members of \eqref{eqblockss} are zero. Otherwise there is
$D\in\CM(\le j)$ such that $y\in\CN_{D}$; again by
Lemma \ref{lem-partition2}, (1) we have that $\CN_{D}\sbs\CN_{C}$ and, by setting $D' = f_{C'C}(B)$, we may compute as follows:
\begin{align*}
(\za\zb)(x,y) &= \sum_{z\in \CN_{C}}\za(x,z)\,\zb(z,y) =
\sum_{z\in \CN_{D}}\za(x,z)\,\zb(z,y) \\
&= \sum_{z\in \CN_{D}} [\za(f_{C'C}(x),f_{C'C}(z))]\,
[\zb(f_{D'D}(z),f_{D'D}(y))] \\
&= \sum_{u\in \CN_{D'}}
[\za(f_{C'C}(x),u)]\,[\zb(u,f_{D'D}(y))] \\
&= \sum_{u\in \CN_{C'}}
[\za(f_{C'C}(x),u)]\,[\zb(u,f_{C'C}(y))] \\
&= (\za\zb)(f_{C'C}(x),f_{C'C}(y)).
\end{align*}
Thus \eqref{eqblockss} holds for all $x,y\in \CN_{C}$, showing that
$\za\zb\in S(I,i)$. The proof that $\zb\za\in S(I,i)$ is similar.
\end{proof}

\begin{re}\label{res Si}
Clearly we have that $\psi_{I,i}(\mathbf{1}) = \ze_{X(I)_i}$ and $S(I,i)$ is a unital subalgebra of $Q(I,i)$ if and only if $X(I)_{i} = X(I)$, namely $I = {\{\le i\}\cup\{i\le\}}$; if it is the case, then $S(I,i) = Q(I,i)$ if and only if $\{\le i\}$ is a chain. As a result, we have that $S(I,i) = Q(I,i)$ for all $i\in I$ if and only if $I$ is a chain.
\end{re}

As $\psi_{I,i}$ restricts to an isomorphism from $\CFM_{X(i\le)}(\BK)$ onto $S(I,i)$, we have in $S(I,i)$ the complete system
\[
\CS_{i} \defug \left\{\psi_{I,i}(\ze_{(u,E)})\left|(u,E)\in X(i\le)\right.\right\}
\]
 of pairwise orthogonal and primitive idempotents, each of which is actually a diagonal matrix of the form $\ze_{Y}$ for a suitable subset $Y\sbs X(I)_i$; to see this and for subsequent purposes we introduce the following additional

\begin{notations}\label{not Y} \phantom{xxxx}
\begin{enumerate}
\item[$\bullet$] For every $u\in\al^{(i\le)}$ and $E\in\CM(i\le)$ we define
\begin{equation}\label{eq YuB}
\begin{aligned}
Y_{(u,E)}\colon &= \bigcup\left\{X(I)_{(u,C\cup E)}\left|\,\,C\in\CM(\le i)\right.\right\} \\
&= \left\{(p,A)\in X(I)\left|\,\,p|_{\{i\le\}} = u, A\cap\{i\le\} = E\right.\right\} \sbs X(I)_i
\end{aligned}
\end{equation}
and set
\begin{equation}\label{eq YuBB}
\CR_{i}\defug \left\{Y_{(u,E)}\left|\,\, u\in\al^{(i\le)}, E\in\CM(i\le)\right.\right\}.
\end{equation}
\item[$\bullet$] For every $C\in\CM(\le i)$ and $V\in\CP_{i,C}$ define
\begin{equation}\label{eq overline}
\overline{V} \defug
\bigcup\{f_{C'C}(V)\mid C'\in\CM(\le i)\}.
\end{equation}
\end{enumerate}
\end{notations}

Now direct computations show that, given $i\in I$, for every $u\in\al^{(i\le)}$, $C\in\CM(\le i)$ and $C\in\CM(i\le)$ we have
\begin{equation}\label{eq eYub}
 \psi_{I,i}(\ze_{(u,E)}) = \ze_{Y_{(u,E)}}
\end{equation}
and
\begin{equation}\label{eq Y}
Y_{(u,E)} =  \overline{X(I)_{(u,C\cup E)}};
\end{equation}
consequently
\[
\CR_{i} =
\{\overline{V}\mid V\in\CP_{i,C}\}
\]
and therefore
\[
\CS_{i} = \left\{\ze_{Y}\left|Y\in \CR_{i}\right.\right\}.
\]

\begin{lem}\label{lem idempprim Hi}
With the above notations, let $i,j\in I$. Then
the following hold:
\begin{enumerate}
  \item If $Y\in\CR_{i}$, $Z\in\CR_{j}$ and $i,j$ are not comparable, then $Y\cap Z = \vu$, while if $i\le j$ and $Y\cap Z \ne \vu$, then $Y\sbs Z$.
  \item $\CR_{i}$ is a partition of $X(I)_{i}$ of cardinality $\al$; it is an $\al$-partition if $i\ne \min(I)$.
  \item If $Z\in\CR_{j}$ and $j$ is not a minimal element of $I$, then $Z$ is the union of members of $\CR_{i}$, for $i$ ranging in $\{<j\}$. In addition, if $i<j$, then $Z$ contains $\al$ elements of $\CR_{i}$; precisely, if $Z = Y_{(u,E)}\in\CR_{j}$, where ${u\in\al^{(j\le)}}$ and $E\in\CM(j\le)$, given $v\in\al^{(i\le)}$ and $F\in\CM(i\le)$, we have that $Y_{(v,F)}\sbs Z$ if and only if $v\left|_{\{j\le\}}= u\right.$ and $F\cap\{j\le\} = E$.
  \item If $i$ is not a maximal element of $I$ and $Y\in\CR_{i}$, then there is some $j>i$ and $Z\in\CR_{j}$ such that $Y\sbs Z$.
  \item If $\za\in S(I,i)$, then the set $X$ of those $x\in X(I)$ such that the $x$-th row of $\za$ is not zero is a union of members of $\CR_{i}$.
\end{enumerate}
 \end{lem}
 \begin{proof}
 (1) Let $Y\in\CR_{i}$, $Z\in\CR_{j}$. If $i,j$ are not comparable then $X(I)_{i}\cap X(I)_{j} = \vu$ by Proposition \ref{pro max}. Thus $Y\cap Z = \vu$, because $Y\sbs X(I)_{i}$ and $Z\sbs X(I)_{j}$. Assume that $i\le j$ and write $Y = Y_{(u,E)}$, $Z = Y_{(v,F)}$ for some $u\in\al^{(i\le)}$, $v\in\al^{(j\le)}$, $E\in\CM(i\le)$ and $F\in\CM(j\le)$. If there is some $(p,A)\in X(I)$ which belongs to $Y\cap Z$, then necessarily $p\left|_{\{i\le\}}= u\right.$, $p\left|_{\{j\le\}}= v\right.$, $A\cap\{i\le\} = E$ and $A\cap\{j\le\} = F$. Inasmuch as $\{j\le\}\sbs \{i\le\}$, then $u\left|_{\{j\le\}}= v\right.$ and $F\sbs E$. From the definition \eqref{eq YuB} we easily infer that $Y\sbs Z$.

  (2) We already know that $\bigcup\CR_{i}\sbs X(I)_{i}$. If $Y,Z\in\CR_{i}$ and $Y\ne Z$, then it follows from (1) that $Y\cap Z = \vu$. Given $x\in X(I)_{i}$, that is $x = (p, C\cup E)$ for some $p\in \al^{(I)}$, $C\in\CM(\le i)$ and $E\in\CM(i\le)$, by setting $u = p|_{\{i\le\}}$ we have that $x\in Y_{(u,E)}$. Thus $\bigcup\CR_{i}= X(I)_{i}$. The remaining part of the statement comes clear from the definitions.

 (3) Assume that $j$ is not minimal and let $Z = Y_{(u,E)}\in\CR_{j}$, where $u\in\al^{(j\le)}$ and $E\in\CM(j\le)$. Given $C\in\CM(\le j)$, we can choose some $i\in C$ with $i<j$ and we obtain from Lemma \ref{lem partpart} that
  \[
  X(I)_{(u,C\cup E)} = \bigcup\left\{X(I)_{(v,C\cup E)}\left|\,\,v\in\al^{(i\le)} \text{ and }v|_{\{j\le\}} = u\right.\right\}.
  \]
  On the other hand, given $v\in\al^{(i\le)}$ with $v|_{\{j\le\}} = u$, we have that $X(I)_{(v,C\cup E)}\sbs Y_{(v,F)}$, where $F = (C\cap\{i\le\})\cup E\in\CM(i\le)$. Thus, by recalling \eqref{eq YuB} we infer that $Y_{(v,F)}\cap Z\ne\vu$ and therefore $Y_{(v,F)}\sbs Z$ by (1). This proves the first statement of (3) and the second follows since the set $\left.\left\{v\in\al^{(i\le)}\right| v|_{\{j\le\}} = u\right\}$ has cardinality $\al$.

  (4) Assume that $i$ is not a maximal element and let $Y_{(u,E)}\in\CR_{i}$, where $u\in\al^{(i\le)}$ and $E\in\CM(i\le)$. By the assumption there is some $j\in E$ such that $i<j$ and, by taking $v = u|_{\{j\le\}}$ and $F = E\cap\{j\le\}$ we see that $Y_{(u,E)}\sbs Y_{(v,F)}$.

  (5) Let $x\in X$. Then, by condition ($*_{i}$), there is some $C\in\CM(\le i)$ such that $x\in\CN_{C}$; in particular $x = (r\cup u, C\cup E)$ for some $r\in\al^{(i\nle)}$, $u\in\al^{(i\le)}$ and $E\in\CM(i\le)$, so that
  \[
  x\in X(I)_{(u,C\cup E)} \sbs Y_{(u,E)} \in \CR_{i}.
  \]
  Next, let us consider any $y\in Y_{(u,E)}$ and let $z\in X(I)$ be such that $\za(x,z) \ne 0$. Then $z\in\CN_{C}$ by condition ($*_{i}$) and, by taking condition \eqref{magicmatrix} into account, we have that $z = (r\cup v, C\cup F)$ for some $v\in\al^{(i\le)}$ and $F\in\CM(i\le)$. Since $y = (s\cup u, D\cup E)$ for some $s\in\al^{(i\nle)}$ and $D\in\CM(\le i)$, by taking conditions \eqref{magicmatrix} and ($**_{i}$) into account we may write:
  \begin{align*}
0 \ne \za(x,z) &= \za((r\cup u, C\cup E),(r\cup v, C\cup F)) \\
&= \za((s\cup u, C\cup E),(s\cup v, C\cup F)) \\
&= \za(f_{CD}(s\cup u, D\cup E),f_{CD}(s\cup v, D\cup F)) \\
&= \za(f_{CD}(y),f_{CD}(s\cup v, D\cup F)) \\
&= \za(y,(s\cup v, C\cup F)).
\end{align*}
Thus $y\in X$ and we conclude that $X$ is the union of those $Y\in \CR_{i}$ for which $X\cap Y \ne \vu$.
\end{proof}

\section{Representing a partially ordered set $I$: the $\BK$-algebra $H(I,J)$ associated to a subset $J$ of $I$.}\label{sect ringDI}

By keeping the same data, settings and notations as in the previous sections, our next objective is to associate to each element $i$ of the $\al$-poset $I$ a (not necessarily unital) $\BK$-subalgebra $H(I,i)$ of $S(I,i)$, in such a way that $\CH = \{H(I,i)\mid i\in I\}$ is an independent set of $\BK$-subspaces of $Q(I)$ with the following features: if $i$ is a maximal element of $I$, then $H(I,i)$ is isomorphic to $\BK$; if $i$ is not maximal, then $H(I,i)$ is isomorphic to $\FR_{X(I)}(\BK)$, therefore it is a simple, semisimple and non-artinian algebra; moreover $H(I,i) H(I,j) = 0$ if $i,j$ are not comparable, while both $H(I,i)H(I,j)$ and $H(I,j) H(I,i)$ are nonzero and are contained in $H(I,i)$ if $i\le j$. Thus, for every subset $J\sbs I$, the subspace
\[
H(I,J)\defug\bigoplus_{j\in J}H(I,j)
\]
turns out to be a $\BK$-subalgebra of $Q(I)$, possibly without an identity.

For every $i\in I$ let us define
the $\BK$-subalgebra $H(I,i)$ of $Q(I)$ as follows:
\[
H(I,i) \defug \begin{cases}
          \psi_{I,i}\left(\FR_{X(i\le)}(\BK)\right) = \Soc(S(I,i)),  \hbox{ if $i$
\underbar{is not} a maximal element of $I$;} \\
          \psi_{I,i}\left(\mathbf{1}_{\CFM_{X(i\le)}(\BK)}\BK\right)= \ze_{X(I)_i}\BK,  \hbox{ if $i$ \underbar{is} a maximal
element of $I$.}
        \end{cases}
\]
Of course $H(I,i)\ne H(I,j)$ if $i\ne j$; also note that $H(I,i)$ is a unital subring of $Q(I)$ if and only if $i$ is the greatest element of $I$, in which case $H(I,i) = \BK$.

 \begin{lem}\label{HsbsFr}
 Assume that $i$ \underbar{is not} a maximal element of $I$. Then
\begin{equation}\label{H_ii}
     H(I,i) = \bigoplus\left\{\ze_{Y}H(I,i) = \ze_{Y}S(I,i)\,\,\left|\,\,
Y\in\CR_{i}\right.\right\}
             \end{equation}
and
     \begin{equation}\label{H_i}
     H(I,i) = H(I,i)\,\ze_{Y}H(I,i)
     \end{equation}
for every $Y\in\CR_{i}$. Moreover, given $\za\in S(I,i)$,
 the following conditions are equivalent:
 \begin{enumerate}
     \item $\za\in H(I,i)$.
     \item $\ze_{\CN_{C}}\za\,\ze_{\CN_{C}}\in \Soc(Q(I,i))$ for all
     $C\in\CM(\le i)$.
     \item There are $Y_{1},\ldots,Y_{n}\in\CR_{i}$ such that the
     $x$-th row of $\za$ is not zero only if $x\in
     Y_{1}\cup\cdots\cup Y_{n}$; equivalently
     \[
     \za = \left(\ze_{Y_{1}}+\cdots+\ze_{Y_{n}}\right)\za.
     \]
     \end{enumerate}
 \end{lem}
 \begin{proof}
 The first statement follows from the fact that $\{\ze_{Y}\mid Y\in\CR_{i}\}$ is a complete set of pairwise orthogonal and primitive idempotents of $S(I,i)$ (see just after Notations \ref{not Y}), while the equivalence (1)$\Leftrightarrow$(3) is clear from \eqref{H_ii}.

 (3)$\Rightarrow$(2). Assume (3) and let $C\in\CM(\le i)$. There are $V_{1},\ldots,V_{n}\in\CP_{i,C}$ such that $Y_{r} = \overline{V_{r}}$ for all $r\in\{1,\ldots,n\}$ and we infer from the definition \eqref{eq overline} that
\[
(Y_{1}\cup\cdots\cup Y_{n})\cap \CN_{C} = V_{1}\cup\cdots\cup V_{n}.
\]
As a result we have that
\[
\ze_{\CN_{C}}(\ze_{Y_{1}}+\cdots+\ze_{Y_{n}}) = \ze_{V_{1}}+\cdots+\ze_{V_{n}}
\]
and (3) of Theorem \ref{theo-posetflr} implies that $\ze_{\CN_{C}}\za\,\ze_{\CN_{C}}\in
\Soc(Q(I,i))$, because $\CP_{i,C}\sbs\CP_{i}$, $\za\in Q(I,i)$ and $\Soc(Q(I,i))$ is an ideal of $Q(I,i)$.

(2)$\Rightarrow$(3) Suppose (2) and take any $C\in\CM(\le i)$. Then it follows from (3) of Theorem \ref{theo-posetflr} that there are $Z_1,\ldots,Z_n\in\CP_{i}$ such that, for every $x\in X(I)$, the $x$-th row of $\ze_{\CN_{C}}\za\,\ze_{\CN_{C}}$ is not zero only if $x\in Z_1\cup\cdots\cup Z_n$. As $\ze_{\CN_{C}}\za\,\ze_{\CN_{C}}$ has zero entries outside the $\ze_{\CN_{C}}\times\ze_{\CN_{C}}$-block, necessarily $Z_1,\ldots,Z_n$ are contained in $\CN_{C}$ and so they belong to $\CP_{i,C}$. Given $x,y\in X(I)$, assume that
$\za(x,y)\ne 0$. Then $x,y\in\CN_{C'}$ for some $C'\in\CM(\le i)$ and $\za(f_{CC'}(x),f_{CC'}(y)) = \za(x,y)\ne 0$, therefore $f_{CC'}(x)\in Z_1\cup\cdots\cup Z_n$ and hence $x\in f_{C'C}(Z_1)\cup\cdots\cup f_{C'C}(Z_n)\sbs \overline{Z_1}\cup\cdots\cup \overline{Z_n}$.
\end{proof}

\begin{re}
    In general we have that $\ze_{\CN_{C}}H(I,i)\,\ze_{\CN_{C}}\not\sbs H(I,i)$,
unless $\{\le i\}$ is a chain. If $i$ \underbar{is not} a maximal element
of $I$, then $H(I,i)\sbs \Soc(Q(I,i))$ \emph{if and only if $\{\le i\}$ has finitely many maximal
chains}.
\end{re}

\begin{re}\label{re ponzipera}
    We feel it worthwhile to emphasize that the main features of the algebras $H(I,i)$ we are going to investigate bear on the particular shape of a matrix $\za\ne \mathbf{0}$ which belongs to $H(I,i)$. Firstly remember that, since $\za\in S(I,i)$, by Lemma \ref{lem idempprim Hi}, (5) the set $X$ of those $x\in X(I)$ such that the $x$-th row of $\za$ is not zero is a union of members of the partition $\CR_{i}$ of $X(I)_{i}$; in turn, these members are a union of members of $\CP_{i,C}\sbs\CP_{i}$ (see \ref{eq YuB}). If $i$ is not a maximal element of $I$ and $x\in X$, by condition ($*_{i}$) we have that $x\in\CN_{C}$ for a necessarily unique $C\in\CM(\le i)$; in addition, given $C\in\CM(\le i)$, by (3) of Theorem \ref{theo-posetflr} and Lemma \ref{HsbsFr} we have that $X\cap\CN_{C}$ is the union of finitely many members of $\CP_{i,C}$ and condition ($**_{i}$) entails that
    \[
    X\cap\CN_{C'} = f_{C'C}(X\cap\CN_{C})\quad \text{for every $C, C'\in\CM(\le i)$}.
    \]
\end{re}

 \begin{lem}\label{nonzerochain}
 Let $j_1<\cdots<j_n$ be a finite chain of $I$ with $n>1$, assume that $\mathbf{0}\ne\za_1\in
\nolinebreak H(I,j_1)$,\ldots,$\mathbf{0}\ne\za_n\in H(I,j_n)$ and let
$C_n\in\CM(\le j_n)$ \footnote{\, \emph{\textbf{WARNING}}! In a context like the present one, in which we have a poset $I$, a subset $J\sbs I$ and an element $j\in J$, by the symbol $\{\le j\}$ we mean the set $\{k\in I\mid k\le j\}$ and \underline{not} $\{k\in J\mid k\le j\}$!} be such that $j_1,\ldots,j_{n-1}\in C_n$. Then the
$(\CN_{C_n}\times \CN_{C_n})$-block
 of $\za = \za_1+\cdots+\za_n$ is not zero; specifically,
 there is some $x\in \CN_{C_n}$ such that the $x$-th row of $\za$
 is not zero and coincides with the $x$-th row of $\za_n$.
 \end{lem}
\begin{proof}
For every $r\in\{1,\ldots,n\}$ set $C_r = C_n\cap\{\le j_r\}$, let $Y_r$ be the subset of those $x\in\CN_{C_r}$ such that the $x$-th row of $\za_r$ is not zero and note that $Y_r\ne\vu$. If $r<n$, the element $j_r$ is not maximal in $I$, therefore $Y_r$ is a (disjoint) union of finitely many
elements of $\CP_{j_{r},C_r}\sbs\CP_{j_r}$ (see Remark \ref{re ponzipera}). Since $\CP_{j_s}$ is
$\al$-coarser than $\CP_{j_r}$ when $r<s\le n$, in particular
$Y_n$ is a (disjoint) union of $\al$ elements of
$\CP_{j_{n-1}}$; on the other hand, by the above
$Y_1\cup\cdots\cup Y_{n-1}$ is contained in the (disjoint) union of
finitely many elements of $\CP_{j_{n-1}}$. Consequently
\[
Y_n\setminus(Y_1\cup\cdots\cup Y_{n-1})\ne\emptyset
\]
and therefore, if $x\in Y_n\setminus(Y_1\cup\cdots\cup Y_{n-1})$,
the $x$-th row of $\za$ coincides with the $x$-th row of $\za_n$,
which is not zero.
\end{proof}

 \begin{lem}\label{nonzerochain2}
 Let $J$ be a finite subset of $I$, let $\za = \sum_{j\in
J}\za_j$, where each $\za_j$ is a non-zero element of $H(I,j)$ for $j\in J$, let $j_1<\cdots<j_n$ be
a maximal chain of $J$ and let
$C_n\in\CM(\le j_n)$ be such that $j_1,\ldots,j_{n-1}\in C_n$. Then the $(\CN_{C_n}\times \CN_{C_n})$-blocks of
$\za$ and $\za_{j_1}+\cdots+\za_{j_n}$ coincide.
 \end{lem}

\begin{proof}
As in the proof of Lemma \ref{nonzerochain}, for every $r\in\{1,\ldots,n\}$ set $C_r = C_n\cap\{\le j_r\}$ and note that since $C_1\sbs\cdots\sbs C_n$, then $\CN_{C_n}\sbs\cdots\sbs\CN_{C_1}$ by (2) of Lemma \ref{lem-partition2}. Consequently, given $j\in J$, if there is some $D\in\CM(\le j)$ such that $\CN_{C_n}\cap \CN_{D}\ne\emptyset$, then $\CN_{C_r}\cap \CN_{D}\ne\emptyset$ for all
$r\in\{1,\ldots,n\}$ and therefore it follows from Lemma
\ref{lem-partition2}(3) that $j$ is comparable with every $j_r$. As a
result $j\in\{j_1,\ldots,j_n\}$, because this latter is a maximal
chain of $J$. Since every matrix in $H(I,j)$ has zero entries outside the $(\CN_{D}\times\CN_{D})$-blocks, for $D\in\CM(\le j)$, we infer from the above that if $j\in
J\setminus\{j_1,\ldots,j_n\}$, then the $(\CN_{C_n}\times \CN_{C_n})$-block of every matrix in $H(I,j)$ is zero and,
consequently, the $(\CN_{C_n}\times \CN_{C_n})$-blocks of $\za$ and
$\za_{j_1}+\cdots+\za_{j_n}$ coincide.
\end{proof}

\begin{theo}\label{lemposetring}
With the above notations and settings, the set
\[
\CH = \{H(I,i)\mid i\in I\}
\]
of
$\BK$-subspaces of $Q(I)$ is independent and the following conditions hold:
\begin{enumerate}
\item Given $i\in I$, the subring $H(I,i)$ of $Q(I)$ has an
identity, given by $\ze_{X(I)_i}$, if and only if $i$ is a maximal
element of $I$; in particular $H(I,i) = \BK$ when $i = \max(I)$.
\item $H(I,i)H(I,j) = \{\mathbf{0}\}$ if $i,j$ are not comparable;
\item Given $i\in I$, if $J\sbs\{i\le\}$ and
$\mathbf{0}\ne\za\in\bigoplus_{j\in J}H(I,j)$, then
\[
\{\mathbf{0}\}\ne H(I,i)\za\sbs H(I,i) \quad\text{and}\quad \{\mathbf{0}\}\ne \za H(I,i)\sbs H(I,i).
\]
in particular there are $Y,Z\in\CR_{i}$ such that
\[
\mathbf{0}\ne\ze_{Y}\za\in H(I,i) \quad\text{and}\quad \mathbf{0}\ne
\za\ze_{Z}\in H(I,i).
\]
Consequently, if $i\le j$, then $H(I,i)H(I,j)\cup H(I,j)H(I,i)\sbs H(I,i)$.
\end{enumerate}
\end{theo}
\begin{proof}
Assume that $J$ is a finite subset of $I$, suppose that $\za = \sum_{j\in
J}\za_j$, where $\mathbf{0}\ne \za_j\in H(I,j)$ for $j\in J$, let us
choose a maximal chain $j_1<\cdots<j_n$ of $J$ and let
$C_n\in\CM(\le j_n)$ be such that $j_1,\ldots,j_{n-1}\in C_n$. Then by Lemma
\ref{nonzerochain2} the $(\CN_{C_n}\times \CN_{C_n})$-blocks of
$\za$ and $\za' = \za_{j_1}+\cdots+\za_{j_n}$ coincide and, on the
other hand, the $(\CN_{C_n}\times \CN_{C_n})$-block of $\za'$ is not
zero by Lemma \ref{nonzerochain}. As a consequence $\za
\ne\mathbf{0}$ and this proves the independence of $\CH$.

(1) If $i\in I$ and $i$ is not a maximal element, then $H(I,i)\is
\FR_{X(i\le)}(\BK)$ as rings, therefore $H(I,i)$ is a ring without an
identity. If $i$ is a maximal element of $I$, then $H(I,i) = \psi_{I,i}\left(\mathbf{1}_{\CFM_{X(i\le)}(\BK)}\BK\right)\is \BK$ and $\ze_{X(I)_i} = \psi_{I,i}\left(\mathbf{1}_{\CFM_{X(i\le)}(\BK)}\right)$ is the multiplicative identity
of $H(I,i)$. If $i = \max(I)$, then $X(I)_i= X(I)$ and so $H(I,i) = \mathbf{1}\BK$ is a unital subring of $Q(I)$.

(2) follows from the property (1) of Proposition \ref{subrings},
since $H(I,i)\sbs S(I,i)$ for all $i\in I$.

(3) Let us prove first that if we take $i,j\in I$ with $i<j$ and $\za\in H(I,j)$, $\zb\in H(I,i)$, then $\za\zb$ and $\zb\za$ are both in $H(I,i)$. According to Proposition \ref{subrings} we have that $\za\zb\in S(I,i)$ and $\zb\za\in S(I,i)$. Given $C\in\CM(\le i)$, we have from
\eqref{mamagicmatrix} that
\begin{equation}\label{equa}
\ze_{\CN_{C}}(\za\zb)\ze_{\CN_{C}} = \ze_{\CN_{C}}\za(\zb\ze_{\CN_{C}}) =
\ze_{\CN_{C}}\za\ze_{\CN_{C}}\zb\ze_{\CN_{C}}.
\end{equation}
By Lemma \ref{HsbsFr} we have that $\ze_{\CN_{C}}\zb\ze_{\CN_{C}}\in
\Soc(Q(I,i))$. On the other hand $\za\in H(I,j)\sbs
S(I,j)\sbs
Q(I,j)\sbs Q(I,i)$. Since $\ze_{\CN_{C}}\in Q(I,i)$ and $\Soc(Q(I,i))$ is an ideal of
$Q(I,i)$, we infer that the third member of \eqref{equa}, and hence the first one, belongs
to $\Soc(Q(I,i))$. As a result $\za\zb\in H(I,i)$, again by Lemma
\ref{HsbsFr}. The proof that $\zb\za\in H(I,i)$ is similar.

Now, let $J$ be a finite subset of $\{i\le\}$, let $\za = \sum_{j\in J}\za_j$, where $\mathbf{0}\ne \za_j\in H(I,j)$ for $j\in J$, let $j_1<\cdots<j_n$ be a maximal chain of $J$, choose ${C_n\in\CM(\le j_n)}$ in such a way that $i,j_1,\ldots,j_n\in C_n$ and set $D = C_n\cap\{\le i\}\in\CM(\le i)$. As seen in the first
part of the present proof, the $(\CN_{C_n}\times \CN_{C_n})$-block
of $\za$ is not zero. Let $x,y\in \CN_{C_n}$ be such that
$\za(x,y)\ne 0$. Since $\CN_{C_n}\sbs \CN_{D}$, there are
(necessarily unique) $V,W\in\CP_{i,D}$ such that $x\in V$ and $y\in
W$. Now $Y =\overline{V},Z =\overline{W}$ are both members of $\CR_{i}$; we have that $x\in Y$ and $y\in Z$, thus both
$\ze_{Y}\za$ and $\za\ze_{Z}$ are nonzero and belong to $H(I,i)$ by the above, because $\ze_{Y},\ze_{Z}\in H(I,i)$.
\end{proof}

A first consequence of Theorem \ref{lemposetring}, as already announced at the beginning of this section, is that for every subset $J$ of $I$ the $\BK$-subspace $H(I,J) = \bigoplus_{j\in J}H(I,j)$
is a $\BK$-subalgebra of $Q(I)$. Of course we have that $H(I,\emptyset) = \{\mathbf{0}\}$ and $H(I,\{i\}) = H(I,i)$ for every $i\in I$.
Note that $H(I,J)$ may fail to be a unital subring of $Q(I)$ and it may
even lack multiplicative identity; we want to investigate under which conditions on $J$ the ring $H(I,J)$ has a multiplicative identity. Firstly, let us consider the subset $X(I,J)$ of $X$ defined by
\begin{gather*}
X(I,J) \defug \bigcup\{X(I)_{i}\mid i\in J\}
= \bigcup\left\{\CN_{C}\mid C\in\CM(\le i), i\in J\right\} \\
= \bigcup\left.\left\{\al^{(I)}\times\CM(\le i\le)\right|
i\in J\right\}
\end{gather*}
(remember that, for every $i\in J$, we have $X(I)_{i} = \bigcup\{\CN_{C}\mid C\in\CM(\le i)\}$ and $\CN_{C} = \bigcup\CP_{i,C}$ for every $C\in\CM(\le i)$). We observe that $X(I,J)$ is the smallest subset of $X(I)$ such that every matrix in $H(I,J)$ has zero entries outside the $(X(I,J)\times X(I,J))$-block. On the other hand, if a matrix $\zu \in Q(I)$ acts as a multiplicative identity on $H(I,J)$, then the following equalities hold:
\begin{equation}\label{eq multunitHJ}
    \zu\ze_{X(I,J)} = \ze_{X(I,J)} = \ze_{X(I,J)}\zu.
\end{equation}
In fact, given $x\in X(I,J)$, there are $i\in J$, $C\in\CM(\le i)$ and
$V\in\CP_{i,C}$ such that $x\in V$ (remember that, for every $i\in J$, we have $X(I)_{i} = \bigcup\{\CN_{C}\mid C\in\CM(\le i)\}$ and $\CN_{C} = \bigcup\CP_{i,C}$ for every $C\in\CM(\le i)$). Inasmuch as
$\ze_{\overline{V}}\in H(I,i)\sbs H(I,J)$, we have that
$\ze_{\overline{V}}\zu = \ze_{\overline{V}} =
\zu\,\ze_{\overline{V}}$ and, since $x\in {\overline{V}}$, we infer that
$\zu(x,y) = \delta(x,y) = \zu(y,x)$ for every $y\in X(I)$, which proves \eqref{eq multunitHJ}.
As a consequence, $H(I,J)+\ze_{X(I,J)}\,\BK$ is the smallest $\BK$-subalgebra of $Q(I)$ which has a multiplicative identity (given by $\ze_{X(I,J)}$) and contains $H(I,J)$ as an ideal; this means that the ring $H(I,J)$ has a multiplicative identity exactly when $\ze_{X(I,J)}\in H(I,J)$.

Let us denote by $\maximal{J}$ the set of those maximal elements of $I$ which are comparable with some element of $J$:
\[
\maximal{J} \defug \ZM(I)\cap\{J\le\}
\]
(recall that $\ZM(I)$ denotes the set of all maximal elements of $I$). Of course it may happen that
$J\not\sbs\{\le\maximal{J}\}$, in particular that $\maximal{J} =
\emptyset$, while $J\ne\emptyset$. If every element of $I$ is bounded by a maximal element or, equivalently, all maximal chains of $I$ have a greatest element, then it is clear that  $X(I,J)\sbs X(I,\maximal{J})$; this inclusion is an equality if and only if, given $m\in\maximal{J}$, every maximal chain of $I$ which is bounded by $m$ contains an element of $J$. Obviously this is the case if $\maximal{J}\sbs J$, in particular when $J$ is an upper subset of $I$; in this latter case it is clear that $\maximal{J} = \ZM(J)$.

\begin{definition}\label{def finshelt}
We say that a subset $J$ of $I$ is \emph{finitely sheltered in} $I$ if the following three conditions hold:
\[
\maximal{J} \text{ is finite,} \quad J\sbs\{\le\maximal{J}\}\quad \text{ and }\quad \maximal{J}\sbs J.
\]
\end{definition}
Note that in this case we have that
\[
X(I,J) = X(I,\maximal{J}).
\]

If $J$ is an upper subset of $I$, then $J$ is finitely sheltered in $I$ if and only if $J$ has a finite cofinal subset; in particular $I$ is finitely sheltered in $I$ exactly when $I$ has a finite cofinal subset and, if it is the case, then every upper subset of $I$ is finitely sheltered in $I$.

\begin{pro}\label{multunit}
Given a nonempty subset $J$ of $I$, the following conditions are equivalent:
\begin{enumerate}
\item $H(I,J)$ has a multiplicative identity (which is necessarily $\ze_{X(I,J)}$).
\item $\ze_{X(I,J)}\in H(I,J)$.
\item $J$ is finitely sheltered in $I$.
\end{enumerate}
If any (and hence all) of these conditions holds, then
\begin{equation}\label{multuniteq}
  \ze_{X(I,J)} =  \ze_{X(I,\maximal{J})} = \sum_{m\in\maximal{J}}\ze_{X(I)_{m}}.
\end{equation}
\end{pro}
\begin{proof}
(1)$\Leftrightarrow$(2): see the previous discussion.

(1)$\Rightarrow$(3). Suppose that (1) holds. This means that there is a finite subset $F\sbs J$ and nonzero matrices $\zd_{i}\in H(I,i)$, for $i\in F$, such that
  \begin{equation}\label{bobobo}
     \ze_{X(I,J)} = \sum_{i\in F}\zd_{i}.
\end{equation}
Let us prove first that $J\sbs\{\le F\}$. Given $j\in J$, it follows from \ref{bobobo} that
\[
H(I,j) = \sum_{i\in F}\zd_{i}H(I,j)
\]
and so, by letting $F'$ be the subset of $F$ of those elements $i$ for which $\zd_{i}H(I,j) \ne 0$, it follows from Theorem \ref{lemposetring}, (2) that $j$ is comparable with all $i\in F'$. Assume that $i<j$ for all $i\in F'$, so that $\zd_{i}H(I,j) \sbs H(I,i)$ for all $i\in F'$ again by Theorem \ref{lemposetring}, (2). Then we have that
\[
H(I,j) \sbs \sum_{i\in F}H(I,i),
\]
which is impossible because $H(I,j)$ and $\sum_{i\in F}H(I,i)$ are independent $\BK$-subspaces of $Q(I)$ by Theorem \ref{lemposetring}. We conclude that $j\le i$ for some $i\in F$, as wanted.

Next, we claim that every maximal element of $F$ is maximal in $I$. Let $i_1<\ldots<i_n$ be a maximal chain of $F$. If we consider any $C\in\CM(\le i_n)$ such that $i_1,\ldots,i_{n-1}\in C$, then it follows from Lemma \ref{nonzerochain2} that the $\CN_{C}\times\CN_{C}$-blocks of $\ze_{X(I,J)}$ and $\zd_{i_1}+\cdots+\zd_{i_n}$ coincide. If $i_n$ were not maximal in $I$, it would follow from Lemma \ref{HsbsFr}, (2) and Theorem \ref{theo-posetflr}, (3) that there are $X_1,\ldots,X_r\in \CP_{i_n}$ such that, given $x\in \CN_{C}$, the $x$-th row of $\ze_{X(I,J)}$ is not zero only if $x\in X_1\cup\cdots\cup X_r$. On the other hand, the $\CN_{C}\times\CN_{C}$-block of $\ze_{X(I,J)}$ is a nonzero scalar matrix and therefore the $x$-th row of $\ze_{X(I,J)}$ is not zero for every $x\in \CN_{C}$. Since $\CN_{C}$ is the union of $\al$ elements of the partition $\CP_{i_{n}}$ (see Notations and Settings \ref{not main}), we have a contradiction and our claim is proved.

Finally, let us show that $\maximal{J}\sbs F$, from which it will follow that $\maximal{J}$ is finite and is contained in $J$, completing the proof that $J$ is finitely sheltered.
Assume, on the contrary, that there is some maximal element $m$ of $I$ and some $j\in J$ such that $m\not\in F$ but $j<m$, let $C\in\CM(I)$ be such that $j,m\in C$ and put $D = C\cap\{\le j\}\in\CM(\le j)$. We have from Lemma \ref{lem-partition2}, (2) that
\[
  \vu\ne \CN_{C}\sbs\CN_{D} \sbs X(I)_{j} \sbs X(I,J),
\]
meaning that the $\CN_{C}\times\CN_{C}$-block of
$\ze_{X(I,J)}$ is not zero. If $i_{1},\ldots,i_{r}$ are the
elements of $F$ such that the $\CN_{C}\times\CN_{C}$-blocks of
$\zd_{i_{1}},\ldots,\zd_{i_{r}}$ are not zero, then $i_{1},\ldots,i_{r}\in\{<m\}$. Indeed, given $t\in\{1,\ldots,r\}$, there is some $E\in\CM(\le i_{t})$ such that $\CN_{D}\cap \CN_{E}\ne\emptyset$, hence $i_{t}$ and $m$ are comparable by Lemma \ref{lem-partition2}, (1) and therefore $i_{t}<m$. We have now
\begin{equation}\label{eq enc}
 \mathbf{0}\ne\ze_{\CN_{C}} =
 \zd_{i_{1}}(\CN_{C},\CN_{C})+\cdots+\zd_{i_{r}}(\CN_{C},\CN_{C}).
\end{equation}
Given $t\in\{1,\ldots,r\}$, since $i_{t}$ is not maximal in $I$, according to Lemma \ref{HsbsFr}, (2) and Theorem \ref{theo-posetflr}, (3) there are finitely many members of $\CP_{i_{t}}$ whose union contains the set of those $x\in X(I)$ such that the $x$-row of $\zd_{i_{t}}$ is not zero. On the other hand, for every $t\in\{1,\ldots,r\}$ we have that $i_{t}<m$ and hence the partition $\CP_{m}$ is coarser than $\CP_{i_{t}}$. We infer that there are $Z_{1},\ldots,Z_{s}\in\CP_{m}$ such that if the $x$-th row of $\zd_{i_{t}}$ is not zero for some $t\in\{1,\ldots,r\}$, then $x\in Z_{1}\cup\cdots\cup Z_{s}$. Consequently we get from \eqref{eq enc} that $\CN_{C}\sbs Z_{1}\cup\cdots\cup Z_{s}\in\CP_{m}$, in contradiction with the fact that $\CN_{C}$ is the union of $\al$ members of $\CP_{m}$ (see Notations and Settings \ref{not main}). This shows that $\maximal{J}\sbs F$, as
 wanted.

    (3)$\Rightarrow$(2) Assume (3). Then it follows that $\{X(I)_{m}\mid m\in\maximal{J}\}$ is a finite partition of $X(I,J)$ (see the proof of Proposition \ref{pro max}). Thus \eqref{multuniteq} holds and, since $\ze_{X(I)_{m}}$ is in $H(I,m)$ for all $m\in\maximal{J}$, it follows that $\ze_{X(I,J)} \in H(I,J)$.
\end{proof}

\begin{cor}\label{cor fincof}
The ring $H(I,I)$ has a multiplicative identity if and only if $I$ has a finite cofinal subset. If it is the case and $J\sbs I$, then $H(I,J)$ is a unital subring of $H(I,I)$ if and only if all maximal elements of $I$ belong to $J$.
\end{cor}

\section{Representing a partially ordered set $I$: the unit-regular $\BK$-algebra $\mathfrak{B}(I)$.}\label{sect algebraB}

Henceforth, unless otherwise stated, we assume that the $\al$-poset $I$ we consider \emph{has a finite cofinal subset}. Starting from this section we focus our interest on the main object of our work: the $\BK$-subalgebra
\[
\mathfrak{B}(I) \defug H(I,I) = \bigoplus_{i\in I}H(I,i)
\]
of $Q(I) = \CFM_{X(I)}(\BK)$, which is unital by Proposition \ref{pro max} and Corollary \ref{cor fincof}. Obviously $\mathfrak{B}(I)$ depends on the cardinal $\al$ and the field $\BK$; if we need to emphasize this we use the notation $\mathfrak{B}_{\al,\BK}(I)$.

Our first result is a complete description of ideals of $H(I,J)$, for any subset $J\sbs I$, and hence of the whole $\BK$-algebra $\mathfrak{B}(I)$.

 \begin{theo}\label{theo ideals}
 Given a subset $J$  of $I$, the assignment $L\mapsto H(I,L)$ defines a lattice isomorphism from the lattice ${\Downarrow\!\! J}$ of all lower subsets of $J$ to the lattice $\BL_{2}(H(I,J))$ of all ideals of $H(I,J)$. In particular, the lattice of all ideals of $\mathfrak{B}(I)$ is isomorphic to $\Downarrow\!\! I$.
 \end{theo}
\begin{proof}
At first, let us prove that if $L\sbs J\sbs I$, then $H(I,L)$ is an ideal of $H(I,J)$ if and only if $L$ is a lower subset of $J$. Assume that $H(I,L)$ is an ideal of $H(I,J)$ and take
$i\in L$, $j\in J$ with $j\le i$. If $j\not\in L$, then $H(I,j)\sbs
H(I,J\setminus L)$, while it follows from Theorem \ref{lemposetring}
that
\[
\{\mathbf{0}\}\ne H(I,i)H(I,j)\sbs H(I,j)\cap H(I,L) \sbs
H(I,J\setminus L)\cap H(I,L) = \{\mathbf{0}\},
\]
hence a
contradiction. Thus necessarily $j\in L$. Conversely, suppose that
$L$ is a lower subset of $J$ and let $i\in L$, $j\in J$.
Then exactly one of the following conditions occurs: a) $j\in
L$, b) $j\not\in L$ and $i,j$ are not comparable, c) $j\not\in L$ and $j>i$. In
all cases it follows from Theorem \ref{lemposetring} that
$H(I,j)H(I,i)\cup H(I,i)H(I,j)\sbs H(I,L)$. Since $H(I,L)$
is already a $\BK$-subspace of $H(I,J)$, this is sufficient to
conclude that it is an ideal of $H(I,J)$.

Suppose now that $\mathfrak{I}$ is an ideal of $H(I,J)$. As $\{\mathbf{0}\} = H(I,\emptyset)$, we may assume that $\mathfrak{I} \ne \{\mathbf{0}\}$ and observe that if $\mathfrak{I}\cap H(I,j)\ne \{\mathbf{0}\}$ for some $j\in J$, then necessarily $H(I,j) = \mathfrak{I}\cap H(I,j)\sbs \mathfrak{I}$, because $H(I,j)$ is a simple ring. Our goal is to prove that $\mathfrak{I} = H(I,L)$, where $L = \{j\in J\mid H(I,j)\sbs \mathfrak{I}\}$.
Suppose that $\mathbf{0}\ne\za\in\mathfrak{I}$  and let us consider the unique decomposition
\[
\za = \sum\{\za_{j}\mid j\in F\}
\]
for some finite subset $F$ of $J$ and nonzero elements $\za_{j}\in H(I,j)$, for $j\in F$. We will reach our goal once we show that $\za_{j}\in \mathfrak{I}$ for every $j\in F$. As the statement is obvious in case $|F| = 1$, let us consider an integer $n\ge 2$, assume inductively that the statement is true whenever $|F|<n$ and suppose that $|F| = n$. We claim that $\mathfrak{I}\cap H(I,k)\ne \{\mathbf{0}\}$ for at least one $k\in F$. If $F$ is an antichain, given any $k\in F$ and an idempotent $\ze$ of $H(I,k)$ such that $\za_{k} = \ze\za_{k}$, it follows from (2) of Theorem \ref{lemposetring} that
\[
\mathbf{0}\ne\za_{k} = \ze\za_{k} = \ze\za \in \mathfrak{I}\cap H(I,k).
\]
Assume that $F$ is not an antichain, let us choose a maximal element $k$ of $F$ which is not minimal in $F$, set $\{j_{1},\ldots,j_{r}\} = \{<k\}\cap F$ and choose any $Z\in\CR_{k}$ such that $\ze_{Z}\za_{k}\ne\mathbf{0}$; note that if $k$ is a maximal element in $I$, then $\za_{k} = \ze_{X(I)_k}a$ for some nonzero $a\in\BK$, therefore
$\mathbf{0}\ne\ze_{Z}\za_{k} = \ze_{Z}\za \in \mathfrak{I}\cap H(I,k)$ for every $Z\in\CR_{k}$. If $k$ is not maximal in $I$ it follows from Lemma \ref{HsbsFr}, (3) that, for every $\a\in\{1,\ldots,r\}$, there are $Y_{\a 1},\ldots,Y_{\a s_{\a}}\in\CR_{j_{\a}}$ such that, by setting $Y_{\a} = Y_{\a 1}\cup\cdots\cup Y_{\a s_{\a}}$, the $x$-th row of $\za_{j_{\a}}$ is not zero only if $x\in Y_{\a}$. Set $Y = Y_{1}\cup\cdots\cup Y_{r}$. We have that $\za_{j_{\a}} = \ze_{Y}\za_{j_{\a}}$ for every $\a\in\{1,\ldots,r\}$ and, on the other side, it follows from (2) and (3) of Lemma \ref{lem idempprim Hi} that there is some $T\in\CR_{j_{1}}\cup\cdots\cup\CR_{j_{r}}$ such that $T\sbs Z$ and $T\cap Y = \vu$, thus
\[
\ze_{T}\za_{j_{1}} = \cdots = \ze_{T}\za_{j_{r}}  = \mathbf{0}.
\]
In addition, if $l\in F\setminus\{\le k\}$ and $U\in\CR_{l}$, then $U\cap Z = \vu$ by Lemma \ref{lem idempprim Hi}, (1) and so $U\cap T = \vu$. Consequently, by (2) of Theorem \ref{lemposetring} we have that $\ze_{T}\za_{l} = \mathbf{0}$ for every $l\in F\setminus\{\le k\}$. It follows from (5) of Lemma \ref{lem idempprim Hi} that the $x$-th row of $\za_{k}$ is not zero for all $x\in Z$; thus, since $\vu\ne T\sbs Z$, we have that $\ze_{T}\za_{k}\ne \mathbf{0}$. By taking (3) of Theorem \ref{lemposetring} into account, we conclude that
\[
\mathbf{0}\ne \ze_{T}\za_{k} = \ze_{T}\za \in \mathfrak{I}\cap H(I,{k})
\]
and our claim is proved. As seen previously, we have that $H(I,k) \sbs \mathfrak{I}$ and therefore $\za_{k}\in \mathfrak{I}$. Consequently $\za - \za_{k}\in \mathfrak{I}$ and the inductive assumption gives us that $\za_{j}\in \mathfrak{I}$ for all $j\in F\setminus\{k\}$. We have so far established the first of the two inclusions
\[
\mathfrak{I} \sbs H(I,L) = \bigoplus_{l\in L}H(I,l)\sbs\mathfrak{I},
\]
the second being clear; thus we have the equality, as wanted. Finally $L$ is a lower subset of $J$, as we have seen at the beginning of the present proof, while the the fact that the assignment $L\mapsto H(I,L)$ defines a lattice isomorphism from $\Downarrow\!\! J$ to $\BL_{2}(H(I,J))$ is straightforward.
\end{proof}

\begin{re}\label{remark split}
If $J\sbs I$ and $L$ is a lower subset of $J$, then the projection map
\[
\lmap{\f_{L,J}}{H(I,J)}{H(I,J\setminus L)}
\]
with kernel $H(I,L)$, that is
\begin{equation}\label{mapfKJ}
\f_{L,J}(\za'+\za'') = \za' \quad\text{ for
all\,\, $\za'\in H(I,J\setminus L),\,\, \za''\in H(I,L)$},
\end{equation}
is a surjective $\BK$-algebra homomorphism which induces an isomorphism of $\BK$-algebras
\[
H(I,J)/H(I,L) \is H(I,J\setminus L).
\]
Let us say that an ideal $\mathfrak{I}$ of a $\BK$-algebra $R$ is \emph{splitting} if there is a $\BK$-subalgebra $S$ of $R$ such that $R = \mathfrak{I}\oplus S$ as $\BK$-vector spaces or, equivalently, the canonical ring epimorphism $R\to R/\mathfrak{I}$ restricts to an isomorphism from $S$ to $R/\mathfrak{I}$. Then it follows from Theorem \ref{theo ideals} that every ideal of $\mathfrak{B}(I)$ is splitting.
\end{re}

The next main feature of the algebra $\mathfrak{B}(I)$ is that it is locally matricial. We recall that a \emph{matricial $\BK$-algebra} is any algebra isomorphic to
\[
\BM_{n_{1}}(\BK)\times\cdots\times \BM_{n_{r}}(\BK)
\]
for some positive integers $n_{1},\ldots, n_{r}$ (see \cite{Good:3}, p.217, for instance). A $\BK$-algebra $R$ is \emph{locally matricial} if every \emph{finite} subset of $R$ is contained in some matricial $\BK$-subalgebra.

Given $i\in I$, since $H(I,i)$ is isomorphic to $\Soc(S(I,i))$, by Litoff's Theorem (see \cite{Jacobson:1}, p.90 and \cite{FaithUtumi:1}) $H(I,i)$ is locally matricial. Note that if $\mathbf{e}$ is any idempotent of $H(I,i)$, then $\mathbf{e}H(I,i)\mathbf{e}$ is a matricial $\BK$-algebra, as it is isomorphic to the endomorphism ring of a finite direct sum of copies of the simple left $H(I,i)$-module $\BK^{(X(I))}$, whose endomorphism ring is isomorphic to $\BK$.

\begin{pro}\label{locmatrix}
If $J$ is any subset of $I$, then the $\BK$-algebra $H(I,J)$ is locally matricial. In particular $\mathfrak{B}(I)$ is locally matricial.
\end{pro}
\begin{proof}
The thesis might be obtained through an easy induction by using the following very recent and much more general result of Goodearl (see \cite[Theorem 3.4]{Good:25}): assume that a $\BK$-algebra $R$ is an extension of a $\BK$-algebra $S$ by a $\BK$-algebra $T$ (in the sense that $R$ contains an ideal $\mathfrak{I}$ such that, as $\BK$-algebras, $\mathfrak{I}\is S$ and $R/\mathfrak{I}\is T$); if $S$ and $T$ are locally matricial, then $R$ is locally matricial as well. However we present a more direct proof which only uses a partial result from \cite{Good:25}.

Clearly, we may assume that $J$ is finite and proceed by induction on the number $n = |J|$. If $n = 0$ then $H(I,J) = \{\mathbf{0}\}$ and there is nothing to prove. Given $n>0$, assume that $H(I,J)$ is locally matricial whenever $|J| = n-1$ and let us consider the case in which $|J| = n$. After choosing a minimal element $m$ of $J$ and by setting $S = H(I,m)$, we have that
\[
H(I,J) = S\oplus H(I,J\setminus\{m\})
\]
as $\BK$-vector spaces. Thus, in order to prove that $H(I,J)$ is locally matricial, it is sufficient to prove that if $F$ and $G$ are finite subsets of $S$ and $H(I,J\setminus\{m\})$ respectively, then there is a matricial $\BK$-subalgebra of $H(I,J)$ containing $F\cup G$. By the inductive assumption, there is a matricial $\BK$-subalgebra $T$ of $H(I,J\setminus\{m\})$ containing $G$. Set $\widehat{J} = J\cup \maximal{J}$ and note that, since we are assuming that $I$ has a finite cofinal subset, then $\widehat{J}$ is finitely sheltered; it follows from Proposition \ref{multunit} that $H(I,\widehat{J})$ has a multiplicative identity $\mathbf{u}$; set $U = T+\BK\mathbf{u}$ and note that $U$ is a matricial $\BK$-algebra, because $T+\BK\mathbf{u} = T+\BK(\mathbf{u}-\mathbf{f}) \is T\times\BK$, where $\mathbf{f}$ is the multiplicative identity of $T$ (note that $\mathbf{u}-\mathbf{f}$ is a central idempotent of $T+\BK\mathbf{u}$).

Let us consider the unital $\BK$-subalgebra $R = S+U$ of $H(I,\widehat{J})$ and note that $S\cap U = \{\mathbf{0}\}$. In addition $R$ is regular by \cite[Lemma 1.3]{Good:3}, because so are $S$ and $U$. According to Proposition 3.3 of \cite{Good:25} there is an idempotent $\mathbf{e}\in S$ such that $F\sbs \mathbf{e}S\mathbf{e}$ and $\mathbf{e}$ centralizes $U$. By setting $V\defug \mathbf{e}S\mathbf{e}+ T$, as $\mathbf{e}$ centralizes $T$ and by using the fact that $S$ is an ideal of $H(I,\widehat{J})$ we see that $V^{2} = V$. Thus $V$ is a $\BK$-subalgebra of $H(I,J)$ in which $\mathbf{e}$ is a central idempotent. We have that $\mathbf{e}S\mathbf{e}\cap T = \{\mathbf{0}\}$; moreover, by using again the fact that $S$ is an ideal of $H(I,\widehat{J})$ and $T\sbs H(I,\widehat{J})$, we have that $\mathbf{e}T \sbs S$ and so $\mathbf{e}T = T\mathbf{e}\sbs \mathbf{e}S\mathbf{e}$, from which $\mathbf{e}V = \mathbf{e}S\mathbf{e}$. We infer that $(\mathbf{1}-\mathbf{e})V \is T$ as $\BK$-algebras and therefore $V\is \mathbf{e}S\mathbf{e}\times T$. Since $\mathbf{e}S\mathbf{e}$ is a matricial $\BK$-algebra, then so is $V$ and $V$ contains $F\cup G$, as wanted.
\end{proof}

We recall that a unital ring $R$ is \emph{unit-regular} if for every $x\in R$ there exists a unit $y\in R$ such that $x = xyx$. It is well known that a unital regular ring $R$ is unit regular if and only if, given three finitely generated projective right $R$-modules $A,B,C$, the condition $A\oplus C\is B\oplus C$ implies $A\is B$. If $R$ is a locally matricial $\BK$-algebra, then $R$ is regular and it is unit-regular in case it has a multiplicative identity. Indeed, given $x\in R$, there is a matricial $\BK$-subalgebra $S$ of $R$ such that $x\in S$. Since $S$ is unit-regular, then there is a unit $y$ of $S$ such that $x = xyx$. If $e$ denotes the multiplicative identity of $S$, we see immediately that $z = y+1-e$ is a unit of $R$ and $x = xzx$. If $R$ is not necessarily unital, one says that $R$ is \emph{locally unit-regular} if, for every idempotent $e\in R$, the unital ring $eRe$ is unit regular (see \cite{Good:25}). If $R$ is a locally matricial $\BK$-algebra and $e$ is an idempotent of $R$, then $eRe$ is locally matricial by \cite[Lemma 3.1]{Good:25}, consequently $R$ is locally unit-regular and we get the following

\begin{cor}\label{corlocmatrix}
If $J$ is any subset of $I$, then the $\BK$-algebra $H(I,J)$ is locally unit-regular. Consequently $H(I,J)$ is unit-regular in case $J$ is finitely sheltered, in particular $\mathfrak{B}(I)$ is unit-regular.
\end{cor}

If $J\sbs I$, let us denote by $J_1$ the set of all
minimal elements of $J$. The {\em dual classical
Krull filtration\/} of the poset $J$ is the ascending chain
$(J_{\a})_{0\le\a}$ of subsets of $J$, indexed over the ordinals, defined recursively as follows (see
\cite{Albu:10}):
\begin{gather*}
J_{0} \defug \emptyset, \\
J_{\a+1} \defug J_{\a}\cup\left(J\setminus J_{\a}\right)_1
\quad \text{ for all }\a, \\
J_{\a} \defug \bigcup_{\b<\a}J_{\b}\quad \text{ if $\a$ is a limit
ordinal.}
\end{gather*}
Note that each $J_{\a}$ is a lower subset of $J$. Clearly there exists a smallest ordinal $\xi$ such that $J_{\xi+1} =
J_{\xi}$ and hence $J_{\a} =
J_{\xi}$ for every $\a>\xi$; moreover $J$ is artinian (i.~e. it satisfies the DCC or,
equivalently, every chain of $J$ is well ordered) if and only if $J
= J_{\xi}$ and, in this case, the ordinal $\xi$ is called the {\em
dual classical Krull dimension\/} of $J$ and the family $\left(J_{\a+1}\setminus J_{\a}\right)_{\a<\xi}$ is a partition of $J$.

Given $J\sbs I$, the next result illustrates the tight connection between the dual classical Krull filtration of $J$ and the Loewy chain $(\Soc_{\a}(H(I,J)))_{0\le\a}$ of ideals of $H(I,J)$.

We recall that if $R$ is a regular ring (unital or not), then every homogeneous component of $\Soc(R_{R})$ is a minimal ideal and $\Soc(R_{R})  = \Soc({_{R}R})$, from which one can easily infer that $\Soc_{\a}(R_{R})  = \Soc_{\a}({_{R}R})$ for every ordinal $\a$.

\begin{theo}\label{artin}
     Let $J$ be a subset of $I$. For every ordinal $\a$ we have that
    \begin{equation}\label{eq loew}
        \Soc_{\a}(H(I,J)) = H(I,J_{\a}).
    \end{equation}
    Consequently $H(I,J)$ is semiartinian if and only if $J$ is artinian, in which the case the Loewy length of $H(I,J)$ equals the dual classical Krull dimension of $J$. In particular $\mathfrak{B}(I)$ is semiartinian if and only if $I$ is artinian.
    \end{theo}
    \begin{proof}
Obviously \eqref{eq loew} holds if $\a = 0$. Now the ideal $H(I,J_{1})$ is semisimple as a right ideal of $H(I,J)$, because it is a direct sum of the simple and semisimple rings $H(I,i)$, for $i\in J_{1}$. On the other hand, since $H(I,J)$ is regular, each homogeneous component of $\Soc(H(I,J))$ is a minimal ideal and so, by Theorem \ref{theo ideals}, it is of the form $H(I,m)$ for some minimal element $m$ of $J$. This shows that \eqref{eq loew} holds if $\a = 1$. Given an ordinal $\a>0$, assume inductively that
\[
\Soc_{\b}(H(I,J)) = H(I,J_{\b})
\]
whenever $\b<\a$. If $\a$ is a limit ordinal, since $J_{\a} = \bigcup_{\b<\a}J_{\a}$, then we have
\[
\Soc_{\a}(H(I,J)) = \bigcup_{\b<\a}\Soc_{\b}(H(I,J))
 = \bigcup_{\b<\a}H(I,J_{\b}) = H(I,J_{\a}).
 \]
Suppose that $\a = \b+1$ for some $\b$, set $J_{\b}^{\bullet\bullet} = J\setminus J_{\b}$ and $J_{\b+1}^{\bullet} = (J\setminus J_{\b})_{1}$, so that $J_{\b+1}$ is the disjoint union of $J_{\b}$ and $J_{\b+1}^{\bullet}$. The projection map
\[
\lmap{\f_{J_{\b},J}}{H(I,J)}{H(I,J_{\b}^{\bullet\bullet})}
\]
with kernel $H(I,J_{\b})$ induces an isomorphism of $\BK$-algebras
\[
H(I,J)/H(I,J_{\b}) \is H(I,J_{\b}^{\bullet\bullet})
\]
(see Remark \ref{remark split}). Since
\begin{gather*}
\f_{J_{\b},J}(H(I,J_{\b+1}))
= \f_{J_{\b},J}[H(I,J_{\b})\oplus H(I,J_{\b+1}^{\bullet})] = H(I,J_{\b+1}^{\bullet})\\
= H(I,(J_{\b}^{\bullet\bullet})_{1})
= \Soc[H(I,(J_{\b}^{\bullet\bullet})],
\end{gather*}
we infer that
\[
H(I,J_{\b+1})/H(I,J_{\b})= \Soc[H(I,J)/H(I,J_{\b})]
\]
and hence
\[
H(I,J_{\b+1}) = \Soc_{\a}(H(I,J)).
\]
The last statement follows directly from the definitions, toghether with the fact that if $K,L\sbs I$, then $H(I,K)\sbs H(I,L)$ if and only if $K\sbs L$.
\end{proof}

\begin{re}
Given a division ring $D$ and an artinian partially ordered set $I$, in \cite{Bac:16} we presented a construction which produces a semiartinian and unit-regular ring $D_{I}$, which is isomorphic to a subring of $\CFM_{\al}(D)$ containing $D$ (this latter identified with the subring of all scalar matrices), where $\al$ is the first infinite cardinal such that $\al\ge |I|$ and $\al\ge |\CM(I)|$. The construction, heavily based on ordinal arithmetic, turns out to be rather tricky and cumbersome with no possibility to extend it to non-artinian partially ordered sets. The present construction overcomes this problem and, in addition, if $I$ is an artinian $\al$-poset having a finite cofinal subset, the algebra $\mathfrak{B}(I)$ satisfies the same properties (1), (2), (3), (4), (5) and (8) of $D_{I}$, as stated in \cite[Theorem 9.5]{Bac:16}; \emph{mutatis mutandis}, their proofs are identical.
\end{re}

\section{$\mathbf{K}_0(\mathfrak{B}(I))$.}\label{sect algebraBKO}

We keep going on with the assumption that $I$ is an $\al$-poset having a finite cofinal subset. In this section we take on the task of investigating the structure of the partially ordered abelian group $\mathbf{K}_0(\mathfrak{B}(I))$. It will turn out that $\mathbf{K}_0(\mathfrak{B}(I))$ is order isomorphic to a group which is a special instance of a more general construction, well known within the literature of lattice-ordered groups (see \cite{AndersonFeil:01}, page 6 for instance). Specifically, given a family $\CH = (H_{i})_{i\in I}$ of partially ordered abelian groups, the \emph{Hahn group} $V(I,H_{i})$ of $\CH$ is the subgroup of the direct product $\prod_{i\in I}H_{i}$ consisting of all elements having noetherian support, endowed with the partial order induced by declaring positive those elements having positive coordinates relative to the maximal elements of the support. The subgroup $\Sigma(I,H_{i})$ of $V(I,H_{i})$ of those elements having finite support, endowed with the induced partial order, is called the \emph{restricted Hahn group}; of course, as a group, $\Sigma(I,H_{i})$ is merely the direct sum $\bigoplus_{i\in I}H_{i}$. When all $H_{i}$ coincide with a single group $H$, the groups $V(I,H_{i})$ and $\Sigma(I,H_{i})$ are called respectively the \emph{Hahn power} and the \emph{restricted Hahn power of $H$ by $I$}. Our main result is to show that $\mathbf{K}_0(\mathfrak{B}(I))$ is order isomorphic to the restricted Hahn power of $\BZ$ by $I$, whose support is just the free abelian group with $I$ as a basis; we give here a description of it with some more details, suited for our needs.

Let us denote by $G(I)$ the free abelian group with $I$ as a basis; we keep the standard practice by representing each element $x\in G(I)$ as a unique linear combination of elements of $I$ with coefficients in $\BZ$ and, if $i\in I$, we denote by $x_{i}$ the coefficient of $i$. As usual, the \emph{support} of an element $x\in G(I)$ is the set $\Supp(x)\defug \{i\in I\mid x_{i}\ne 0\}$. As it was shown in \cite{BacSpin:17}, the set
\begin{equation}\label{eq M(I)}
M(I) \defug \{x\in G(I)\mid x_{i}>0 \text{ if $i$ is a maximal element of $\Supp(x)$}\}
\end{equation}
is a conical submonoid of $G(I)$, in the sense that if $x,y\in M(I)$ and $x+y = 0$, then $x = 0 = y$. We will consider $G(I)$ partially ordered with $M(I)$ as positive cone, so that $G(I)$ is the \emph{restricted Hahn power} of $\BZ$ by $I$. It was shown in \cite[Proposition 4.2]{BacSpin:17} that $G(I)$ is a dimension group in which the sum of all maximal elements of $I$ is an order-unit and, consequently, $G(I)$ is directed and $M(I)$ is a Riesz decomposition monoid, in the sense that if $x,y,z\in M(I)$ and $x\le y+z$, then there exist $a,b\in M(I)$ such that $a\le x$, $b\le y$ and $x = a+b$ (see \cite[Ch. 3]{Good:10}).

The following two results will not be used in the present work; nonetheless these propositions illustrate remarkable properties of the group $G(I)$. We recall that an element $p$ of an abelian monoid $M$ is \emph{prime} if it is not invertible and, whenever $p\le x+y$ for some $x,y\in M$, then either $p\le x$ or $p\le y$ (here ``$\le$'' is the algebraic preorder: if $a,b\in M$, then $a\le b$ means that $b = a+c$ for some $c\in M$). If $M$ is a conical, cancellative Riesz decomposition monoid (hence the algebraic preorder is a partial order), it is easy to see that a non-invertible element $p\in M$ is prime if and only if, given $a,b\in M$, the equality $p = a+b$ implies that either $p = a$ or $p = b$.

\begin{pro}\label{pro primehahn}
The prime elements of the monoid $M(I)$ are precisely the mini\-mal elements of $I$.
\end{pro}
\begin{proof}
Let us consider the subset
\[
\CA = \{i+x\mid x\in G(I), i\in I, \Supp(x)\sbs\{<i\} \text{ and $x_{j}<0$ for all $j\in\Supp(x)$}\}
\]
of $M(I)$, let $\langle\CA\rangle$ be the submonoid of $M(I)$ generated by $\CA$ and let us prove first that $\langle\CA\rangle = M(I)$. Given $x\in M(I)$, if $|\Supp(x)|\le 1$, then it is clear that $x\in \langle\CA\rangle$. Given a positive integer $n$, assume that $x\in \langle\CA\rangle$ whenever $x\in M(I)$ and $|\Supp(x)|<n$, let $x\in M(I)$ be such that $|\Supp(x)|=n$ and let $m$ be a maximal element of $\Supp(x)$. By setting
\[
y = x_{m}m + \sum\{x_{i}i\mid i\in\{<m\}\text{ and }x_{i}<0\},
\]
we have that $x' = x-y\in M(I)$. In fact, let $k$ be a maximal element of $\Supp(x')$. If $k$ is maximal in $\Supp(x)$ too, since $\Supp(x') = \Supp(x)\setminus\Supp(y)$ we have that $x'_{k} = x_{k}>0$. Otherwise there is some $h\in\Supp(x)$ such that $k<h$. Necessarily $h\in\Supp(y)\sbs\{\le m\}$, from which $k< m$. Since $k\ne m$ and $k\nin\Supp(y)$, we conclude that $x_{k}\ge 0$. As $|\Supp(x')|<n$, by the inductive assumption we have that $x'\in \langle\CA\rangle$; since $y\in \langle\CA\rangle$, it follows that $x = x'+y\in \langle\CA\rangle$, as wanted.

Let $p$ be a prime element of $M(I)$. We have that $p = x_{(1)}+\cdots+x_{(n)}$ for some $x_{(1)},\ldots,x_{(n)}\in\CA$. Since $M(I)$ is a conical, cancellative refinement monoid, then $p = x_{(\a)}$ for some $\a\in\{1,\ldots,n\}$, that is there are $i\in I$ and $y\in G(I)$, with $\Supp(y)\sbs\{<i\}$ and $y_{j}<0$ for $j\in\Supp(y)$, in such a way that $p = i+y$. Thus $p = (i+2y)+(-y)$ and, since both $i+2y$ and $-y$ are in $M(I)$ and $p\ne -y$, primality of $p$ implies that $p = i+2y$. Consequently $y = 0$ and so $p = i$. If $i$ were not minimal, that is $i>j$ for some $j\in I$, then $p = i = (i-j)+j$. Again, since both $i-j$ and $j$ are in $M(I)$, primality of $p$ implies that either $i = i-j$ or $i = j$, which is impossible.

Conversely, let $i$ be a minimal element of $I$, let $x,y\in M(I)\setminus \{0\}$, assume that $i = x+y$ and let $j$ be a maximal element of $\Supp(x)\cup\Supp(y)$ such that $j\ne i$. Then $x_{j} + y_{j} = 0$ and so either $x_{j}<0$, or $y_{j}<0$: a contradiction with the assumption that $x$ and $y$ are nonzero elements of $M(I)$. This means that $\Supp(x)\cup\Supp(y) = \{i\}$ and, since $M(I)$ is cancellative, we conclude that either $x = i$ or $y = i$.
\end{proof}

If $G$ is a partially ordered abelian group, we denote by $G^{+}$ the positive cone of $G$, namely $G^{+} = \{x\in G\mid 0\le x\}$. A subgroup $H$ of $G$ is an \emph{ideal} of $G$ if $H$ is a directed and convex subgroup (see \cite{Good:10}, page 8). We recall that $H$ is directed if and only if for every $x\in H$ there are $x',x''\in H^{+} = G^{+}\cap H$ such that $x = x'-x''$ (see \cite[Proposition 15.16]{Good:3}).

\begin{pro}\label{pro idealhahn}
The assignment $L\mapsto G(L)$ defines an isomorphism from the lattice of all lower subsets of $I$ to the lattice of all ideals of $G(I)$.
\end{pro}
\begin{proof}
Let $L$ be a lower subset of $I$ and let us prove first that $G(L)$ is convex. Given $x\in G(I)$ and $y\in G(L)$, assume that $0< x \le y$ but $x\nin G(L)$. Then there exists some $i\in I\setminus L$ such that $x_{i}\ne 0$ and, since $\Supp(x)$ is finite, there is a maximal element $m$ of $\Supp(x)$ such that $i\le m$ and $x_{m}>0$, because $x\in M(I)$; note that $m\nin L$ otherwise, since $L$ is a lower subset of $I$, from $i\le m$ we would get $i\in L$. By setting $z = y-x$, we observe that $m$ must be maximal in $\Supp(z)\sbs \Supp(x)\cup \Supp(y)$, otherwise there would be some $l\in\Supp(y)\sbs L$ such that $i\le m\le l$, from which $i\in L$. Since $z\in M(I)$, then $z_{m}>0$, in contradiction with the fact that, $z_{m} = -x_{m}<0$, because $y_{m} = 0$. Thus necessarily $x\in G(L)$ and so $G(L)$ is convex.

Next, let $0\ne x\in G(L)$. Then $x = x'-x''$, where
\[
x' = \sum\{x_{i}\cdot i\mid i\in\Supp(x), x_{i}>0\}\text{ and }
x'' = \sum\{-x_{i}\cdot i\mid i\in\Supp(x), x_{i}<0\}.
\]
Since $x', x''\in M(L) = G(L)\cap M(I)$, we infer that $G(L)$ is directed and we conclude that $G(L)$ is an ideal of $G(I)$.

Conversely, let $H$ be an ideal of $G(I)$. Obviously, if $H = \{0\}$ then $H = G(\emptyset)$. Assume that $H\ne \{0\}$, set $L = \bigcup\{\Supp(x)\mid x\in H^{+}\}$ and let us prove that $H = G(L)$. To this purpose, inasmuch as $H$ is directed it is sufficient to show that $L\sbs H$. Given $x\in H^{+}\setminus\{0\}$, if $x = x_{i}\cdot i$ for some $i\in I$ and $x_{i}\in\BN^{*}$, then obviously $i\in H$ by the convexity of $H$. Given an integer $n>1$, assume inductively that $\Supp(x)\sbs H$ whenever $x\in H^{+}$ and $|\Supp(x)|<n$. Given $x\in H^{+}$ with $|\Supp(x)|=n$, set $K = \{i\in I\mid x_{i}<0\}$ and let us consider the elements
\[
y = \sum_{i\in\Supp(x)\setminus K}x_{i}\cdot i\quad \text{and}\quad z = \sum_{i\in K}(-x_{i})\cdot i,
\]
so that $x = y-z$. Note that $y,z$ and $y-2z$ belong to $M(I)$; since $x = z+y-2z$, we infer that $0\le z<x$ and hence $z\in H^{+}$, because $H$ is convex. Consequently from $y = x+z$ it follows that $y\in H^{+}$ as well. If $K\ne\vu$, then $0<|\Supp(x)\setminus K|<n$ and $|K|<n$; since $\Supp(x)\setminus K = \Supp(y)$ and $K = \Supp(z)$, from the inductive assumption it follows that $\Supp(x)\sbs H$. If $K = \vu$, then $x = y$ and from the convexity of $H$ it is easy to infer that $\Supp(x)\sbs H$ and this completes the proof that $H = G(L)$.

Now, if $i,j\in I$ are such that $i\le j$ and $j\in L$, then $0\le i\le j$ in $H$. By the convexity of $H$ we infer that $i\in H$ and we conclude that $i\in L$, showing that $L$ is a lower subset of $I$.
\end{proof}

Given a ring $R$, we denote by $\FP_{R}$ the class of finitely generated and projective right $R$-modules. Concerning basic notions and notations about the Grothendiek group $\KO(R)$ we follow \cite[Ch 15]{Good:3}; thus, in particular, for a single $A\in\FP_{R}$ we denote by $[A]$ the stable isomorphism class of $A$ and, for a subset $\ZA$ of $\FP_{R}$, we set
\[
[\ZA]\defug \{[A]\mid A\in\ZA\}.
\]
Recall that if $R$ is unit-regular, then $[A]$ is merely the isomorphism class of $A$, for every $A\in\FP_{R}$.

Given $i\in I$, every minimal right ideal of $H(I,i)$ is generated by a primitive idempotent and any two minimal right ideals of $H(I,i)$ are isomorphic. If $\zu$ is a primitive idempotent of $H(I,i)$, then the element $[\zu\mathfrak{B}(I)]$ of $\mathbf{K}_0(\mathfrak{B}(I))$ does not depend on the choice of $\zu$. Indeed, if $\zv$ is another primitive idempotent of $H(I,i)$, since $\zu H(I,i)$ and $\zv H(I,i)$ are isomorphic minimal right ideals of $H(I,i)$, then $\zu\mathfrak{B}(I)\is \zv\mathfrak{B}(I)$, as a consequence of the following very basic lemma (see \cite[Proposition 21.20]{Lam:1}).

\begin{lem}\label{lem idempot}
 Let $R$ be a not necessarily unital ring and let $e,f$ be idempotents of $R$. Then $eR_{R}$ and $fR_{R}$ are isomorphic if and only if there are $a,b\in R$ such that
  \[
  a = fae,\qquad b = ebf,\qquad  ba = e,\qquad ab = f.
  \]
Consequently, if $S$ is a not necessarily unital ring and $\map\et RS$ is any ring homomorphism, in particular if $R$ is a subring of $S$, then every isomorphism from $eR_{R}$ to $fR_{R}$ gives rise to an isomorphism from $\et(e)S_{S}$ to $\et(f)S_{S}$.
 \end{lem}

For every $i\in I$ let us choose a primitive idempotent $\zu_{i}$ of $H(I,i)$ and let us consider the subset
\[
\ZU\defug \{\zu_{i}\mathfrak{B}(I)\mid i\in I\}
\]
of $\FP_{\mathfrak{B}(I)}$. Note that if $i$ is maximal in $I$, the unique choice is $\zu_{i} = \ze_{X(I)_{i}}$. Then, by the above, the subset $[\ZU]$ of $\mathbf{K}_0(\mathfrak{B}(I))$ does not depend on the choice of the $\zu_{i}$'s; as we shall see, the structure of $\mathbf{K}_0(\mathfrak{B}(I))$ is completely determined by $[\ZU]$. In several instances we shall use the fact that, for every $i\in I$ which \emph{is not maximal} and $Y\in\CR_{i}$, we have that $\zu_{i}H(I,i)\is \ze_{Y}H(I,i)$ (see \eqref{eq eYub}) and hence $\zu_{i}\mathfrak{B}(I) \is \ze_{Y}\mathfrak{B}(I)$.

The following lemma is a special instance of \cite[Corollary 2.23]{Good:3}

\begin{lem}\label{lem idempreg}
Let $e, f$ be two idempotents of a regular ring $R$. If $eR\lesssim fR$, then $ReR\sbs RfR$.
 \end{lem}

\begin{pro}\label{lem basis}
 With the above notations, by considering $[\ZU]$ equipped with the partial order induced by that of $\mathbf{K}_0(\mathfrak{B}(I))$, the assignment $i\mapsto [\zu_{i}\mathfrak{B}(I)]$ defines an order isomorphism from $I$ to $[\ZU]$.
 \end{pro}
\begin{proof}
Given $i,j\in I$, assume that $i\le j$. If $j$ is not maximal in $I$, we can take any $Z\in \CR_{j}$ and we have that
$\zu_{j}\mathfrak{B}(I) \is \ze_{Z}\mathfrak{B}(I)$.
According to Lemma \ref{lem idempprim Hi}, (3), we can choose an element $Y\in \CR_{i}$ such that $Y\sbs Z$. As a consequence we have that
\[
\zu_{i}\mathfrak{B}(I) \is \ze_{Y}\mathfrak{B}(I)\sbs \ze_{Z}\mathfrak{B}(I) \is \zu_{j}\mathfrak{B}(I)
\]
and so $[\zu_{i}\mathfrak{B}(I)] \le [\zu_{j}\mathfrak{B}(I)]$. If $j$ is maximal in $I$, the above argument still works by taking $Z = X(I)_{j}$ and $Y = Y_{(u,E)}$, where $u\in\al^{(i\le)}$ and $E$ is any maximal chain of $I$ such that $i,j\in E$, since we have again that $Y \sbs Z$. Conversely, assume that this latter condition holds, meaning that $\zu_{i}\mathfrak{B}(I) \lesssim \zu_{j}\mathfrak{B}(I)$. Then it follows from Lemma \ref{lem idempreg} that
\[
H(I,\{\le i\}) = \mathfrak{B}(I)\zu_{i}\mathfrak{B}(I) \sbs \mathfrak{B}(I)\zu_{j}\mathfrak{B}(I) = H(I,\{\le j\}),
\]
therefore $i\le j$.
\end{proof}

For the purposes of our investigation of the group $\mathbf{K}_0(\mathfrak{B}(I))$, it is necessary to recall some basic facts about regular and semiartinian rings; thus, in what follows we let $R$ be such a ring and $\xi+1$ be its Loewy length. For every right $R$-module $M$ there is a smallest ordinal $h(M)\le \xi+1$ such that $M = M\cdot\Soc_{h(M)}(R)$ and $h(M)$ is a successor ordinal in case $M$ is finitely generated. In particular, if $e = e^{2}\in R$ and $\a$ is the unique ordinal such that $e\in\Soc_{\a+1}(R)\setminus\Soc_{\a}(R)$, then $h(eR) = \a+1$. If $A$ is a finitely generated and projective right $R$-module, then $A\cdot\Soc_{\a}(R) = \Soc_{\a}(A)$ for every $\a\le\xi+1$ (see \cite[Proposition 1.5]{Bac:15}), therefore $A/A\cdot \Soc_{h(A)-1}(R)$ is always a semisimple $R/\Soc_{h(A)-1}(R)$-module; we say that $A$ is \emph{eventually simple} in case $A/A\cdot \Soc_{h(A)-1}(R)$ is simple and we consider the class
\[
\ESFP_{R} \defug \{A\in\FP_{R}\mid\text{$A$ is eventually simple}\}.
\]
We say that a subset $\ZA\sbs\ESFP_{R}$ is \emph{complete and irredundant} if for every simple right $R$-module $U$ there is a \underbar{unique} $A\in\ZA$ such that $A/A\cdot \Soc_{h(A)-1}(R)\is U$. It can be shown that $R$ admits a set $E$ of idempotents such that $\{eR\mid e\in E\}$ is a complete and irredundant subset of $\ESFP_{R}$ (see \cite[Page 70]{BacCiamp:14}).

Let us denote by $\Simp_{R}$ a set of representatives of the isomorphism classes of simple right $R$-modules. It is easily seen that the assignment $U\mapsto r_{R}(U)$ defines a bijection from $\Simp_{R}$ to the set $\Prim_{R}$ of right primitive ideals of $R$; consequently we are enabled to consider the partial ordering $\preccurlyeq$ in $\Simp_{R}$ (which we called in \cite{Bac:15} the \emph{natural partial order}) by prescribing that $U\preccurlyeq V$ exactly when $r_{R}(U)\sbs r_{R}(V)$.
Given $U,V\in\Simp_{R}$, we have that $U\preccurlyeq V$ if and only if $h(U)\le h(V)$ and, for every $A\in\FP_{R}$ such that $A/A\cdot\Soc_{h(V)-1}(R)\is V$, it is true that $U\lesssim A/A\cdot\Soc_{h(U)-1}(R)$ (see \cite[Theorem 2.2]{Bac:15}); moreover it follows from \cite[Theorem 2.1]{Bac:16} (\footnote{\,\,\,We take here an opportunity to remark that a mistake has been left to stand in the final and published version of the quoted paper. Precisely, in the quoted theorem and its proof ten instances appear of a factor module of the form $(***+L_{\b+1})/L_{\b}$, which must be read as $(***+L_{\b})/L_{\b}$ instead, as one can easily infer from the context.}) that $U\prec V$ if and only if $h(U)< h(V)$ and, if $A\in\FP_{R}$ is such that $A/A\cdot\Soc_{h(V)-1}(R)\is V$, then $A/A\cdot\Soc_{h(U)-1}(R)$ contains an infinite direct sum of copies of $U$.

Let $A\in\ESFP_{R}$ and let us consider the simple module $V = A/A\Soc_{h(A)-1}(R)$.  We say that $A$ is \emph{order-selective} (see \cite{BacSpin:17}) if, for every $U\in\Simp_{R}$, the property $U\lesssim A/A\Soc_{h(U)-1}(R)$ is equivalent to $U\preccurlyeq V$. By the above $A$ is order-selective if and only if, given any $U\in\Simp_{R}$ such that $h(U)<h(V)$, the condition $U\lesssim A/A\Soc_{h(U)-1}(R)$ implies that $A/A\Soc_{h(U)-1}(R)$ contains an infinite direct sum of copies of $U$. Finally, we say that a complete and irredundant subset $\ZA\sbs\ESFP_{R}$ is \emph{order representative} if all its elements are order-selective; if it is the case, by \cite[(3.1)]{BacSpin:17} the property $A\lesssim B$ is equivalent to $A/A\cdot\Soc_{h(A)-1}(R)\preccurlyeq B/B\cdot\Soc_{h(B)-1}(R)$ for every $A,B\in \ZA$.

\begin{pro}\label{pro basis}
Let $J$ be an artinian and finitely sheltered subset of $I$ and, with the above notations, let us consider the set $\ZU_{J} \defug \{\zu_{j}H(I,J)\mid j\in J\}$. Then:
\begin{enumerate}
  \item $\ZU_{J}$ is an order representative subset of $\ESFP_{H(I,J)}$;
  \item $\KO(H(I,J))$ is free abelian with $[\ZU_{J}]$ as a basis and
\[
\KO(H(I,J))^{+} = M([\ZU_{J}]).
\]
\end{enumerate}
 \end{pro}
\begin{proof}
(1) We first observe that if $j\in J$ and $\a$ is the unique ordinal such that $j\in J_{\a+1}\setminus J_{\a}$, then $h(\zu_{j}H(I,J)) = \a+1$ and
\[
(\zu_{j}+H(I,J_{\a}))[H(I,J)/H(I,J_{\a})] \is \zu_{j}H(I,J)/\zu_{j}H(I,J_{\a})] \is \zu_{j}H(I,j)
\]
is a simple right $H(I,J)/H(I,J_{\a})$-module, showing that $\zu_{j}H(I,J)$ is eventually simple.

Next, let $V$ be any simple right $H(I,J)$-module and set $\b+1 = h(V)$. Then \[
V = V\cdot H(I,J_{\b+1}) \quad\text{ and }\quad V\cdot H(I,J_{\b}) = 0,
\]
so that $V$ is a simple right module over $H(I,J\setminus J_{\b}) \is H(I,J)/H(I,J_{\b})$ and, as such, we have
\begin{equation}\label{eq complirred}
 V = V\cdot[H(I,J_{\b+1}\setminus J_{\b})] = V\cdot \Soc(H(I,J\setminus J_{\b})).
\end{equation}
 Now the homogeneous components of $\Soc(H(I,J\setminus J_{\b}))$ are precisely the $H(I,k)$ for $k\in J_{\b+1}\setminus J_{\b}$, thus it follows from \eqref{eq complirred} that $V$ is isomorphic to $\zu_{k}H(I,k) = \zu_{k}H(I,J_{\b+1}\setminus J_{\b})$ for a unique $k\in J_{\b+1}\setminus J_{\b}$. Finally, by using Theorem \ref{artin}, Remark \ref{remark split} and Theorem \ref{lemposetring}, (3), we get
 \begin{gather*}
 V \is \zu_{k}H(I,k) = \zu_{k}H(I,J\setminus J_{\b}) \is \zu_{k}H(I,J)/\zu_{k}H(I,J_{\b}) \\
  = \zu_{k}H(I,J)/\zu_{k}\Soc_{\b}(H(I,J)).
 \end{gather*}
We conclude that $\ZU_{J}$ is a complete and irredundant subset of $\ESFP_{H(I,J)}$.

In order to complete the proof of (1), it remains to show that every member of $\ZU_{J}$ is order selective. Thus, given $k\in J$, let $U$ be any simple right $H(I,J)$-module such that $h(U) < h(\zu_{k}H(I,J))$, set $\a+1 = h(U)$, assume that $U \lesssim \zu_{k}H(I,J)/\zu_{k}\Soc_{\a}(H(I,J))$ and let us show that $\zu_{k}H(I,J)/\zu_{k}\Soc_{\a}(H(I,J))$ contains an infinite direct sum of copies of $U$. As we proved above, there is a unique $j\in J$ such that $U\is \zu_{j}H(I,J)/\zu_{j}\Soc_{\a}(H(I,J))$, thus
\[
\zu_{j}H(I,J)/\zu_{j}\Soc_{\a}(H(I,J)) \lesssim \zu_{k}H(I,J)/\zu_{k}\Soc_{\a}(H(I,J)),
\]
that is
\begin{equation}\label{eq complirredd}
\zu_{j}H(I,J)/\zu_{j}H(I,J_{\a}) \lesssim \zu_{k}H(I,J)/\zu_{k}H(I,J_{\a}).
\end{equation}
It follows that $\zu_{j}H(I,J\setminus J_{\a}) \lesssim \zu_{k}H(I,J\setminus J_{\a})$ and, by using Lemma \ref{lem idempreg}, we infer that
\begin{gather*}
H(I,j) = H(I,J\setminus J_{\a})\zu_{j}H(I,J\setminus J_{\a}) \\
 \sbs H(I,J\setminus J_{\a})\zu_{k}H(I,J\setminus J_{\a}) \sbs H(I,\{\le k\}\cap J)
 \end{gather*}
and therefore $j< k$. Suppose that $k$ is not maximal and choose any $Z\in \CR_{k}$; then we have that $\zu_{k}H(I,k)\is \ze_{Z}H(I,k)$ and hence $\zu_{k}H(I,J)\is \ze_{Z}H(I,J)$ by Lemma \ref{lem idempot}. According to Lemma \ref{lem idempprim Hi}, (3), there is a subset $\CS$ of $\CR_{j}$ such that $|\CS| = \al$ and $Y\sbs Z$ for every $Y\in\CS$. Consequently
\begin{equation}\label{eq para}
\bigoplus_{Y\in\CS}\ze_{Y}H(I,J) \sbs \ze_{Z}H(I,J) \is \zu_{k}H(I,J),
\end{equation}
If $k$ is maximal in $J$, we choose a maximal chain $E$ of $I$ such that $j,k\in E$ and take $\CS = \{Y_{(u,E)}\mid u\in\al^{(j\le)}\}\sbs \CR_{j}$. Then again $|\CS| = \al$ and \eqref{eq para} still holds with $\ze_{X(I)_{j}}$ in place of $\ze_{Z}$. As a result we have in both cases that
\begin{gather*}
\bigoplus_{Y\in\CS}(\ze_{Y}H(I,J)/\ze_{Y}H(I,J_{\a})) \lesssim \zu_{k}H(I,J)/\zu_{k}H(I,J_{\a})\\
 \is \zu_{k}H(I,J)/\zu_{k}\Soc_{\a}H(I,J)
\end{gather*}
and we are done, because $U\is \ze_{Y}H(I,J)/\ze_{Y}H(I,J_{\a})$ for every $Y\in\CS$.

(2) is now a consequence of \cite[Theorems 3.6 and 4.3]{BacSpin:17}.
\end{proof}

We are now in a position to state and prove the main theorem of this section.

\begin{theo}\label{theo K0}
Let $I$ be an $\al$-poset having a finite cofinal subset. With the above notations, the group $\mathbf{K}_0(\mathfrak{B}(I))$ is free abelian with $[\ZU]$ as a basis and
\begin{equation}\label{eq poscone}
\KO(\mathfrak{B}(I))^{+} = M([\ZU]).
\end{equation}
Consequently the assignment $i\mapsto [\zu_{i}\mathfrak{B}(I)]$ induces an isomorphism
\[
\lmap{\rho(I)}{G(I)}{\KO(\mathfrak{B}(I))}
\]
of partially ordered abelian groups.
\end{theo}
\begin{proof}
We first prove that $[\ZU]$ is a $\BZ$-independent subset of $\KO(\mathfrak{B}(I))$.
Let $J$ and $K$ be two disjoint and finite subsets of $I$, assume that $(a_{j})_{j\in J}$ and $(b_{k})_{k\in K}$ are families of non negative integers such that
\begin{equation}\label{eq lindep}
\sum_{j\in J}a_{j}[\zu_{j}\mathfrak{B}(I)] =
\sum_{k\in K}b_{k}[\zu_{k}\mathfrak{B}(I)],
\end{equation}
and let us prove that, consequently, $a_{j} = 0 = b_{k}$ for all $j\in J$ and $k\in K$. The statement is trivial if $J\cup K = \vu$ (it cannot be false...). By proceeding inductively on $n =|J\cup K|$, let $n> 0$ and assume that the statement is true whenever $|J\cup K|<n$. Inasmuch as $\mathfrak{B}(I)$ is unit-regular, \eqref{eq lindep} is equivalent to the existence of an isomorphism (\footnote{If $\mathfrak{I}$ is a right ideal of some ring $R$ and $n$ is a natural number, we use here the notation $\mathfrak{I}^{n}$ to mean the direct sum of $n$ copies of $\mathfrak{I}$, not their product!})
\begin{equation}\label{eq lindep1}
\bigoplus_{j\in J}(\zu_{j}\mathfrak{B}(I))^{a_{j}} \is
\bigoplus_{k\in K}(\zu_{k}\mathfrak{B}(I))^{b_{k}}.
\end{equation}
Consequently, if $J = \vu$ (resp. $K = \vu$), then $b_{k} = 0$ for all $k\in K$  (resp. $a_{j} = 0$ for all $j\in J$) and induction works. Thus, assume that $J \ne\vu\ne K$.
If we set $L = J\cup K$, it follows from (2) and (3) of Theorem \ref{lemposetring} that
\[
\zu_{i}\mathfrak{B}(I) = \zu_{i}H(I,\{\le L\})
\]
for every $i\in L$, therefore \eqref{eq lindep1} can be rewritten as
\begin{equation}\label{eq lindep2}
\bigoplus_{j\in J}(\zu_{i}H(I,\{\le L\}))^{a_{j}} \is
\bigoplus_{k\in K}(\zu_{k}H(I,\{\le L\}))^{b_{k}}.
\end{equation}
Let $L'$ be the set of all maximal elements of $L$ and note that $\{<L'\}$ is a lower subset of $\{\le L\}$, therefore $H(I,\{<L'\})$ is an ideal of the ring $H(I,\{\le L\})$ and we have a $\BK$-algebra isomorphism
\[
H(I,\{\le L\})/H(I,\{<L'\})\is H(I,L')
\]
(see Remark \ref{remark split}). After tensoring (over $H(I,\{\le L\})$) by
\[
H(I,\{\le L\})/H(I,\{<L'\}),
\]
the isomorphism \eqref{eq lindep2} gives rise to an isomorphism
\begin{equation}\label{eq lindep3}
\bigoplus_{j\in J\cap L'}(\zu_{j}H(I,L'))^{a_{j}} \is
\bigoplus_{k\in K\cap L'}(\zu_{k}H(I,L'))^{b_{k}}.
\end{equation}
Since $H(I,L') = \bigoplus_{i\in L'}H(I,i)$ is a (not necessarily artinian) semisimple ring ha\-ving each $H(I,i)$ as a homogeneous component, then the $\zu_{i}H(I,i)$ are pairwise non-isomorphic minimal right ideals of $H(I,L')$ for $i$ ranging in $L'$, thus \eqref{eq lindep3} entails that $a_{j} = 0 = b_{k}$ for all $j\in J\cap L'$ and $k\in K\cap L'$. As a result we have in $\mathbf{K}_0(\mathfrak{B}(I))$ the equality
\begin{equation}\label{eq lindep4}
\sum_{j\in J\setminus L'}a_{j}[\zu_{j}\mathfrak{B}(I)] =
\sum_{k\in K\setminus L'}b_{k}[\zu_{j}\mathfrak{B}(I)]
\end{equation}
and induction applies, because $|(J\setminus L')\cup(K\setminus L')|<n$.

We have so far shown that $[\ZU]$ is a $\BZ$-linearly independent subset of $\mathbf{K}_0(\mathfrak{B}(I))$. Let $\CF(I)$ be the set of all finite subsets of $I$ which contain all maximal elements of $I$ and note that $H(I,J)$ is a unital subring of $\mathfrak{B}(I)$ for every $J\in \CF(I)$. Inasmuch as $\CF(I)$ is directed under the ordering by inclusion and $\mathfrak{B}(I))$ is the union of its unital subrings $H(I,J)$, for $J\in\CF(I)$, one might conclude the proof by using the fact that the functor $\mathbf{K}_0(-)$ preserves direct limits (see \cite[Proposition 15.11]{Good:3}), together with Proposition \ref{pro basis}. Nonetheless, we wish to display here a direct and relatively fast argument. Thus, let us denote by $\map{\mu_{J}}{H(I,J)}{\mathfrak{B}(I)}$ the inclusion map for every $J\in \CF(I)$ and let us first show that
\begin{equation}\label{eq union}
  \mathbf{K}_0(\mathfrak{B}(I)) = \bigcup_{J\in \CF(I)}\mathbf{K}_0(\mu_{J})(H(I,J)).
\end{equation}
Remember that the group $\mathbf{K}_0(\mathfrak{B}(I))$ is generated by the elements of the form $[\ze\mathfrak{B}(I)]$, where $\ze$ is an idempotent of $\mathfrak{B}(I)$, because $\mathfrak{B}(I)$ is a regular ring. Now, if $\ze$ is any idempotent of $\mathfrak{B}(I)$, since
\[
\mathfrak{B}(I) = \bigcup_{J\in \CF(I)}H(I,J)
\]
we can choose some $J\in \CF(I)$ such that $\ze\in H(I,J)$. We have
\[
[\ze\mathfrak{B}(I)] = [\ze H(I,J)\otimes_{H(I,J)}\mathfrak{B}(I)] = \mathbf{K}_0(\mu_{J})([\ze H(I,J)]),
\]
consequently \eqref{eq union} immediately follows.

Given $J\in \CF(I)$, by Proposition \ref{pro basis} the group $\mathbf{K}_0(H(I,J))$ is free abelian with the set $\ZU_{J} = \{[\zu_{j}H(I,J)]\mid j\in J\}$ as a basis. Thus \eqref{eq union} enables us to conclude that
\[
  [\ZU] = \bigcup_{J\in\CF(I)}\{[\zu_{i}\mathfrak{B}(I)]\mid i\in J\} = \bigcup_{J\in\CF(I)}
\{\mathbf{K}_0(\mu_{J})([\zu_{i}H(I,J)])\mid i\in J\}
\]
generates $\mathbf{K}_0(\mathfrak{B}(I))$ and hence is a basis.

Next, let $\mathbf{0} \ne \ze = \ze^{2}\in\mathfrak{B}(I)$ and choose some $J\in\CF(I)$ such that $\ze\in H(I,J)$. By Proposition \ref{pro basis} there is a family $(a_{j})_{j\in J}$ of integers such that
\[
[\ze H(I,J)] = \sum_{j\in J}a_{j}[\zu_{j} H(I,J)]
\]
in $\mathbf{K}_0(H(I,J))$ and $a_{j}>0$ whenever $j$ is a maximal element of $J$. Then in $\mathbf{K}_0(\mathfrak{B}(I))$ we have
\[
[\ze\mathfrak{B}(I)] = \mu_{J}([\ze H(I,J)])
= \sum_{j\in J}a_{j}\mu_{J}([\zu_{j} H(I,J)])
= \sum_{j\in J}a_{j}[\zu_{j}\mathfrak{B}(I)].
\]
This is enough to conclude that the first member of \eqref{eq poscone} is contained in the second. Conversely, let $(a_{i})_{i\in I}$ be any family of integers having finite support $F$, such that $a_{i}>0$ whenever $i$ is a maximal element of $F$, and choose any $J\in\CF(I)$ such that $F\sbs J$. By Proposition \ref{pro basis} there is some $B\in\FP_{H(I,J)}$ such that
\[
[B] = \sum_{j\in J}a_{j}[\zu_{j} H(I,J)].
\]
Consequently we have
\begin{gather*}
\sum_{i\in I}a_{i}([\zu_{j}\mathfrak{B}(I)])
= \sum_{j\in J}a_{j}\mathbf{K}_0(\mu_{J})([\zu_{j} H(I,J)])
= \mathbf{K}_0(\mu_{J})([B]) \\
= [B\otimes_{H(I,J)}\mathfrak{B}(I)].
\end{gather*}
This shows the reverse inclusion in \eqref{eq poscone} and hence the equality. The last statement follows from Proposition \ref{lem basis}.
\end{proof}

\section{Primeness versus primitivity. The primitive spectrum.}

In this section we investigate which properties of the $\al$-poset $I$ are related with the primeness and the primitivity of the algebra $\mathfrak{B}(I)$. As we shall see, while if $|I|=\al_{0}$ primeness of $\mathfrak{B}(I)$ always implies primitivity, for every cardinal $\al>\al_{0}$ there exists some $\al$-poset $I$ such that $\mathfrak{B}(I)$ is prime but not primitive. As we know (see Remark \ref{remark split}), every factor ring of $\mathfrak{B}(I)$ is isomorphic to $H(I,J)$ for a unique upper subset $J$ of $I$; the main result of this section is that the ring $H(I,J)$ is prime if and only if $J$ is downward directed, while the existence of a coinitial chain in $J$ is necessary and sufficient for $H(I,J)$ to be primitive.

At any rate, based on the results so far obtained, we can exhibit easily here a first special set of primitive factor rings of $\mathfrak{B}(I)$, namely those of the form $H(I,i\le)$. In fact, given $i\in I$, we have from Theorem \ref{artin} that
\[
\Soc(H(I,i\le)) = H(I,i) \ne \mathbf{0}
\]
is homogeneous. If $\mathfrak{I}$ is any nonzero ideal of $H(I,i\le)$ and $L$ is the unique lower subset of $\{i\le\}$ such that $\mathfrak{I} = H(I,L)$, it follows from Theorem \ref{lemposetring}, (3) that $H(I,i)H(I,l) \ne \mathbf{0} \ne H(I,l)H(I,i)$ for every $l\in L$, meaning that $H(I,i)\mathfrak{I} \ne \mathbf{0} \ne \mathfrak{I}H(I,i)$. This shows that $H(I,i\le)$ is a primitive ring and so $H(I,i\nle)$ is a primitive ideal of $\mathfrak{B}(I)$. As we shall see, those of the form $H(I,i\nle)$, for $i\in I$, are the only primitive ideals of $\mathfrak{B}(I)$ exactly when $I$ is artinian.

In order to reach our goal we rely on the following three results, the first of which is a characterization of right primitive rings which was already used in \cite{AbramsBellRangaswamy:011} and seems to have been observed for the first time by Formanek in \cite{Formanek:1} (see also \cite[Lemma 11.28]{Lam:1}).

\begin{pro}\label{pro primitive1}
      A ring $R$ is right primitive if and only if there is a proper right ideal $\mathfrak{I}$ such that $\mathfrak{I}+\mathfrak{A} = R$ for every ideal $\mathfrak{A}\ne 0$.
    \end{pro}

The second is a generalization of \cite[Proposition 3.4]{AbramsBellRangaswamy:011}; this latter was used by Abrams, Bell and Rangaswamy in order to establish sufficient conditions on a graph $E$ for its Leavitt algebra $L(E)$ (over some field) to be primitive. In a similar manner, our generalization will enable us to state that the existence of a coinitial chain in $J$ is a sufficient condition for $H(I,J)$ to be primitive.

\begin{pro}\label{pro primitive2}
      Let $R$ be a unital ring, let $A$ be any totally ordered set and assume that there is a family $(e_{i})_{i\in A}$ of nonzero idempotents of $R$ such that, given $i,j\in A$ with $i\le j$, then $e_{i}e_{j} = e_{i}$ (resp. $e_{j}e_{i} = e_{i}$). If every nonzero ideal of $R$ contains some $e_{i}$, then $R$ is right (resp. left) primitive.
    \end{pro}
    \begin{proof}
The proof is totally the same as that of \cite[Proposition 3.4]{AbramsBellRangaswamy:011}, however we repeat it for completeness. Concerning right primitivity, it follows from the assumption that if $i,j\in A$ and $i\le j$, then $Re_{i}\sbs Re_{j}$ and so $(1-e_{i})R\sps(1-e_{j})R$, therefore
\[
    \{(1-e_{i})R\mid i\in A\}
    \]
is a chain of proper right ideals of $R$ and so $\mathfrak{I} = \bigcup\{(1-e_{i})R\mid i\in A\}$ is a proper right ideal of $R$. Let $\mathfrak{A}$ be any nonzero ideal of $R$ and let $i\in A$ be such that $e_{i}\in \mathfrak{A}$. Then
\[
1 = e_{i} + (1-e_{i}) \in \mathfrak{A}+\mathfrak{I}
\]
and therefore $\mathfrak{A}+\mathfrak{I} = R$. We conclude from Proposition \ref{pro primitive1} that $R$ is right primitive. The proof relative to left primitivity is just specular.
\end{proof}

Finally, the following lemma is the key to establish that the existence of a coinitial chain in $J$ is a necessary condition for $H(I,J)$ to be primitive.

\begin{lem}\label{lem coinchain}
      Given a partially ordered set $I$, the following conditions are equivalent:
     \begin{enumerate}
       \item $I$ has a coinitial chain;
       \item There is a family $(F_{i})_{i\in I}$ of finite subsets of $I$ such that $F_{i}\sbs\{\le i\}$ for every $i\in I$ and, given $i,j\in I$, at least one element of $F_{i}$ is comparable with some element of $F_{j}$.
     \end{enumerate}
    \end{lem}
    \begin{proof}
(1)$\Rightarrow$(2) If $A$ is a coinitial chain of $I$, for every $i\in I$ choose $k_{i}\in A$ such that $k_{i}\le i$ and define $F_{i} = \{k_{i}\}$. Clearly the family $(F_{i})_{i\in I}$ meets the requirements stated in (2).

(2)$\Rightarrow$(1) Assume (2) and observe that, consequently, $I$ is downward directed. Suppose that no chain of $I$ is coinitial; of course, we may assume that $I$ is not empty, otherwise there is nothing to prove. As a first consequence $I$ must be infinite and has no minimal element, otherwise $I$ would admit a smallest element $m$ and $\{m\}$ would be a coinitial chain. We claim that, for every ordinal $\l$, there exists a family $(i_{\a})_{\a<\l}$ of elements of $I$ such that if $\b<\a<\l$, then  $i_{\b}>i_{\a}$ and $i_{\a}$ is a lower bound for $F_{i_{\b}}$. Thus $\{i_{\a}\mid\a<\l\}$ would be a subset of $I$ having the same cardinality of $\l$ and this leads to a contradiction as soon as the cardinality of $\l$ exceeds that of $I$. By proceeding recursively on $\a<\l$, let us choose any $i_{0}\in I$. Suppose that $0<\a<\l$ and assume that we have defined $i_{\b}$ for every $\b<\a$, in such a way that if $\g<\b<\a$, then $i_{\b}<i_{\g}$ and $i_{\b}$ is a lower bound for $F_{i_{\g}}$. If $\a = \b+1$ for some $\b$, we choose a lower bound $j$ for $F_{i_{\b}}\cup\{i_{\b}\}$ and, since $j$ cannot be a minimal element, we choose $i_{\a}<j$, so that $i_{\b}>i_{\a}$ and $i_{\a}$ is a lower bound for $F_{i_{\b}}$. As a consequence, if $\g<\b$ then $i_{\a}$ is a lower bound for $F_{i_{\g}}$ too. Suppose that $\a$ is a limit ordinal and let us consider the chain $C = \{i_{\b}\mid \b<\a\}$. Since $C$ is not coinitial, there is some $h\in I$ such that $h\nin\{C\le\}$. Let $k$ be a lower bound for $F_{h}$ and let us prove that $k<i_{\b}$ whenever $\b<\a$. Given $\b<\a$, by the assumption (2) there are two elements $u\in F_{i_{\b}}$ and $v\in F_{h}$ which are comparable. It is not the case that $u\le v$ otherwise, by the inductive assumption, $i_{\b+1}\le u\le v\le h$ and hence $h\in\{C\le\}$, because $i_{\b+1}\in C$. Thus $v<u$ and so $k\le v<u\le i_{\b}$, as wanted. Finally, let us define define $i_{\a}\defug k$; by the above and the inductive assumption, $i_{\a}$ is a lower bound for $F_{i_{\d}}$ for all $\d<\a$. Thus our claim is true and the proof is complete.
\end{proof}

\begin{theo}\label{theo gprimprimitive}
Let $I$ be a nonempty $\al$-poset having a finite cofinal subset and let $J$ be an upper subset of $I$. Then we have:
     \begin{enumerate}
       \item $H(I,J)$ is prime if and only if $J$ is downward directed;
       \item $H(I,J)$ is right primitive if and only if $J$ has a coinitial chain, if and only if $H(I,J)$ is left primitive.
     \end{enumerate}
\end{theo}
\begin{proof}
First note that $H(I,J)$ has a multiplicative identity $\zu$, by the assumption and Proposition \ref{multunit}.

(1) According to Theorem \ref{theo ideals} the ring $H(I,J)$ is prime if and only if\linebreak $H(I,\{\le i\})H(I,\{\le j\}) \ne 0$ for all $i,j\in I$. In view of Theorem \ref{lemposetring} the latter condition is equivalent to requiring that, given $i,j\in I$, there are two comparable elements $h,k\in I$ such that $h\le i$ and $k\le j$; this exactly means that $\{i,j\}$ has a lower bound.

(2) Let $C$ be a coinitial chain of $J$; by the Hausdorff Maximal Principle we may assume that $C$ is a maximal chain of $J$. For every $i\in C$ let us denote by $\vu_{i}$ the constant map $k\mapsto\vu$ from $\{i\le\}$ to $\al$ and let us consider the element
\[
Y_{(\vu_{i},C\cap\{i\le\})} = \left\{(p,B)\in X(I)\left|\,\, p|_{\{i\le\}}\right. =\vu_{i} \text{ and }B\cap\{i\le\} = C\cap\{i\le\}\right\} \in \CR_{i}
\]
and set $\ze_{i} = \ze_{Y_{(\vu_{i},C\cap\{i\le\})}}\in H(I,i)\sbs H(I,J)$. If $i,j\in C$ and $i<j$, then $Y_{(\vu_{i},C\cap\{i\le\})}\sbs Y_{(\vu_{j},C\cap\{j\le\})}$ and hence
\begin{equation}\label{eq idemprim}
\ze_{i}\ze_{j} = \ze_{i} = \ze_{j}\ze_{i};
\end{equation}
 moreover $\ze_{i} \ne \mathbf{0}$ for every $i\in C$, because $\vu\nin\CR_{i}$. We have now the family $(\ze_{i})_{i\in C}$ of idempotents of the unital ring $H(I,J)$, therefore right primitivity of $H(I,J)$ will follow from Proposition \ref{pro primitive2} once we prove that every nonzero ideal of $H(I,J)$ contains $\ze_{i}$ for some $i\in C$. Given a nonzero ideal $\mathfrak{A}$ of $H(I,J)$, it follows from Theorem \ref{theo ideals} that $\mathfrak{A} = H(I,L)$ for some nonempty lower subset $L$ of $J$. Given any $l\in L$, by the assumption there is some $i\in C$ such that $i\le l$. Then necessarily $i\in L$, therefore $\ze_{i}\in H(I,L)\sbs\mathfrak{A}$ and we are done. Clearly, since \eqref{eq idemprim} holds whenever $i,j\in C$ and $i<j$, then $H(I,J)$ is left primitive as well.

Conversely, assume that $H(I,J)$ is right primitive. By Proposition \ref{pro primitive1} there is a proper right ideal $\mathfrak{I}$ of $H(I,J)$ such that $\mathfrak{I}+\mathfrak{A} = H(I,J)$ for every nonzero ideal $\mathfrak{A}$ of $H(I,J)$. In particular, in view of Theorem \ref{theo ideals} we have that $\mathfrak{I}+H(I,\{\le i\}) = H(I,J)$ for every $i\in J$. As a result, for every $i\in J$ we can choose a nonzero element $\za_{i}\in H(I,\{\le i\})$ such that $\zu-\za_{i}\in \mathfrak{I}$ and there is a subset $F_{i} = \{h_{i1},\ldots,h_{in_{i}}\}$ of $\{\le i\}$ and nonzero elements $\zb_{i1}\in H(I,h_{i1}),\ldots,\zb_{in_{i}}\in H(I,h_{in_{i}})$ such that
\[
\za_{i} = \zb_{i1}+\cdots+\zb_{in_{i}}.
\]
The family $(F_{i})_{i\in J}$ satisfies the condition stated in (2) of Lemma \ref{lem coinchain}. To see this, observe that if $i,j\in J$, then necessarily $\za_{i}\za_{j}\ne \mathbf{0}$, otherwise we would get
\[
\za_{j} = (\zu-\za_{i})\za_{j} \in \mathfrak{I}
\]
and hence $\zu\in \mathfrak{I}$. Consequently $\zb_{ih}\zb_{jk}\ne \mathbf{0}$ for some $h\in F_{i}$ and $k\in F_{j}$ and so $H(I,h)H(I,k)\ne 0$; by Theorem \ref{lemposetring} this means that $h$ and $k$ are comparable. Now it follows from Lemma \ref{lem coinchain} that $J$ admits a coinitial chain.

Finally, it is clear that a completely specular argument applies by assuming that $H(I,J)$ is left primitive.
\end{proof}

As a particular case we get the following

\begin{cor}\label{theo primprimitive}
Let $I$ be a nonempty $\al$-poset having a finite cofinal subset. Then we have:
     \begin{enumerate}
       \item $\mathfrak{B}(I)$ is prime if and only if $I$ is downward directed;
       \item $\mathfrak{B}(I)$ is right primitive if and only if $I$ has a coinitial chain, if and only if $\mathfrak{B}(I)$ is left primitive.
     \end{enumerate}
\end{cor}

If $|I| = \al_{0}$ and $I$ is downward directed, then it is easy to see that $I$ admits a coinitial chain. It seems worth to emphasize that the condition $|I| = \al_{0}$ does not make $I$ an $\al_{0}$-poset, because $I$ might have uncountably many maximal chains. As an example, let $I = A\cup B$, where $A$ and $B$ are disjoint countable sets, together with two bijections $h\mapsto a_{h}$ and
$h\mapsto b_{h}$ from $\BZ$ to $A$ and $B$ respectively. Then the binary relation
\[
\{(a_{h},a_{k})\mid h\le k\} \cup\{(b_{h},b_{k})\mid h\le k\}\cup
\{(a_{h},b_{k})\mid h< k\} \cup\{(b_{h},a_{k})\mid h<k\}
 \]
 is a partial order in $I$ with the following Hasse diagram:
 
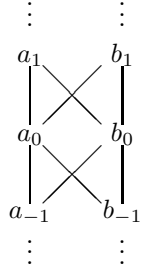
\begin{equation*}
\begin{split}
\begin{picture}(100,110)(0,0)
\putl{25}{80}{0}{-1}{22}
\putl{25}{50}{0}{-1}{22}
\putl{60}{80}{0}{-1}{22}
\putl{60}{50}{0}{-1}{22}
\putbb{25}{95}{$\vdots$}  \putbb{60}{95}{$\vdots$}
\putbb{25}{80}{$a_1$} \putbb{60}{80}{$b_1$}
\putbb{25}{50}{$a_0$} \putbb{60}{50}{$b_0$}
\putbb{25}{20}{$a_{-1}$} \putbb{60}{20}{$b_{-1}$}
\putbb{25}{5}{$\vdots$}  \putbb{60}{5}{$\vdots$}
\putl{30}{50}{1}{-1}{22} \putl{52}{50}{-1}{-1}{22}
\putl{30}{80}{1}{-1}{22} \putl{52}{80}{-1}{-1}{22}
\end{picture}
\end{split}
\end{equation*}
If $X$ is any subset of $\BZ$, then
\[
\{a_{h}\mid h\in X\}\cup \{b_{h}\mid h\in \BZ\setminus X\}
\]
is a maximal chain of $I$ and every maximal chain of $I$ arises in this way. Thus $I$ is not an $\al_{0}$-poset, but it is a $2^{\al_{0}}$-poset.

In any case, if $|I| = \al_{0}$, then $I$ is an $\al$-poset for every cardinal $\al\ge 2^{\al_{0}}$ and $\mathfrak{B}(I)$ is prime if and only if it is primitive.

Now, let $X$ be a set with $\al = |X| >\al_{0}$ and let $I$ be the lattice of all cofinite subsets of $X$ partially ordered by inclusion. Then $I$ is downward directed and admits $X$ as the greatest element. Since every chain in $I$ is at most countable, the set of all maximal chains of $I$ has cardinality $\al$ and so $I$ is an $\al$-poset. It is not difficult to see that no chain of $I$ is coinitial, therefore $\mathfrak{B}(I)$ is prime but it is not primitive.

We know from Theorems \ref{theo ideals} that each ideal of $\mathfrak{B}(I)$ is of the form $H(I,L)$ for a unique lower subset $L$ of $I$ and $\mathfrak{B}(I)/H(I,L)$ is in a natural way isomorphic to $H(I,I\setminus L)$ (see Remark \ref{remark split}). As a consequence of Theorem \ref{theo gprimprimitive} we have the following immediate description of the primitive spectrum $\Prim_{\mathfrak{B}(I)}$ of $\mathfrak{B}(I)$, which is at the same time right and left.

\begin{pro}\label{pro primitid}
The primitive ideals of $\mathfrak{B}(I)$ are those of the form $H(I,\{A\nleqslant\})$ for some chain $A$ of $I$.
\end{pro}

In other words the assignment $A\mapsto H(I,\{A\nleqslant\})$ defines a surjective map from the set of all chains of $I$ to $\Prim_{\mathfrak{B}(I)}$; clearly, it is also injective if and only if $I$ is an antichain. Among the chains in $I$ we have the singletons and we have seen directly, at the beginning of the present section, that each $i\in I$ gives rise to the primitive ideal $H(I,i\nle)$. Thus we may consider the map
\[
\lmap{\Psi}{I}{\Prim_{\mathfrak{B}(I)}}
\]
defined by
\[
\Psi(i) = H(I,\{i\nle\})
\]
for all $i\in I$. Of course  $\Psi(i)\sbs\Psi(j)$ if and only if $i\le j$, so that $\Psi$ is an order imbedding. With the next result we describe the remarkable instance in which $\Psi$ is an order isomorphism.

\begin{pro}\label{pro artinian}
With the above notations, the following conditions are equivalent:
\begin{enumerate}
  \item $\Psi$ is an order isomorphism;
  \item $I$ is artinian;
  \item $\mathfrak{B}(I)$ is semiartinian.
\end{enumerate}
\end{pro}
\begin{proof}
(1)$\Rightarrow$(2) If $\Psi$ is bijective and $C$ is a non-empty chain of $I$, then the ideal $H(I,C\nle)$ is primitive and so there is some $i\in I$ such that $H(I,C\nle) = \Psi(i) = H(I,i\nle)$. Consequently $\{C\nle\} = \{i\nle\}$ and therefore $i$ is the smallest element of $C$. This shows that $I$ is artinian.

(2)$\Rightarrow$(3): see Theorem \ref{artin}.

(3)$\Rightarrow$(1) Assume that $\mathfrak{B}(I)$ is semiartinian, let $\mathfrak{P}$ be a primitive ideal of $\mathfrak{B}(I)$ and let $L$ be the unique lower subset of $I$ such that $\mathfrak{P} = H(I,L)$. We have that $H(I,I\setminus L) \is \mathfrak{B}(I)/\mathfrak{P}$ is primitive and semiartinian, therefore
\[
\Soc(H(I,I\setminus L)) = H(I,(I\setminus L)_{1})
\]
(see Theorem \ref{artin}) is homogeneous and is a minimal ideal of $H(I,I\setminus L)$. As a consequence $(I\setminus L)_{1} = \{m\}$ for some $m\in I\setminus L$ and $m$ must be the smallest element of $I\setminus L$. We conclude that $I\setminus L = \{m\le\}$ and so $L = \{m\nle\}$, proving that $\mathfrak{P} = \Psi(m)$.
\end{proof}

\section{Making $G(-)$ and $\mathfrak{B}(-)$ functors: the category $\Pos_{\al}$. }\label{sect functalgebraB}

In this last section we show that each of the assignments $I\mapsto G(I)$ and ${I\mapsto \mathfrak{B}(I)}$ admits a covariant and a contravariant functorial extension to appropriate categories of partially ordered sets. Precisely, if we consider the category $\Pos$  whose objects are all posets, while the morphisms are all maps which are order imbeddings (\footnote{\,\,\,We recall that if $I$, $J$ are partially ordered sets, a map $\map{f}IJ$ is an \emph{(order) imbedding} if, for every $i,i'\in I$, the properties $i\le i'$ and $f(i)\le f(i')$ are equivalent; if it is the case, $f$ is obviously injective.}) of the domain as an upper subset of the codomain, then the assignment $I\mapsto G(I)$ extends to a pair of functors $G(-)$, $G^{*}(-)$, the first covariant and the second contravariant, from $\Pos$ to the category $\Poab$ whose objects are all partially ordered abelian groups and morphisms are all order preserving group homomorphisms.
Next we show that, given an infinite cardinal $\al$ and a field $\BK$, the assignment $I\mapsto \mathfrak{B}(I)$ extends to a pair of functors $\mathfrak{B}(-)$, $\mathfrak{B}^{*}(-)$, the first covariant and the second contravariant, from the subcategory $\Pos_{\al}$ of $\Pos$ of all $\al$-posets \emph{having a finite cofinal subset} to the category $\Alg_{\BK}$ of \emph{unital} $\BK$-algebras and \emph{not necessarily unital} $\BK$-algebra homomorphisms. Finally, we will see that these four functors match with the functor $\KO(-)$, in the sense that there are natural equivalences $\rho(-)\colon G(-)\approx \KO(-)\circ \mathfrak{B}(-)$ and $\rho^{*}(-)\colon G^{*}(-)\approx \KO(-)\circ \mathfrak{B}^{*}(-)$.

Let $I$ and $J$ be partially ordered sets and let $\map{f}IJ$ be an isotone map. As $G(I)$ and $G(J)$ are the free abelian groups generated by $I$ and $J$ respectively, the most ``natural'' group homomorphism $G(f):G(I)\to G(J)$ to be assigned to $f$ is clearly the unique homomorphism which extends $f$ to $G(I)$, that is
\[
G(f)\left(\sum_{i\in I}x_{i}i\right) = \sum_{i\in I}x_{i}f(i).
\]
Of course we have that $G(1_{I}) = 1_{G(I)}$ and, if $g$ is a second isotone map and $fg$ is defined, then $G(fg) = G(f)G(g)$.
We would like $G(f)$ to be isotone, that is
\begin{equation}\label{eq isotone}
    G(f)(M(I))\sbs M(J).
\end{equation}

However this may fail if $f$ is not injective. For example, suppose that $I = \{h,k\}$, where $h<k$, and $J=\{j\}$; then the unique map $f:I\to J$ is obviously isotone. If we consider the element $x = -2h+k\in M(I)$, we have that $G(f)(x) = -j\nin M(J)$. Fortunately \eqref{eq isotone} holds if $f$ is injective. In fact, assume that $x\in M(I)$ and let $m$ be a maximal element of $\Supp(G(f)(x)) = f(\Supp(x))$. Then the unique element ${l\in M(I)}$ such that $m = f(l)$ is maximal in $\Supp(x)$. In fact, if $i\in\Supp(x)$ and ${l\le i}$, then $m =f(l)\le f(i)$ and $f(i)\in\Supp(G(f)(x))$, therefore $m = f(i)$ and hence $l = i$. Consequently $G(f)(x)_{m} = x_{l}>0$ and this proves that $G(f)(x)\in M(J)$.

We are now allowed to speak about the covariant functor $G(-)$ from the category of all partially ordered sets and \emph{injective isotone maps} to the category $\Poab$.

Passing to consider the contravariant side, let again $I$ and $J$ be partially ordered sets and $\map fIJ$ an isotone map. Again, in order to assign to $f$ an order preserving group homomorphism from $G(J)$ to $G(I)$ we have at our disposal a ``natural'' path to follow. In fact, by identifying $G(J)$ as the subgroup $\BZ^{(J)}$ of $\BZ^{J}$ of all maps from $J$ to $\BZ$ having finite support, we see that $f$ induces the homomorphism of groups
\[
 \lmap{f^{*}}{\BZ^{(J)}}{\BZ^{I}}
 \]
defined by $f^{*}(x) = x\circ f$ for all $x\in G(J)$, that is $(f^{*}(x))_{i} = x_{f(i)}$ for all $i\in I$. Given $x\in G(J)$, we have that
\[
\Supp(f^{*}(x)) = f^{-1}(\Supp(x)),
\]
therefore $f^{*}(x)\in G(I)$ for all $x\in G(J)$ if and only if $f^{-1}(j)$ is finite for each $j\in J$. But, even if this latter condition holds, if $f$ is not injective the homomorphism $f^{*}$ may fail to be an isotone map, as the following example shows.

\begin{ex}\label{ex contraiso}
{\rm Let $I = \{i,j,k\}$ and $J = \{u,v\}$ with $i<j, i<k, j\nleqslant k, k\nleqslant j$ and $u<v$; next define $\map fIJ$ with $f(i) = f(j) = u$ and $f(k) = v$. If we consider $x = v-u\in M(J)$, then we see that $f^{*}(x) = -i-j+k\nin M(I)$, because $j$ is maximal in $I$, therefore $f^{*}$ is not isotone.}
\end{ex}

Thus, also for the contravariant case we must restrict our attention to the injective isotone maps and, as we are going to see with the next result, only part of them are allowed. We observe that if $\map fIJ$ is injective, then $f^{*}$ is an epimorphism, as it is the projection of $\BZ^{(J)}$ onto $\BZ^{(f(I))}\is \BZ^{(I)}$ with kernel $\BZ^{(J\setminus f(I))}$. Moreover, for every $i\in I$ and $j\in J$ we have that
\[
(f^{*}(j))_{i} = \begin{cases} 1, \text{ if $j = f(i)$;} \\
                                0, \text{ otherwise.}
                                \end{cases}
\]
Consequently:
\begin{equation}\label{eq imbedding}
\text{for every $i\in I$ and $j\in J$,}\qquad
f^{*}(j) = \begin{cases} i, \text{ if $j = f(i)$;} \\
                                0, \text{ otherwise.}
                                \end{cases}
\end{equation}

\begin{pro}\label{pro homohahn}
Let $I,J$ be partially ordered sets and let $\map fIJ$ be an injective and isotone map. Then the homomorphism $f^{*}:G(J)\to G(I)$ is isotone, namely
\begin{equation}\label{eq homohahn}
f^{*}(M(J))\sbs M(I),
\end{equation}
if and only if $f$ is an order imbedding and $f(I)$ is an upper subset of $J$.
\end{pro}
\begin{proof}
Suppose that \eqref{eq homohahn} holds and let $i,j\in I$ be such that $f(i)\le f(j)$. Then $f(j)-f(i)\in M(J)$ and consequently, by using \eqref{eq imbedding}, we see that $j-i = f^{*}(f(j)- f(i))\in M(I)$. As a result $i\le j$ and this proves that $f$ is an imbedding. Next, let $i\in I$, $j\in J$ and assume that $f(i)\le j$ but $j\nin f(I)$. Then $-i = f^{*}(j-f(i))\in M(I)$ again by \eqref{eq imbedding}, which is impossible. Thus necessarily $j\in f(I)$ and so $f(I)$ is an upper subset of $J$.

Conversely, assume that  $f$ is an imbedding and $f(I)$ is an upper subset of $J$. Given $x\in M(J)$, $x\ne 0$, and a maximal element $m$ of $\Supp(f^{*}(x))$, we claim that $(f^{*}(x))_{m}>0$. Suppose, on the contrary, that $(f^{*}(x))_{m}\le 0$. Then $x_{f(m)} = (f^{*}(x))_{m}\le 0$ and so $f(m)$ is a non-maximal element of $\Supp(x)$. Let $j\in \Supp(x)$ be such that $f(m)<j$; then $j\in f(I)$ by the assumption. If we let $i\in I$ be such that $f(i) = j$, again from the assumption we have that $m<i$. But $(f^{*}(x))_{i} = x_{f(i)} = x_{j}\ne 0$, hence $i\in\Supp(f^{*}(x))$ and this contradicts the maximality of $m$ in $\Supp(f^{*}(x))$. Thus $(f^{*}(x))_{m}>0$ and therefore $f^{*}(x)\in M(I)$.
\end{proof}

We are led to consider the category whose objects are all partially ordered sets and morphisms are those maps which are order imbeddings of the domain as an upper subset of the codomain, which we denote by $\Pos$. It is obvious that $1_{I}^{*} = 1_{G(I)}$ for every object $I$ of $\Pos$, while if $\map fIJ$ $\map gJK$ are morphisms in $\Pos$, then $(gf)^{*} = f^{*}g^{*}$. Thus we can legitimately define the contravariant functor $G^{*}(-)$ from the category $\Pos$ to the category $\Poab$.

The condition of being an order imbedding of the domain as an upper subset of the codomain, for isotone maps between partially ordered sets, is very close to a known condition on morphisms between directed graphs; we feel it is worth to enlighten that connection.

 We recall that a \emph{(directed) graph} is an ordered 4-tuple
\[
E = (E^{0},E^{1},s_{E},r_{E}),
\]
 where $E^{0},E^{1}$ are two disjoint sets and
\[
\map{s_{E}}{E^{1}}{E^{0}}, \qquad \map{r_{E}}{E^{1}}{E^{0}}
\]
are maps. The elements of $E^{0}$ are called the \emph{vertices} and those of $E^{1}$ the \emph{arrows} (or \emph{edges}) of $E$. A \emph{graph morphism} from a graph $E$ to a graph $F$ is an ordered pair $f = (f^{0}, f^{1})$ consisting of two maps $f^{0}: E^{0}\to F^{0}$ and $f^{1}: E^{1}\to F^{1}$ such that
$s_Ff^{1} = f^{0}s_E$ and $r_Ff^{1} = f^{0}r_E$. We denote by $\ZD\zi\ZG\zr$ the category of all directed graphs and arbitrary graph morphisms.

Every partially ordered set $I$ gives rise to its \emph{associated graph} $(I^{0},I^{1},s_{I},r_{I})$, where $I^{0} = I$, $I^{1} = \{(i,j)\in I\times I\mid i<j\}$ and $s_{I}((i,j)) = i$, $r_{I}((i,j)) = j$ for all $(i,j)\in I^{1}$. If $I,J$ are partially ordered sets and $f:I\to J$ is an isotone map, then we may consider the two maps $f^{0}:I^{0}\to J^{0}$, $f^{1}:I^{1}\to J^{1}$, where $f^{0}(i) = f(i)$ for all $i\in I$ and, if $i<j$ in $I$, then $f^{1}((i,j)) = (f^{0}(i),f^{0}(j))$; we can check immediately that the pair $(f^{0}, f^{1})$ is a graph morphism from $(I^{0},I^{1},s_{I},r_{I})$ to $(J^{0},J^{1},s_{J},r_{J})$. Conversely, if $(f^{0}, f^{1})\colon (I^{0},I^{1})\to (J^{0},J^{1})$ is any graph morphism, then the map $f\defug f^{0}\colon I\to J$ is isotone. We say that $(f^{0}, f^{1})$ is the \emph{graph morphism associated} to $f$.

In seeking an appropriate subcategory of $\ZD\zi\ZG\zr$ to which one can extend, as functors, the maps which assign to every directed graph $E$ its path $\BK$-algebra $\BK E$ and its Leavitt path algebra $L_{\BK}(E)$, in \cite{Good:5} Goodearl defines a graph morphism $f = (f^{0}, f^{1})\colon E\to F$ to be a $CK$-\emph{morphism} (short for \emph{Cuntz-Krieger morphism}) if it satisfies the following conditions:
\begin{enumerate}
  \item $f^{0}$ and $f^{1}$ are both injective;
  \item For each $v\in E^{0}$ which is neither a sink nor an infinite emitter, $f^{1}$ induces a bijection from $s_{E}^{-1}(v)$ to $s_{F}^{-1}(f^{0}(v))$
\end{enumerate}
(for row-finite graphs, meaning that each vertex emits at most finitely many arrows, $CK$-morphisms are precisely the \emph{complete morphisms} of \cite{AraMorenoPardo:01}, p. 161); then he shows that the assignments $E\mapsto \BK E$ and $E\mapsto L_{\BK}(E)$ admit natural functorial extensions to the subcategory $\ZC\ZK\ZG\zr$ of $\ZD\zi\ZG\zr$ of all graphs together with $CK$-morphisms. Let us say that $f$ is a \emph{strict $CK$-morphism} if it satisfies the above condition (1) together with the following stronger form of (2):
\begin{enumerate}\setcounter{enumi}{2}
  \item For \emph{every} $v\in E^{0}$, $f^{1}$ induces a bijection from $s_{E}^{-1}(v)$ to $s_{F}^{-1}(f^{0}(v))$.
\end{enumerate}

\begin{pro}\label{pro CK}
Let $I,J$ be partially ordered sets and let $\map{f}IJ$ be an isotone map. Then $f$ is an order imbedding and $f(I)$ is an upper subset of $J$ if and only if the associated graph morphism  $(f^{0}, f^{1})\colon (I^{0},I^{1})\to (J^{0},J^{1})$  is a strict $CK$-morphism.
\end{pro}
\begin{proof}
Suppose that $f$ is an order imbedding and $f(I)$ is an upper subset of $J$. Then both $f^{0}$, $f^{1}$ are clearly injective. Given $i\in I$, we have that $f^{1}$ induces an injective map from $s_{I}^{-1}(i)$ to $s_{J}^{-1}(f^{0}(i))$, which is also surjective. In fact, given $(k,l)\in s_{J}^{-1}(f^{0}(i))$, necessarily $k = f^{0}(i)<l$. As $f^{0}(I) = f(I)$ is an upper subset of $J$, there is some $j\in I$ such that $l = f(j)$ and, since $f$ is an order imbedding, we have that $i<j$, thus $(i,j)\in I^{1}$ and hence $(k,l) = f^{1}(i,j)$.

Conversely, suppose that $(f^{0}, f^{1})$ is a strict $CK$-morphism. Then $f = f^{0}$ is injective. Let $i,j\in I$ and assume that $f(i)<f(j)$, that is $(f(i), f(j))\in I^{1}$. As $(f(i), f(j))\in s_{J}^{-1}(f^{0}(i))$, by the assumption there is a unique $(i',j')\in s_{I}^{-1}(i)$ such that $(f(i), f(j)) = f^{1}(i',j')$. Necessarily $i' = i$, $j' = j$, so that $i<j$. Next, let $i\in I$ and $k\in J$ be such that $f(i)<k$. Then $(f(i), k)\in s_{J}^{-1}(f^{0}(i))$ and, again from the assumption, we infer that $(f(i), k) = f^{1}((i,j))$ for a unique $j\in I$. Thus $k = f(j)\in f(I)$ and we are done.
\end{proof}

We now proceed to find the promised two functorial extensions of the assignment $I\mapsto \mathfrak{B}(I)$.

Let $\map fIJ$ be a morphism of $\Pos_{\al}$ and let $i\in I$. Since $f$ restricts to an order isomorphism from $\{i\le\}$ to $\{f(i)\le\}$, the assignment $(u,E)\mapsto (u\cdot f|_{\{i\le\}},f^{-1}(E))$ defines a bijection from $X(f(i)\le)$ to $X(i\le)$ and therefore we get the $\BK$-algebra isomorphism
\[
\lmap{\tilde{f}_{i}}{\CFM_{X(i\le)}(\BK)}{\CFM_{X(f(i)\le)}(\BK)}
\]
defined by
\[
\left(\tilde{f}_{i}(\za)\right)((u,E),(v,F))= \za((u\cdot f|_{\{i\le\}},f^{-1}(E)),(v\cdot f|_{\{i\le\}},f^{-1}(F)))
\]
for all $\za\in\CFM_{X(i\le)}(\BK)$, $u,v\in\al^{(f(i)\le)}$ and $E,F\in\CM(f(i)\le)$. We denote by $\overline{f}_{i}$ the isomorphism from $\FR_{X(i\le)}$ to $\FR_{X(f(i)\le)}$ induced by $\tilde{f}_{i}$ \emph{when $i$ is not maximal}; if $i$ is maximal, we use the same symbol $\overline{f}_{i}$ to denote the obvious isomorphism $\mathbf{1}_{\CFM_{X(i\le)}(\BK)}k \mapsto \mathbf{1}_{\CFM_{X(f(i)\le)}(\BK)}k$. Recall that, by the definition of $H(I,i)$, the assignment $\za\mapsto\psi_{I,i}(\za)$ (see Proposition \ref{subrings} and the beginning of section \ref{sect ringDI}) defines a $\BK$-algebra isomorphism
\[
\lmap{\a_{I,i}}{\FR_{X(i\le)}(\BK)}{H(I,i)}
\]
when $i$ is not maximal, and an isomorphism
\[
\lmap{\a_{I,i}}{\mathbf{1}_{\CFM_{X(i\le)}(\BK)}\BK}{H(I,i)} = \ze_{X(I)_i}\BK
\]
when $i$ is maximal. Thus we obtain the $\BK$-algebra isomorphism
\[
\lmap{\widehat{f}_{i}\defug \a_{J,f(i)}\overline{f}_{i}\a_{I,i}^{-1}}
{H(I,i)}{H(J,f(i))}
\]
for each $i\in I$ and hence the $\BK$-linear isomorphism
\[
\lmap{\widehat{f}\defug\bigoplus_{i\in I}\widehat{f}_{i}}{\bigoplus_{i\in I}H(I,i) = \mathfrak{B}(I)}{\bigoplus_{i\in I}H(J,f(i)) = H(J,f(I))}.
\]

\begin{lem}\label{pro isoalg}
With the above setting and notations, $\widehat{f}$ is a $\BK$-algebra isomorphism.
\end{lem}
\begin{proof}
Given $i\in I$, let us start by computing explicitly $\widehat{f}_{i}(\za)$ for each $\za\in H(I,i)$. First note that if $i$ is \emph{is maximal} in $I$, then $f(i)$ is maximal in $J$ and, since $H(I,i)=\ze_{X(I)_{i}}\BK$ and $H(J,f(i))=\ze_{X(J)_{f(i)}}\BK$, we have that
\begin{equation}\label{eq blabbb}
\widehat{{f}_{i}}(\ze_{X(I)_{i}}k) = \ze_{X(J)_{f(i)}}k \quad \text{for all $k\in\BK$}.
\end{equation}
Assume that $i$ \emph{is not maximal}. We observe that, inasmuch as $H(I,i)\sbs S(I,i)$, if $\za\in H(I,i)$ then for every $s,t\in\al^{(\nle i)}$, $u,v\in\al^{(i\le)}$, $C,D\in\CM(\le i)$ and $E,F\in\CM(i\le)$ we have that
\[
    \za((s\cup u,C\cup E),(s\cup v,C\cup F))
    = \za((t\cup u,D\cup E),(t\cup v,D\cup F)).
\]
Thus we can define the map
\[
\lmap{\b}{H(I,i)}{\FR_{X(i\le)}(\BK)}
\]
as follows: given $\za\in H(I,i)$, for every $u,v\in\al^{(i\le)}$ and $E,F\in\CM(i\le)$
\[
\b(\za)((u,E),(v,F)) \defug
\za((t\cup u,C\cup E),(t\cup v,C\cup F)),
\]
where $t\in\al^{(\nle i)}$ and $C\in\CM(\le i)$ are chosen arbitrarily. We claim that
\[\a_{I,i}^{-1} = \b.
 \]
Indeed, given $\zb\in\FR_{X(i\le)}(\BK)$, let $u,v\in\al^{(i\le)}$ and $E,F\in\CM(i\le)$. By taking arbitrarily $t\in\al^{(\nle i)}$ and $C\in\CM(\le i)$ we have that
\begin{align*}
\left[\b\left(\a_{I,i}(\zb)\right)\right]((u,E),(v,F))
&= \a_{I,i}(\zb)((t\cup u,C\cup E),(t\cup v,C\cup F)) \\
&= \zb((u,E),(v,F)).
\end{align*}
Since $\a_{I,i}$ is an isomorphism, this proves our claim.

Now we observe that, given $r,s\in\al^{(f(i)\nle)}$, $u,v\in\al^{(f(i)\le)}$ and $A,B\in\CM(J)$, we have
\[
\widehat{f}_{i}(\za)((r\cup u,A), (s\cup v,B)) \ne 0
\]
only if
\[
f(i)\in A\cap B,\quad A\cap\{\le f(i)\} = B\cap\{\le f(i)\}\quad \text{and} \quad r=s.
\]
Assume that this is the case, set
\[
E = f^{-1}(A)\cap\{i\le\}, \qquad
F =  f^{-1}(B)\cap\{i\le\}
\]
and choose arbitrarily $t\in\al^{(i\nle)}$ and $C\in\CM(\le i)$. Since $f$ is a morphism of $\Pos_{\al}$, then $f(\{i\le\}) = \{f(i)\le\}$ and $f(\{\nle i\})\sbs\{\nle f(i)\}$, thus we can compute as follows:
\begin{equation}\label{eq eta}
\begin{split}
\widehat{f}_{i}(\za)&((r\cup u,A), (r\cup v,B))=\left[\overline{f_{i}}\left(
\a_{I,i}^{-1}(\za)\right)\right]((u,A\cap\{f(i)\le\}),(v, B\cap\{f(i)\le\}))\\
&=\left[\a_{I,i}^{-1}(\za)\right]((u\cdot f|_{\{i\le\}},E),(v\cdot f|_{\{i\le\}}, F))\\
&=\za((t\cup (u\cdot f|_{\{i\le\}}),C\cup E),(t\cup (v\cdot f|_{\{i\le\}}), C\cup F)) \\
&=\za(((r\cdot f|_{\{i\nle\}})\cup (u\cdot f|_{\{i\le\}}), C\cup E),((r\cdot f|_{\{i\nle\}})\cup (v\cdot f|_{\{i\le\}}), C\cup F)) \\
&=\za(((r\cup u)f, C\cup E),((r\cup v)f, C\cup F)),
\end{split}
\end{equation}
where the last equality comes from the fact that $\za\in Q_{i}$.

Now we can prove that
\begin{equation}\label{eq homoalg}
    \widehat{f}(\za\zb) = \widehat{f}(\za)\cdot\widehat{f}(\zb)
\end{equation}
for every $\za,\zb\in\mathfrak{B}(I)$. Thanks to the linearity of $\widehat{f}$ it is sufficient to assume that $\za\in H(I,i)$ and $\zb\in H(I,j)$ for some $i,j\in I$ with $i\ne j$. Inasmuch as $\widehat{f}(H(I,i))\sbs H(J,f(i))$ for all $i\in I$, if $i$ and $j$ are not comparable, then both members of \eqref{eq homoalg} are zero. Assume that $i<j$. Then $\za\zb\in H(I,i)$ and both members of \eqref{eq homoalg} are in $H(J,f(i))$; consequently, in order to check that \eqref{eq homoalg} holds it is sufficient to prove that if $r\in\al^{(f(i)\nle)}$, $u,v\in\al^{(f(i)\le)}$ and $A,B\in\CM(J)$ are such that $f(i)\in A\cap B$ and $A\cap\{\le f(i)\} = B\cap\{\le f(i)\}$, the equality
\begin{equation}\label{eq homoalgg}
    \left[\widehat{f}(\za\zb)\right]((r\cup u,A), (r\cup v,B))
    = \left[\widehat{f}(\za)\cdot\widehat{f}(\zb)\right]((r\cup u,A), (r\cup v,B))
\end{equation}
holds. Assume that $j$ \emph{is not maximal} in $I$. By setting
\begin{gather*}
C = A\cap\{\le f(i)\} = B\cap\{\le f(i)\}, \quad D = f^{-1}(C)\in\CM(\le i),\\
E = f^{-1}(A)\cap\{i\le\}, \quad   \text{and}\quad F = f^{-1}(B)\cap\{i\le\}
\end{gather*}

we obtain from \eqref{eq eta}:
\[
  \begin{aligned}
    &\left[\widehat{f}(\za\zb)\right]((r\cup u,A), (r\cup v,B)) \\
    &\phantom{XXXXXX} = \left[\widehat{f}_{i}(\za\zb)\right]((r\cup u,A), (r\cup v,B)) \\
    &\phantom{XXXXXX} =(\za\zb)(((r\cup u)f, D\cup E),((r\cup v)f, D\cup F))\\
    &\phantom{XXXXXX}= \sum_{x\in X(I)}\za(((r\cup u)f, D\cup E),x)\cdot
\zb(x,((r\cup v)f, D\cup F))).
\end{aligned}
\]

If we write
\[
v' = v|_{\{f(i)\le\}\setminus\{f(j)\le\}},
\]
in order that a single summand of the last sum be nonzero it is necessary that
\[
x = ((r\cup v'\cup w)f, D\cup G),
\]
where $w\in\al^{(f(j)\le)}$ and $G$ is a maximal chain of $\{i\le\}$ which belongs to the set
\[
\CG\ = \bigl\{\phantom{\bigotimes}\!\!\! \!\!\!\!\!\!G\in\CM(i\le)\mid j\in G,\, f(G)\cap[f(i),f(j)] = B\cap[f(i),f(j)] \bigr\}.
\]
Thus we may write
\begin{equation}\label{eq homoalggg}
  \begin{aligned}
    &\left[\widehat{f}(\za\zb)\right]((r\cup u,A), (r\cup v,B)) \\
&\phantom{xxxxx}=\sum_{\substack{w\in\al^{(f(j)\le)}\\ G\in\CG}}
\left\{\phantom{\bigotimes}\!\!\!\!\!\!\!\za(((r\cup u)f,D\cup E), ((r\cup v'\cup w)f, D\cup G))\right. \\
&\phantom{xxxxxxxxxxxxxxx}\cdot
\zb(((r\cup v'\cup w)f, D\cup G),((r\cup v)f,D\cup F))\Big{\}}.
\end{aligned}
\end{equation}
On the other hand we have
\[
\begin{aligned}
    &\left[\widehat{f}(\za)\cdot\widehat{f}(\zb)\right]
    ((r\cup u,A), (r\cup v,B))\\
&\phantom{xxxxxxx}= \sum_{y\in X(J)}\left\{\phantom{\bigotimes}\!\!\!\!\!\!\!
[\widehat{f}(\za)]((r\cup u,A), y)
\cdot[\widehat{f}(\zb)](y, (r\cup v,B))\right\}.
\end{aligned}
\]
Here again, a single summand is not zero only if
\[
y = (r\cup v'\cup w, C\cup H),
\]
where $w\in\al^{(f(j)\le)}$ and $H$ is a maximal chain of $\{f(i)\le\}$ which belongs to the set
\[
\CH\ = \bigl\{\phantom{\bigotimes}\!\!\!\!\!\!\!\!\!H\in\CM(\{f(i)\le\})\mid f(j)\in H,\, H\cap[f(i),f(j)] = B\cap[f(i),f(j)] \bigr\}.
\]
Thus, by using \eqref{eq eta} we obtain:
\[
  \begin{split}
    &\left[\widehat{f}(\za)\cdot\widehat{f}(\zb)\right]
    ((r\cup u,A), (r\cup v,B))\\
&\phantom{xxxxxxxx}=\sum_{\substack{w\in\al^{(f(j)\le)}\\ H\in\CH}}
\left\{\phantom{\bigotimes}\!\!\!\!\!\!\!\!\!
[\widehat{f}(\za)]((r\cup u,A), (r\cup v'\cup w,C\cup H))\right. \\
&\left.\phantom{xxxxxxxxxxxx}
\cdot[\widehat{f}(\zb)] ((r\cup v'\cup w,C\cup H), (r\cup v,B))\right\}\\
&\phantom{xxxxxxxx}=\sum_{\substack{w\in\al^{(f(j)\le)}\\ H\in\CH}}
\left\{\phantom{\bigotimes}\!\!\!\!\!\!\!\za(((r\cup u)f,D\cup E),((r\cup v'\cup w)f, D\cup f^{-1}(H))\right. \\
&\phantom{xxxxxxxxxxxx}
\cdot \zb(((r\cup v'\cup w)f, D\cup f^{-1}(H)),((r\cup v)f, D\cup F))\Big{\}}.
\end{split}
\]
Finally, since $\CG = \{f^{-1}(H)\mid H\in\CH\}$, the latter sum coincides with the sum at the second member of \eqref{eq homoalggg}, therefore the equality \eqref{eq homoalg} holds.

Finally, assume that $j$ \emph{is maximal} in $I$, so that $\zb = \ze_{X(I)_{j}}\cdot k$ for some $k\in\BK$ and \eqref{eq homoalg} becomes
\begin{equation}\label{eq abcd}
\widehat{f}(\za\ze_{X(I)_{j}})\cdot k = (\widehat{f}(\za)\cdot\widehat{f}(\ze_{X(I)_{j}})) \cdot k,
\end{equation}
by taking \eqref{eq blabbb} into account; of course we may assume that $k = 1$. With the same meaning of $D,E,F$ as above, by using \eqref{eq eta} we have
\[
  \begin{aligned}
    &\left[\widehat{f}(\za\ze_{X(I)_{j}})\right]((r\cup u,A), (r\cup v,B)) \\
    &\phantom{XXXXXX} = \left[\widehat{f}_{i}(\za\ze_{X(I)_{j}})\right]((r\cup u,A), (r\cup v,B)) \\
    &\phantom{XXXXXX} =(\za\ze_{X(I)_{j}})(((r\cup u)f, D\cup E),((r\cup v)f, D\cup F))\\
    &\phantom{XXXXXX}= \begin{cases} \za(((r\cup u)f, C\cup E),((r\cup v)f, C\cup F)), \text{ if $f(j)\in B$;} \\
    $0$, \text{ otherwise.}
    \end{cases}
\end{aligned}
\]
Similarly we have
\[
  \begin{aligned}
    &\left[\widehat{f}(\za)\cdot\widehat{f}(\ze_{X(I)_{j}})\right]((r\cup u,A), (r\cup v,B)) \\
    &\phantom{XXXXXX} = \left[\widehat{f_{i}}(\za)\cdot\widehat{f_{i}}(\ze_{X(I)_{j}})\right]((r\cup u,A), (r\cup v,B)) \\
    &\phantom{XXXXXX} = \left[\widehat{f_{i}}(\za)\cdot\ze_{X(J)_{f(j)}}\right]((r\cup u,A), (r\cup v,B)) \\
    &\phantom{XXXXXX}= \begin{cases} \widehat{f_{i}}(\za)((r\cup u,A), (r\cup v,B)), \text{ if $f(j)\in B$;} \\
    $0$, \text{ otherwise.}
    \end{cases} \\
    &\phantom{XXXXXX}= \begin{cases} \za(((r\cup u)f, C\cup E),((r\cup v)f, C\cup F)), \text{ if $f(j)\in B$;} \\
    $0$, \text{ otherwise}
    \end{cases}
\end{aligned}
\]
and the equality \eqref{eq abcd} is proved.

The case in which $i>j$ is handled in the same manner and we conclude that $\widehat{f}$ is a $\BK$-algebra homomorphism.
\end{proof}

We are now in a position to define the $\BK$-algebra homomorphisms
\[
\lmap{\mathfrak{B}(f)}{\mathfrak{B}(I)}
{\mathfrak{B}(J)}\quad \text{and}\quad
\lmap{\mathfrak{B}^{*}(f)}{\mathfrak{B}^{*}(J)}
{\mathfrak{B}^{*}(I)}:
\]
we take the first as the composition of $\widehat{f}$ followed by the inclusion into $\mathfrak{B}(J)$ and the second as the composition of the projection of $\mathfrak{B}(J)$ onto $H(J,f(I))$, with kernel the ideal $H(J,J\setminus f(I))$ (see \ref{remark split} and note that $J\setminus f(I)$ is a lower subset of $J$), followed by $\widehat{f}^{-1}$.

\begin{pro}\label{pro covfunct}
With the above settings, $\mathfrak{B}(-)$ and $\mathfrak{B}^{*}(-)$ are functors from $\Pos_{\al}$ to the category $\Alg_{\BK}$, the first covariant and the second contravariant.
\end{pro}
\begin{proof}
If $i\in I$ is not maximal, since $\overline{\left(1_{I}\right)}_{i} = 1_{\FR_{X(i\le)}(\BK)}$, then $\widehat{\left(1_{I}\right)}_{i} = 1_{H(I,i)}$; if $i$ is maximal, then $\overline{\left(1_{I}\right)}_{i}$ is the identity map on $\mathbf{1}_{\CFM_{X(i\le)}(\BK)}\BK$ and so, again, $\widehat{\left(1_{I}\right)}_{i} = 1_{H(I,i)}$. As a consequence $\mathfrak{B}(1_{I}) = 1_{\mathfrak{B}(I)} = \mathfrak{B}^{*}(1_{I})$. Let $\map fIJ$ and $\map gJK$ be morphisms in $\Pos_{\al}$. A straightforward computation shows that
\[
(\overline{gf})_{i} = \overline{g}_{f(i)}\cdot\overline{f_{i}},
\]
from which we infer easily that
\[
(\widehat{gf})_{i} = \widehat{g}_{f(i)}\widehat{f}_{i} \quad \text{and}
\quad(\widehat{gf})^{-1}_{i} = \widehat{f}^{-1}_{i}\widehat{g}\,^{-1}_{f(i)}.
\]
Consequently
\[
\mathfrak{B}(gf) =  \mathfrak{B}(g)\mathfrak{B}(f) \quad \text{and}\quad
\mathfrak{B}^{*}(gf) =  \mathfrak{B}^{*}(f)\mathfrak{B}^{*}(g).
\]
\end{proof}

\begin{re}\label{remark nonunit}
If $I\in\Pos_{\al}$, we know that $\mathfrak{B}(I)$ is a unital $\BK$-algebra, specifically
\[
1_{\mathfrak{B}(I)} = \sum\{\ze_{X(I)_{m}}| m\in\ZM(I)\}
\]
(see Proposition \ref{multunit}). If $f\colon I\to J$ is a morphism of $\Pos_{\al}$, while the $\BK$-algebra homomorphism $\mathfrak{B}^{*}(f)$ is always unital, $\mathfrak{B}_{\al,K}(f)$ is unital if and only if $f(\ZM(I)) = \ZM(J)$.
\end{re}

Given a \emph{unital} ring $R$, it is standard to define $\KO(R)$ as the Grothendieck group of the commutative monoid of isomorphism classes of all finitely generated projective right $R$-modules; this makes $\KO(-)$ a covariant functor from the category of all unital rings and all unital ring homomorphisms, to the category of partially ordered abelian groups. Fortunately a natural extension of $\KO(-)$ to the category $\ZR\zn\zg$ (this is a notation which seems due to Faith) of not necessarily unital rings and not necessarily unital ring homomorphisms is available (see Ch. 5, \S 7 in \cite{Rosen:1}). Specifically, given $R\in\ZR\zn\zg$, one considers first the ``unitization'' ring $R_{+}$ of $R$, defined as the abelian group $R\times\BZ$ together with the multiplication $(x,m)(y,n)\defug (xy+my+nx,mn)$, so that the assignment $x\mapsto (x,0)$ identifies $R$ as an ideal of $R_{+}$; thus, by considering the canonical epimorphism $\pi_{R}\colon R_{+}\to R_{+}/R \is \BZ$, one defines $\KO(R)$ as the kernel of $\KO(\pi_{R})$. If $\a\colon R\to S$ is a morphism in $\ZR\zn\zg$ and $\a_{+}\colon R_{+}\to S_{+}$ is the unital homomorphism defined by $(x,m)\mapsto (\a(x),m)$, then $\KO(\a)$ is taken as the unique homomorphism which makes commutative the diagram
\[
\begin{CD}
\KO(R) @>\sbs>> \KO(R_{+})\\
@V{\KO(\a)}VV @VV{\KO(\a_{+})}V\\
\KO(S) @>\sbs>> \KO(S_{+}).
\end{CD}
\]
If $R$ has a multiplicative identity $e$, then $(e,0)$ is a central idempotent of $R_{+}$ and so $R_{+}\is R\times\BZ$; consequently the above definition of $\KO(R)$ matches with the standard one.

Suppose that $R$, $S$ are unital regular rings and $\a\colon R\to S$ is a morphism in $\ZR\zn\zg$. By the above, for every idempotent $e\in R$ we have:
\begin{gather*}
\KO(\a)([eR]) = \KO(\a_{+})([(e,0)R_{+}]) = [(e,0)R_{+}\otimes_{R_{+}}S_{+}]
= [(\a(e),0)S_{+}] \\ = [\a(e)S\times\{0\}] = [\a(e)S].
\end{gather*}
Since $\KO(R)$ is generated by the set $\{[eR]\mid e = e^{2}\in R\}$, we see that $\KO(\a)$ is the unique homomorphism such that $\KO(\a)([eR]) = [\a(e)S]$ for all idempotents $e\in R$.

An alternative definition of $\KO(R)$ uses the equivalence classes of idempotents in the ring $\ZM_{\infty}(R)$ of all $\al_{0}\times\al_{0}$ matrices having only finitely many nonzero entries (see for example Section 5 of \cite{Blackadar:001}); in this way an unified treatment of the unital and nonunital cases is obtained.

Let us come back to the functors $\mathfrak{B}(-)$ and $\mathfrak{B}^{*}(-)$. In view of all above, given a morphism $f\colon I\to J$ of $\Pos_{\al}$, it makes sense to speak about the homomorphism
$\map{\KO(\mathfrak{B}(f))}{\KO(\mathfrak{B}(I))} {\KO(\mathfrak{B}(J))}$ and hence about the composite functor $\KO(\mathfrak{B}(-))$. If $I\in\Pos_{\al}$, by Theorem \ref{theo K0} there is a unique isomorphism $\map{\rho(I)}{G(I)}{\KO(\mathfrak{B}(I))}$ which sends each $i\in I$ to $[\zu\mathfrak{B}(I)]$, where $\zu$ is an arbitrary primitive idempotent of $H(I,i)$. We have then two diagrams
\[
\begin{CD}
G(I) @>\rho(I)>> \KO(\mathfrak{B}(I))\\
@V{G(f)}VV @VV{\KO(\mathfrak{B}(f))}V\\
G(J) @>\rho(J)>> \KO(\mathfrak{B}(J)),
\end{CD}
\qquad\qquad\qquad
\begin{CD}
G^{*}(I) @>\rho(I)>> \KO(\mathfrak{B}^{*}(I))\\
@A{G^{*}(f)}AA @AA{\KO(\mathfrak{B}^{*}(f))}A\\
G^{*}(J) @>\rho(J)>> \KO(\mathfrak{B}^{*}(J)).
\end{CD}
\]
Let $i\in I$ and let $\zu$ be any primitive idempotent of $H(I,i)$. Then $(\mathfrak{B}(f))(\zu) = \widehat{f}_{i}(\zu)$ is a primitive idempotent of $H(J,f(i))$ and so we obtain:
\begin{gather*}
\KO(\mathfrak{B}(f))(\rho(I)(i))
= \KO(\mathfrak{B}(f))([\zu\mathfrak{B}(I)])\\
= [(\mathfrak{B}(f)(\zu))\mathfrak{B}(J)]
= [(\widehat{f}_{i}(\zu))\mathfrak{B}(J)]
= \rho(J)(f(i))
= \rho(J)(G(f)(i)).
\end{gather*}
Since $G(I)$ is free with $I$ as a basis, this shows that the first diagram commutes.

Next, let $j\in J$ and let $\zv$ be a primitive idempotent of $H(J,j)$. If $j\nin f(I)$, then $j\in\Ker(G^{*}(f))$ (see \eqref{eq imbedding}) and $H(J,j)\sbs\Ker(\mathfrak{B}^{*}(f))$, hence both $\rho(I)(G^{*}(f)(j))$ and $\KO(\mathfrak{B}^{*}(f))(\rho(J)(j))$ are zero. If $j\in J$ and $i$ is the unique element of $I$ such that $j = f(i)$, then:
\begin{gather*}
\KO(\mathfrak{B}^{*}(f))(\rho(J)(j))
= \KO(\mathfrak{B}^{*}(f))([\zv\mathfrak{B}^{*}(J)])\\
= [(\mathfrak{B}^{*}(f)(\zv))\mathfrak{B}^{*}(I)]
= [(\widehat{f}^{-1}(\zv))\mathfrak{B}^{*}(I)]
= [(\widehat{f}^{-1}_{i}(\zv))\mathfrak{B}^{*}(I)]\\
= \rho(I)(i)
= \rho(I)(G^{*}(f)(j)).
\end{gather*}
Here again, since $G(J)$ is free with $J$ as a basis, this shows that the second diagram commutes.
Finally we may conclude as follows:

\begin{pro}\label{pro natequiv}
With the above notations and settings, the assignment $I\mapsto\rho(I)$ defines two natural equivalences
\[
\rho(-)\colon G(-)\approx \KO(-)\circ\mathfrak{B}(-)\quad\text{and} \quad \rho^{*}(-)\colon G^{*}(-)\approx  \KO(-)\circ \mathfrak{B}^{*}(-).
\]
\end{pro}

\noindent\textbf{\emph{Acknowledgement}}. The author wish to express his deep gratitude to the refe\-ree. Her/his extremely accurate work was an invaluable help in preparing the final version of this paper, by fixing a number of issues present in the preliminary manuscript.

\end{document}